\documentclass[a4paper,12pt]{smfart}

\usepackage{ucs}
\usepackage[latin1]{inputenc}

\usepackage[francais,english]{babel} \AutoSpaceBeforeFDP
\usepackage[T1]{fontenc}
\usepackage{amstext,amssymb,amsthm}
\usepackage{smfthm}
\usepackage{graphicx}
\usepackage{color}
\usepackage{epsfig}
\usepackage[all]{xy}
\usepackage{enumerate}
\usepackage{version}
\usepackage{marvosym}

\usepackage{geometry}
\definecolor{darkred}{rgb}{0.9,0.,.2}
\definecolor{darkblue}{rgb}{0.,0.,.6}
\definecolor{darkgreen}{rgb}{0.,.6,0.1}
\usepackage[colorlinks=true,urlcolor=darkblue,linkcolor=darkred,citecolor=darkgreen]{hyperref}
\usepackage{subfig}

\newcommand{\N}{\mathbb{N}}
\newcommand{\Z}{\mathbb{Z}}

\newcommand{\R}{\mathbb{R}}
\newcommand{\C}{\mathcal{C}} 

\renewcommand{\ss}{\mathrm{SL_{n+1}(\mathbb{R})}}

\newcommand{\SO}{\mathrm{SO_{n,1}(\mathbb{R})}}

\newcommand*{\so}[1]{\mathrm{SO}_{#1,1}(\mathbb{R})}
\newcommand*{\s}[1]{\mathrm{SL}_{#1}(\mathbb{R})}

\newcommand{\LG}{\Lambda_{\Gamma}}
\newcommand{\G}{\Gamma}
\newcommand{\g}{\gamma}
\renewcommand{\L}{\Lambda}

\newcommand{\ph}{\varphi}
\renewcommand{\l}{\lambda}
\newcommand{\h}{\mathtt{h}}
\newcommand{\dgg}{\delta_{\Gamma}}

\renewcommand{\C}{\mathcal{C}}
\newcommand{\U}{\mathcal{U}}

\newcommand{\M}{\mathcal{M}}

\renewcommand{\H}{\mathcal{H}}
\newcommand{\Hh}{\mathcal{H}}

\newcommand{\V}{\mathcal{V}}

\newcommand{\E}{\mathcal{E}}

\renewcommand{\AA}{\mathcal{A}}

\renewcommand{\O}{\Omega}
\newcommand{\Og}{\mathcal{O}_{\Gamma}}
\newcommand{\dO}{\partial \Omega}
\renewcommand{\d}{d_{\Omega}}

\renewcommand{\P}{\mathcal{P}}

\newcommand{\PP}{\mathbb{P}}

\newcommand{\Quo}{\Omega/\!\raisebox{-.90ex}{\ensuremath{\Gamma}}}
\newcommand*{\Quotient}[2]{\ensuremath{#1/\!\raisebox{-.90ex}{\ensuremath{#2}}}}

\newcommand{\Aut}{\textrm{Aut}}

\newcommand{\Vol}{\textrm{Vol}}

\newcommand{\Stab}{\textrm{Stab}}

\newcommand{\NW}{\mathtt{NW}}

\newcommand{\HH}{\mathbb{H}}

\theoremstyle{plain}
\newtheorem{qu}{Question}
\newtheorem{fait}{Fait}

\theoremstyle{definition}

\theoremstyle{remark}

\title[Le flot géodésique en géométrie de Hilbert]{Le flot géodésique des quotients g\'eom\'etriquement finis des géométries de Hilbert}

\author{
\href{mailto:mickael.crampon@usach.cl}{Mickaël Crampon}
}
\address{Universidad de Santiago de Chile, Departamento de Matem\'atica y Ciencia de la Computaci\'on, Av. Las Sophoras 173 - Estaci\'on Central, Santiago de Chile - Chile}

\author{
\href{mailto:ludovic.marquis@univ-rennes1.fr}{Ludovic Marquis}
}
\address{IRMAR, 263 Av. du G\'en\'eral Leclerc, CS 74205 - 35042 Rennes Cedex - France \newline{} \newline{}}
\date{} 

\email{mickael.crampon@usach.cl \\ ludovic.marquis@univ-rennes1.fr}

\urladdr{
http://mikl.crampon.free.fr/
\href{http://mikl.crampon.free.fr/}{\Mundus}
\newline
http://perso.univ-rennes1.fr/ludovic.marquis
\href{http://perso.univ-rennes1.fr/ludovic.marquis}{\Mundus}
}

\begin{document}

\begin{abstract}
On \'etudie le flot g\'eod\'esique des quotients g\'eom\'etriquement finis $\Quo$ de géométries de Hilbert, en particulier ses propri\'et\'es de r\'ecurrence.\\
On prouve, sous une hypoth\`ese g\'eom\'etrique sur les cusps, que le flot g\'eod\'esique est uniform\'ement hyperbolique. Sans cette hypoth\`ese, on construit un exemple o\`u celui-ci a un exposant de Lyapunov nul.\\
On fait le lien entre la dynamique du flot g\'eod\'esique et certaines propri\'et\'es du convexe $\O$ et du groupe $\G$. On en d\'eduit des r\'esultats de rigidit\'e, qui \'etendent ceux de Benoist et Guichard pour les quotients compacts.\\
Enfin, on s'int\'eresse au lien entre entropie volumique et exposant critique; on montre entre autres qu'ils co\"incident lorsque le quotient est de volume fini.
\end{abstract}

\begin{altabstract}
We study the geodesic flow of geometrically finite quotients $\Quo$ of Hilbert geometries, in particular its recurrence properties.\\
We prove that, under a geometrical assumption on the cusps, the geodesic flow is uniformly hyperbolic. Without this assumption, we provide an example of a quotient whose geodesic flow has a zero Lyapunov exponent.\\
We make the link between the dynamics of the geodesic flow and some properties of the convex set $\O$ and the group $\G$. As a consequence, we get various rigidity results which extend previous results of Benoist and Guichard for compact quotients.\\
Finally, we study the link between volume entropy and critical exponent; for example, we show that they coincide provided the quotient has finite volume.
\end{altabstract}

\maketitle
\tableofcontents

\frontmatter

\section{Introduction}

\emph{Cet article dynamique fait logiquement suite \`a l'article g\'eom\'etrique \cite{Crampon:2012fk}, dans lequel nous \'etudions la notion de finitude g\'eom\'etrique en g\'eom\'etrie de Hilbert. Avec \cite{Crampon:2011fk}, ils forment un seul et m\^eme travail que nous avons d\'ecoup\'e en trois pour des raisons \'evidentes de longueur. Concernant la g\'eom\'etrie des vari\'et\'es g\'eom\'etriquement finies, nous ne rappellerons dans ce texte que les r\'esultats dont nous ferons usage et renvoyons le lecteur \`a \cite{Crampon:2012fk} pour plus d'informations.}\\

Une g\'eom\'etrie de Hilbert est un espace m\'etrique $(\O,\d)$ o\`u $\O$ est un ouvert proprement convexe de l'espace projectif réel $\PP^n=\PP^n(\R)$ et $\d$ est la distance d\'efinie sur $\O$ par 
$$d_{\O}(x,y) = \displaystyle \frac{1}{2} \Big|\ln \big([p:q:x:y]\big)\Big|,\ x,y\in\O\ \textrm{distincts};$$
dans cette formule, les points $p$ et $q$ sont les points d'intersection de la droite $(xy)$ avec le bord $\dO$ de $\O$. Ces g\'eom\'etries ont été introduites par Hilbert comme exemples de g\'eom\'etries dans lesquelles les droites sont des g\'eod\'esiques. Leur d\'efinition imite celle de l'espace hyperbolique dans le mod\`ele projectif de Beltrami, qui correspond \`a la g\'eom\'etrie de Hilbert d\'efinie par un ellipso\"ide.

\begin{center}
\begin{figure}[h!]
  \centering
\includegraphics[width=6cm]{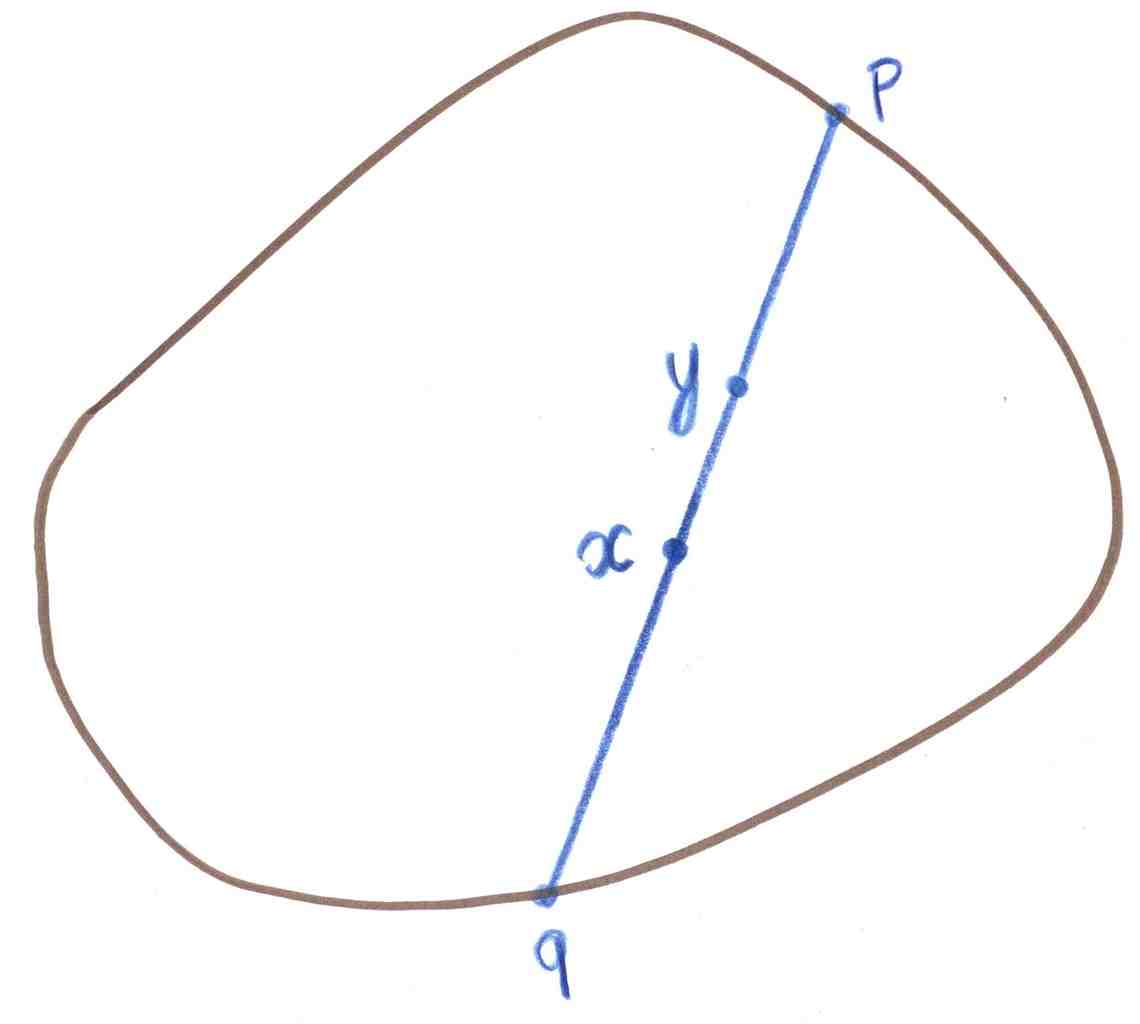}
\caption{La distance de Hilbert}

\end{figure}
\end{center}

Lorsque l'ouvert convexe $\O$ est strictement convexe, la g\'eom\'etrie de Hilbert est uniquement g\'eod\'esique: les droites sont les seules g\'eod\'esiques. On peut dans ce cas d\'efinir le flot g\'eod\'esique sans recourir \`a des \'equations g\'eod\'esiques, comme on le fait de fa\c con traditionnelle en g\'eom\'etrie riemannienne. Le flot g\'eod\'esique est ainsi le flot d\'efini sur le fibr\'e homog\`ene $H\O = \Quotient{T\O\smallsetminus\{0\}}{\R^+}$ de la fa\c con suivante: si $w=(x,[\xi])$ est un point de $H\O$, consistant en un point $x$ de $\O$ et une direction tangente $[\xi]$, on trouve son image $\ph^t(w)$ en suivant la droite g\'eod\'esique partant de $x$ dans la direction $[\xi]$.\\
Les g\'eom\'etries de Hilbert sont des espaces finsl\'eriens: la m\'etrique de Hilbert est engendr\'ee par un champ de normes $F:T\O\longrightarrow\R$ sur $\O$, donn\'e par la formule
$$F(x,\xi) = \frac{|\xi|}{2}\Bigg(\frac{1}{|xx^-|} + \frac{1}{| xx^+|} \Bigg),\ (x,\xi)\in T\O,$$
o\`u $x^+$ et $x^-$ sont les points d'intersection de la droite $\{x+\l \xi,\ \l\in\R\}$ avec $\partial \O$  (voir le paragraphe \ref{para_def_dist} pour plus de pr\'ecisions). Lorsque $\dO$ est de classe $\C^2$ \`a hessien d\'efini positif, alors on peut d\'efinir les g\'eod\'esiques au moyen d'une \'equation diff\'erentielle, et le flot g\'eod\'esique est le flot de cette \'equation.\\

La g\'eom\'etrie de Hilbert d\'efinie par un ouvert strictement convexe \`a bord $\C^1$ poss\`ede \emph{un certain comportement hyperbolique}. Par exemple, dans ce cadre-l\`a, on voit appara\^itre naturellement, au moyen des horosph\`eres, les vari\'et\'es stables et instables du flot g\'eod\'esique. Le flot g\'eod\'esique est dans ce cas de classe $\C^1$ et l'espace tangent \`a $H\O$ admet une d\'ecomposition en
$$TH\O = \R.X \oplus E^s \oplus E^u,$$
o\`u $X$ est le g\'en\'erateur du flot, $E^s$ est la distribution stable tangente au feuilletage stable, et $E^u$ est la distribution instable.\\
Les flots g\'eod\'esiques des vari\'et\'es riemanniennes compactes de courbure n\'egative sont les premiers exemples de flots d'Anosov, ou uniform\'ement hyperboliques. Cette propri\'et\'e d'hyperbolicit\'e ne d\'epend que des bornes sur la courbure et elle reste donc vraie pour une vari\'et\'e riemannienne non compacte \`a courbure n\'egative $K<-a^2<0$.\\
Pour une g\'eom\'etrie de Hilbert quelconque, on ne peut esp\'erer obtenir de propri\'et\'e d'hyperbolicit\'e. En effet, le comportement asymptotique autour d'une g\'eod\'esique d\'epend de la r\'egularit\'e du bord du convexe au point extr\'emal de la g\'eod\'esique (voir \cite{Crampon:2011fk2} pour une \'etude d\'etaill\'ee). Par contre, si la g\'eom\'etrie admet un quotient assez petit, on peut s'attendre \`a des propri\'et\'es de r\'ecurrence sur le quotient.\\
C'est le cas lorsqu'il existe un quotient compact: dans \cite{MR2094116}, Yves Benoist a prouv\'e que le flot g\'eod\'esique d'un quotient compact d'une g\'eom\'etrie de Hilbert (avec $\O$ strictement convexe) \'etait un flot d'Anosov. Notre premier th\'eor\`eme g\'en\'eralise cela au flot g\'eod\'esique de certaines vari\'et\'es g\'eom\'etriquement finies.\\

Les vari\'et\'es g\'eom\'etriquement finies sont en quelque sorte les vari\'et\'es non compactes les plus simples. Leur caract\'eristique essentielle pour nous est que leur c\oe ur convexe se d\'ecompose en une partie compacte et un nombre fini de cusps. C'est essentiel car le c\oe ur convexe est le support de l'ensemble non errant du flot g\'eod\'esique; c'est donc l\`a que se concentre la dynamique. On renvoie au fait \ref{decomposition} ou \`a l'article \cite{Crampon:2012fk} pour plus de d\'etails.\\
\`A chaque fois, on va essayer de comprendre s\'epar\'ement ce qu'il se passe sur la partie compacte puis sur les parties cuspidales. De fa\c con g\'en\'erale, on ne peut rien dire sans faire d'hypoth\`eses sur la g\'eom\'etrie des cusps:

\begin{prop}[Proposition \ref{nonanosov}]\label{contreex_intro}
Il existe une vari\'et\'e g\'eom\'etriquement finie $M=\Quo$ dont le flot g\'eod\'esique a un exposant de Lyapunov nul. En particulier, le flot g\'eod\'esique n'est pas uniform\'ement hyperbolique.
\end{prop}

Dans ce texte, nous \'etudierons donc principalement les vari\'et\'es g\'eom\'etriquement finies dont les cusps sont \og asymptotiquement hyperboliques\fg {}: dans un cusp, la m\'etrique de Hilbert est \'equivalente \`a une m\'etrique hyperbolique qui a les m\^emes g\'eod\'esiques (non param\'etr\'ees); voir la d\'efinition \ref{asymphypdefi}. Parmi les vari\'et\'es g\'eom\'etriquement finies \`a cusps asymptotiquement hyperboliques, on trouve en particulier les vari\'et\'es de volume fini, et plus g\'en\'eralement celles dont les sous-groupes paraboliques maximaux sont de rang maximal, c'est-\`a-dire qu'ils agissent cocompactement sur $\dO\smallsetminus\{p\}$, o\`u $p$ est le point fixe du groupe parabolique consid\'er\'e.\\
Il est fort possible que pour toute vari\'et\'e g\'eom\'etriquement finie $M=\Quo$, il existe un ouvert $\O'$, $\G$-invariant, strictement convexe et \`a bord $\C^1$, tel que le quotient $M'=\Quotient{\O'}{\G}$ soit g\'eom\'etriquement fini \`a cusps asymptotiquement hyperboliques. La raison principale qui nous pousse \`a penser qu'une telle construction est possible est que les sous-groupes paraboliques d'un tel groupe $\G$ sont conjugu\'es à des sous-groupes paraboliques de $\SO$.\\

Pour ces vari\'et\'es-l\`a, on peut prouver le

\begin{theo}[Th\'eor\`eme \ref{anosov}]\label{anosovintro}
Soient $\O$ un ouvert strictement convexe et \`a bord $\C^1$, et $M=\Quo$ une variété géométriquement finie \`a cusps asymptotiquement hyperboliques. Le flot géodésique de la m\'etrique de Hilbert est uniformément hyperbolique sur son ensemble non errant $\NW$: le fibré tangent à $HM$ admet en tout point de $\NW$ une d\'ecomposition $\ph^t$-invariante
$$THM = \R.X \oplus E^s \oplus E^u,$$
telle qu'il existe des constantes $\chi,C >0$ pour lesquelles
\begin{equation}\label{inegaliteanosov} 
\|d\ph^t Z^s\| \leqslant C e^{-\chi t},\ \|d\ph^{-t} Z^u\| \leqslant C e^{-\chi t},\ Z^s\in E^s,\ Z^u\in E^u,\ t\geqslant 0.
\end{equation}
\end{theo}

De fa\c con g\'en\'erale, on prouvera aussi les propri\'et\'es de r\'ecurrence suivantes:

\begin{prop}[Proposition \ref{melange}]\label{melange_intro}
Soient $\O$ un ouvert strictement convexe et \`a bord $\C^1$, et $M=\Quo$ une variété quotient. Le flot g\'eod\'esique de $M$ est topologiquement m\'elangeant sur son ensemble non errant.
\end{prop}

Notre deuxi\`eme th\'eor\`eme concerne la r\'egularit\'e du bord des ouverts convexes $\O$ qui admettent un quotient g\'eom\'etriquement fini $M=\Quo$ non compact. Ce r\'esultat est li\'e au fait que les propri\'et\'es hyperboliques des orbites du flot géodésique se lisent directement sur la r\'egularit\'e du bord au niveau de leur point extr\'emal.\\
Bien entendu, cela permet de d\'ecrire le bord uniquement au niveau de l'ensemble limite $\LG$ du groupe. Ce n'est pas \'etonnant puisque celui-ci constitue l'ensemble des points extr\'emaux des g\'eod\'esiques r\'ecurrentes. De plus, c'est la seule partie du bord qui est \emph{impos\'ee} par le groupe $\G$: on peut en effet modifier le bord (presque) \`a sa guise hors de l'ensemble limite; c'est d'ailleurs ainsi qu'on obtient l'exemple de la proposition \ref{contreex_intro}. Pour un quotient compact ou de volume fini, l'ensemble limite est le bord tout entier et donc le convexe $\O$ est enti\`erement d\'etermin\'e par le groupe $\G$.

\begin{theo}[Corollaire \ref{regu_geo_fini}]\label{regbord_intro}
Soient $\O$ un ouvert strictement convexe et \`a bord $\C^1$, et $M=\Quo$ une variété géométriquement finie \`a cusps asymptotiquement hyperboliques. Il existe $\varepsilon>0$ tel que le bord $\dO$ du convexe $\O$ soit de classe $\C^{1+\varepsilon}$ en tout point de $\LG$.
\end{theo}

Via la caract\'erisation des quotients de volume fini par leur ensemble limite, on obtient le corollaire suivant.

\begin{coro}[Corollaire \ref{regu_geo_fini}]\label{regbord_volfini_intro}
Soit $\O$ un ouvert strictement convexe et \`a bord $\C^1$. Si $\O$ admet un quotient de volume fini, alors son bord $\dO$ est de classe $\C^{1+\varepsilon}$ pour un certain $\varepsilon>0$.
\end{coro}

Lorsque $\O$ admet un quotient compact, Olivier Guichard a pu d\'eterminer exactement la r\'egularit\'e optimale du bord, c'est-\`a-dire le plus grand $\varepsilon$ tel que le bord $\dO$ soit $\C^{1+\varepsilon}$. Celle-ci est encore une fois d\'etermin\'ee par le groupe $\G$, via les valeurs propres de ses \'el\'ements hyperboliques. Cela n'est pas \'etonnant, \'etant donn\'e que les orbites p\'eriodiques sont denses, et que celles-ci sont en bijection avec les classes de conjugaison d'\'el\'ements hyperboliques de $\G$.\\
Si on se restreint \`a l'ensemble limite et l'ensemble non errant, cette observation reste valable pour un quotient quelconque. Nous pouvons ainsi prouver un r\'esultat similaire pour les ouverts convexes qui admettent un quotient g\'eom\'etriquement fini \`a cusps asymptotiquement hyperboliques. Pour l'\'enonc\'e, d\'efinissons d'abord
$$\varepsilon(\LG) = \sup \{\varepsilon\in [0,1],\ \textrm{le bord}\ \dO\ \textrm{est}\ \C^{1+\varepsilon}\ \textrm{en tout point de}\ \LG\}.$$
Ensuite, pour tout élément hyperbolique $\g \in \G$, notons
$$\varepsilon(\g) = \sup \{\varepsilon\in [0,1],\ \textrm{le bord}\ \dO\ \textrm{est}\ \C^{1+\varepsilon}\ \textrm{au point attractif }\ x_{\g}^+ \textrm{ de } \g\},$$
et $\varepsilon(\G) = \inf\{ \varepsilon(\g),\ \g\in\G\ \textrm{hyperbolique}\}$. Ainsi, le bord $\dO$ est $\C^{1+\varepsilon(\G)}$ en tout point fixe hyperbolique. On obtient alors:

\begin{theo}[Th\'eor\`eme \ref{bordexposant}]\label{bordexposant_intro}
Soient $\O$ un ouvert strictement convexe et \`a bord $\C^1$, et $M=\Quo$ une variété géométriquement finie \`a cusps asymptotiquement hyperboliques. On a 
$$\varepsilon(\LG) = \varepsilon(\G).$$
\end{theo}

Notre démonstration de ce th\'eor\`eme est diff\'erente de celle de Guichard et repose sur l'extension d'un th\'eor\`eme de Ursula Hamenst\"adt \cite{MR1279472}, qui s'int\'eresse au meilleur coefficient de contraction d'un flot uniform\'ement hyperbolique.\\
Pour l'\'enoncer, il nous faut d\'efinir les meilleurs coefficients de contraction du flot sur l'ensemble non errant 

$$\chi(\NW) = \sup\{\chi,\ \textrm{il existe}\ C>0\ \textrm{tel que l'in\'egalit\'e (\ref{inegaliteanosov}) ait lieu en tout point de}\ \NW\},$$
et sur les orbites p\'eriodiques
$$\chi(Per)=\inf\{\chi(w)\ |\ w\in \NW\ \text{périodique}\}.$$

\begin{theo}[Th\'eor\`eme \ref{hamenstadt}]\label{hamenstadt_intro}
Soient $\O$ un ouvert strictement convexe et \`a bord $\C^1$, et $M=\Quo$ une variété géométriquement finie \`a cusps asymptotiquement hyperboliques. On a 
$$\chi(Per) = \chi(\NW).$$
\end{theo}

Comme corollaire de ces r\'esultats et du travail pr\'ec\'edent \cite{Crampon:2012fk}, on obtient un r\'esultat de rigidit\'e:

\begin{theo}[Corollaire \ref{rigidite_geofini}]\label{rigidite_geofini_intro}
Soit $\O$ un ouvert strictement convexe et \`a bord $\C^1$, qui admet une action g\'eom\'etriquement finie d'un groupe $\G$ contenant un \'el\'ement parabolique. Si le bord $\dO$ est de classe $\C^{1+\varepsilon}$ pour tout $0<\varepsilon<1$, alors $\G$ est un sous-groupe d'un conjugu\'e de $\SO$.
\end{theo}

Comme cas particulier, on obtient une version volume fini d'un th\'eor\`eme de Benoist \cite{MR2094116} qui concernait les quotients compacts:

\begin{coro}[Corollaire \ref{rigidite_volfini}]\label{rigidite_volfini_intro}
Soit $\O$ un ouvert strictement convexe et \`a bord $\C^1$ qui admet un quotient de volume fini non compact. Si le bord $\dO$ est de classe $\C^{1+\varepsilon}$ pour tout $0<\varepsilon<1$, alors $\O$ est un ellipso\"ide.
\end{coro}

Remarquons que dans l'\'enonc\'e de ce th\'eor\`eme, tout comme dans celui du th\'eor\`eme \ref{regbord_volfini_intro}, l'une des hypoth\`eses strictement convexe/\`a bord $\C^1$  est superflue: c'est une cons\'equence du travail de Daryl Cooper, Darren Long et Stephan Tillmann \cite{Cooper:2011fk}.\\

\`A la fin de ce texte, on revient sur la repr\'esentation sph\'erique de $\s2$ dans $\s5$, que nous avions \'etudi\'e dans \cite{Crampon:2012fk} car elle permettait de distinguer les deux notions de finitude g\'eom\'etrique que nous y avions introduites. En particulier, on avait vu que l'ensemble des ouverts proprement convexes pr\'eserv\'es par cette repr\'esentation formait, \`a action de $\s5$ pr\`es, une famille croissante $\{\O_r,\ 0\leqslant r\leqslant \infty\}$. Parmi eux, les convexes $\O_0$ et $\O_{\infty}$, duaux l'un de l'autre, n'\'etaient ni strictement convexes ni \`a bord $\C^1$. Les autres par contre l'\'etaient. En fait, on peut d\'eterminer pr\'ecis\'ement leur r\'egularit\'e (voir d\'efinition \ref{defi_calpha} pour les notions de r\'egularit\'e $\C^{1+\varepsilon}$ et la $\beta$-convexit\'e):

\begin{prop}[Proposition \ref{reg_spherique}]\label{reg_spherique_intro}
Pour $0<r<\infty$, le bord de l'ouvert convexe $\O_r$ est de classe $\C^{4/3}$ et $4$-convexe.
\end{prop}

Nous avons inclus pour finir une premi\`ere \'etude de la croissance des groupes discrets dont l'action est g\'eom\'etriquement finie sur $\O$. L'objet principal est l'exposant critique $\dgg$ du groupe $\G$, qui mesure la croissance exponentielle du groupe agissant sur $\O$; \`a savoir
$$\dgg = \limsup_{R\to +\infty} \frac{1}{R} \log \sharp\{g\in \G,\ d_{\O}(x,gx)\leqslant R\}.$$
Lorsque $\G$ est un groupe cocompact, il est imm\'ediat que l'exposant critique et l'entropie volumique de la g\'eom\'etrie de Hilbert sont \'egaux. Rappelons que l'entropie volumique de la g\'eom\'etrie de Hilbert $(\O,\d)$ est le taux de croissance exponentiel des volumes des boules:
$$h_{vol}(\O) = \limsup_{R\to +\infty} \frac{1}{R} \log \Vol_{\O} B(x,R).$$
Lorsque $\G$ n'est plus cocompact, on a de fa\c con g\'en\'erale $\dgg\leqslant h_{vol}$ mais il n'y a a priori plus de raisons pour que ces deux quantit\'es co\"incident, m\^eme si $\G$ est de covolume fini: Fran\c coise Dal'bo, Marc Peign\'e, Jean-Claude Picaud et Andrea Sambusetti ont construit des exemples de r\'eseaux non uniformes d'espaces de courbure n\'egative pinc\'ee o\`u $\dgg<h_{vol}$.\\
Dans notre cas, le fait que les cusps d'une vari\'et\'e de volume fini soient asymptotiquement hyperboliques entraîne l'\'egalit\'e:

\begin{theo}[Th\'eor\`eme \ref{volegaltop}]\label{volegaltop_intro}
Soient $\O$ un ouvert strictement convexe et \`a bord $\C^1$, et $\G$ un sous-groupe discret de $\Aut(\O)$ de covolume fini. Alors
$$\delta_{\G} = h_{vol}(\O).$$
\end{theo}

Ce r\'esultat peut m\^eme s'\'etendre au cas des actions g\'eom\'etriquement finies de la mani\`ere suivante:

\begin{theo}[Th\'eor\`eme \ref{volegaltopgeneral}]\label{volegaltopgeneral_intro}
Soit $\G$ un sous-groupe discret de $\Aut(\O)$ dont l'action sur $\O$ est géométriquement finie. Alors
$$\delta_{\G} = \limsup_{R\to +\infty} \frac{1}{R} \log \Vol_{\O}(B(o,R)\cap C(\LG)),$$
o\`u $o$ est un point quelconque de $\O$.
\end{theo}

\paragraph*{Plan}
Les sections 2 et 3 sont des pr\'eliminaires portant respectivement sur les g\'eom\'etries de Hilbert et leur flot g\'eod\'esique.\\
La section 4 explique ce qui nous sera utile sur les vari\'et\'es g\'eom\'etriquement finies, en pr\'esentant notamment l'hypoth\`ese d'asymptoticit\'e hyperbolique des cusps.\\
La section 5 est consacr\'ee \`a la démonstration du th\'eor\`eme \ref{anosovintro}. Bien que l'id\'ee soit claire et tr\'es simple, la d\'emonstration reste malgr\'e tout quelque peu technique.\\
La section 6 se concentre sur les propri\'et\'es de r\'ecurrence du flot g\'eod\'esique d'une vari\'et\'e quelconque $M=\Quo$; en particulier, on y montre la proposition \ref{melange_intro}.\\
Dans la section 7, on s'int\'eresse \`a la r\'egularit\'e du bord de l'ouvert convexe. C'est l\`a qu'on montre le th\'eor\`eme \ref{regbord_intro} et le corollaire \ref{regbord_volfini_intro}. Une bonne partie de cette section est d\'edi\'ee au th\'eor\`eme \ref{bordexposant_intro}, via le th\'eor\`eme \ref{hamenstadt_intro} dont la démonstration pr\'esente quelques technicit\'es.\\
La section 8 construit le contre-exemple de la proposition \ref{contreex_intro} et d\'etaille la proposition \ref{reg_spherique_intro}.
Enfin, dans la partie 9, on montre les th\'eor\`emes \ref{volegaltop_intro} et \ref{volegaltopgeneral_intro} qui lient exposant critique et entropie volumique. L\`a encore, les démonstrations pr\'esentent quelques difficult\'es techniques.

\paragraph*{Remerciements} Nous tenons \`a remercier Fran\c coise Dal'bo pour son int\'er\^et, notamment pour la partie 9, ainsi que Fran\c cois Ledrappier qui nous a indiqu\'e l'article \cite{MR1279472} de Hamenst\"adt.\\
Le premier auteur est financ\'e par le programme FONDECYT N$^\circ$ 3120071 de la CONICYT (Chile) .

\mainmatter

\section{G\'eom\'etries de Hilbert}

\subsection{Distance et volume}\label{para_def_dist}
\par{
Une \emph{carte affine} $A$ de $\PP^n$ est le complémentaire d'un hyperplan projectif. Une carte affine possède une structure naturelle d'espace affine. Un ouvert $\O$ de $\PP^n$ différent de $\PP^n$ est \emph{convexe} lorsqu'il est inclus dans une carte affine et qu'il est convexe dans cette carte. Un ouvert convexe $\O$ de $\PP^n$ est dit \emph{proprement convexe} lorsqu'il existe une carte affine contenant son adhérence $\overline{\O}$. Autrement dit, un ouvert convexe est proprement convexe lorsqu'il ne contient pas de droite affine. Un ouvert proprement convexe $\O$ de $\PP^n$ est dit \emph{strictement convexe} lorsque son bord $\partial \O$ ne contient pas de segment non trivial. 
}
\\
\label{base}

Hilbert a introduit sur un ouvert proprement convexe $\O$ de $\PP^n$ la distance qui porte aujourd'hui son nom. Pour $x \neq y \in \O$, on note $p,q$ les points d'intersection de la droite $(xy)$ et du bord $\partial \O$ de $\O$, de telle fa\c con que $x$ soit entre $q$ et $y$, et $y$ entre $x$ et $p$ (voir figure \ref{dist}). On pose

$$
\begin{array}{ccc}
d_{\O}(x,y) = \displaystyle\frac{1}{2}\ln \big([p:q:x:y]\big) = \displaystyle\frac{1}{2}\ln \bigg(\frac{|qy|\cdot |
px|}{|qx| \cdot |py|} \bigg) & \textrm{et} &
d_{\O}(x,x)=0,
\end{array}
$$
o\`u
\begin{enumerate}
\item la quantit\'e $[p:q:x:y]$ désigne le birapport des points $p,q,x,y$;
\item $| \cdot |$ est une norme euclidienne quelconque sur une carte affine $A$ qui contient l'adhérence $\overline{\O}$ de $\O$.
\end{enumerate}

Le birapport \'etant une notion projective, il est clair que $d_{\O}$ ne dépend ni du choix de $A$, ni du choix de la norme euclidienne sur $A$.

\begin{center}
\begin{figure}[h!]
  \centering
\includegraphics[width=6cm]{hilbertdistance.jpg}\ \includegraphics[width=7cm]{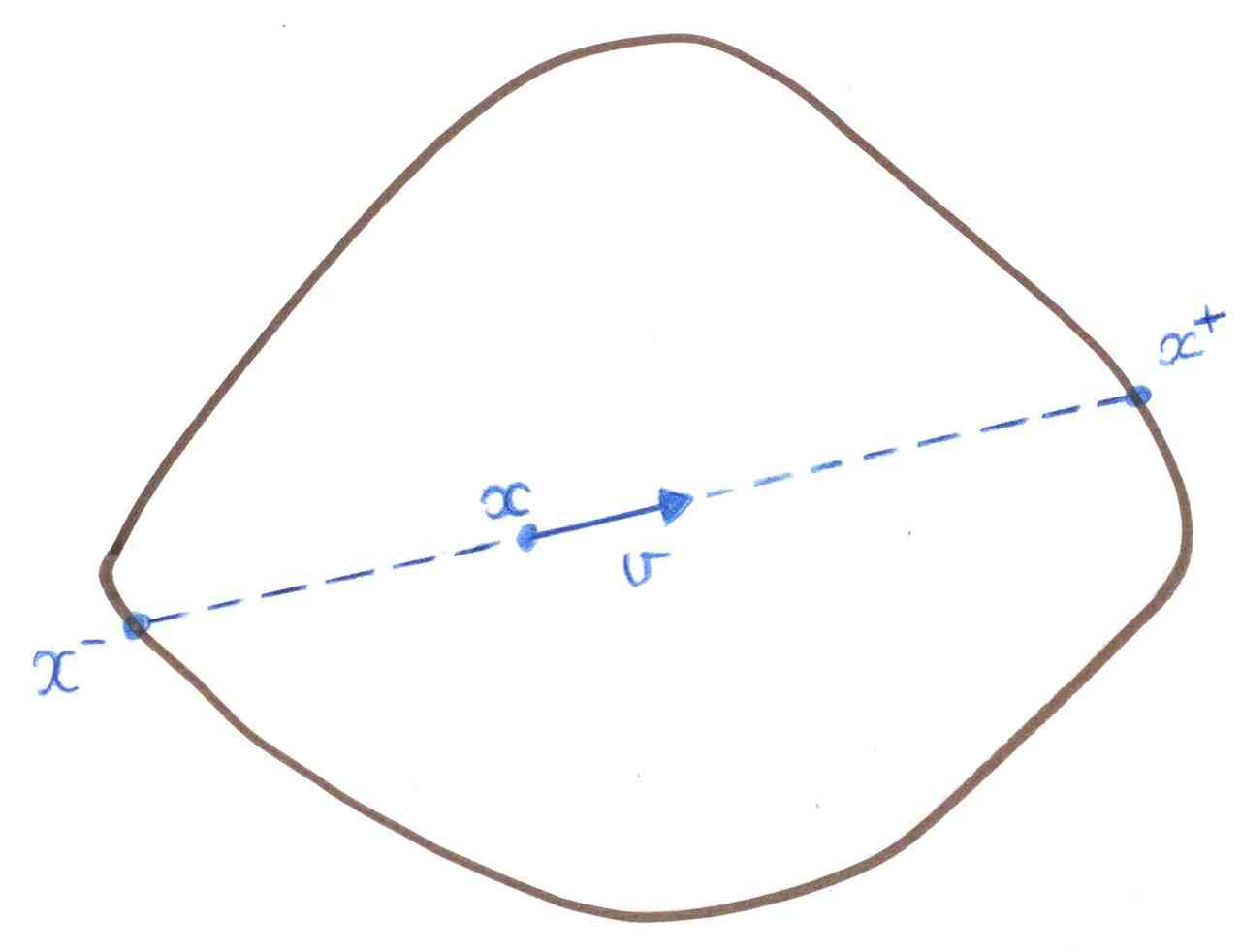}
\caption{La distance de Hilbert et la norme de Finsler}
\label{dist}
\end{figure}
\end{center}

\begin{fait}
Soit $\O$ un ouvert proprement convexe de $\PP^n$.
\begin{enumerate}
\item $d_{\O}$ est une distance sur $\O$;
\item $(\O,d_{\O})$ est un espace métrique complet;
\item La topologie induite par $d_{\O}$ coïncide avec celle induite par $\PP^n$;
\item Le groupe $\Aut(\O)$ des transformations projectives de $\ss$ qui préservent $\O$ est un sous-groupe fermé de $\ss$ qui agit par isométries sur $(\O,d_{\O})$. Il agit donc proprement sur $\O$.
\end{enumerate}
\end{fait}

La distance de Hilbert $d_{\O}$ est induite par une structure finslérienne sur l'ouvert $\O$. On choisit une carte affine $A$ et une m\'etrique euclidienne $|\cdot|$ sur $A$ pour lesquelles $\O$ appara\^it comme un ouvert convexe born\'e. On identifie le fibré tangent $T \O$ de $\O$ à $\O \times A$. Soient $x \in \O$ et $v \in A$, on note $x^+=x^+(x,v)$ (resp. $x^-$) le point d'intersection de la demi-droite définie par $x$ et $v$ (resp $-v$) avec $\partial \O$ (voir figure \ref{dist}). On pose
\begin{equation}\label{metrique_finsler}
 F(x,v) = \frac{|v|}{2}\Bigg(\frac{1}{|xx^-|} + \frac{1}{| xx^+|} \Bigg),
\end{equation}
quantit\'e indépendante du choix de $A$ et de $|\cdot|$, puisqu'on ne consid\`ere que des rapports de longueurs.


\begin{fait}
Soient $\O$ un ouvert proprement convexe de $\PP^n$ et $A$ une carte affine qui contient $\overline{\O}$. La distance induite par la métrique finslérienne $F$ est la distance $d_{\O}$. Autrement dit on a les formules suivantes:
\begin{itemize}
\item $\displaystyle{F(x,v) = \left. \frac{d}{dt}\right| _{t=0} d_{\O}(x,x+tv)}$, pour $v \in A$;
\item $d_{\O}(x,y) = \inf \displaystyle\int_0^1 F(\dot\sigma(t))\ dt$, où l'infimum est pris sur les chemins $\sigma$ de classe $\C^1$ tel que $\sigma(0)=x$ et $\sigma(1)=y$.
\end{itemize}
\end{fait}

\par{
Il y a plusieurs mani\`eres naturelles de construire un volume pour une g\'eom\'etrie de Finsler, la d\'efinition riemannienne acceptant plusieurs g\'en\'eralisations. Nous travaillerons avec le volume de Busemann, not\'e $\textrm{Vol}_{\O}$.\\
Pour le construire, on se donne une carte affine $A$ et une m\'etrique euclidienne $|\cdot|$ sur $A$ pour lesquelles $\O$ appara\^it comme un ouvert convexe born\'e. On note $B_{T_x\O}(r) = \{ v \in T_x \O \, | \, F(x,v) < r \}$ la boule de rayon $r>0$ de l'espace tangent à $\O$ en $x$, $\textrm{Vol}$ la mesure de Lebesgue sur $A$ associ\'ee \`a $|\cdot|$ et $v_n=\textrm{Vol}(\{ v \in A \, | \, |v| < 1 \})$ le volume de la boule unit\'e euclidienne en dimension $n$.

\par{
Pour tout borélien $\mathcal{A} \subset \O \subset A$, on pose:
$$\textrm{Vol}_{\O} (\mathcal{A})= \int_{\mathcal{A}} \frac{v_n}{\textrm{Vol}(B_{T_x\O}(1))}\ d\Vol(x)$$
L\`a encore, la mesure $\textrm{Vol}_{\O}$ est indépendante du choix de $A$ et de $|\cdot|$. En particulier, elle est préservée par le groupe $\Aut(\O)$.
}

La proposition suivante permet de comparer deux géométries de Hilbert entre elles.

\begin{prop}\label{compa}
Soient $\O_1$ et $\O_2$ deux ouverts proprement convexes de $\PP^n$ tels que $\O_1 \subset \O_2$.
\begin{itemize}
\item Les m\'etriques finsl\'eriennes $F_1$ et $F_2$ de $\O_1$ et $\O_2$ v\'erifient: $F_2(w) \leqslant F_1(w)$, $w \in T \O_1 \subset T \O_2$, l'\'egalit\'e ayant lieu si et seulement si $x^+_{\O_1}(w)=x^+_{\O_2}(w)$ et $x^-_{\O_1}(w)=x^-_{\O_2}(w)$.
\item Pour tous $x,y \in \O_1$, on a $d_{\O_2}(x,y) \leqslant d_{\O_1}(x,y)$.
\item Les boules m\'etriques et m\'etriques tangentes vérifient, pour tout $x \in \O_1$ et $r>0$, $B_{\O_1}(x,r) \subset B_{\O_2}(x,r)$ et $B_{T_x\O_1}(r) \subset B_{T_x\O_2}(r)$, avec \'egalit\'e si et seulement si $\O_1=\O_2$.
\item Pour tout bor\'elien $\mathcal{A}$ de $\O_1$, on a $\textrm{Vol}_{\O_2}(\mathcal{A}) \leqslant \textrm{Vol}_{\O_1}(\mathcal{A})$.
\end{itemize}
\end{prop}

\subsection{Fonctions de Busemann et horosphères}\label{busemannhoro}

Nous supposons dans ce paragraphe que l'ouvert proprement convexe $\O$ de $\PP^n$ est strictement convexe et \`a bord $\C^1$. Dans ce cadre, il est possible de d\'efinir les fonctions de Busemann et les horosphères de la m\^eme mani\`ere qu'en g\'eom\'etrie hyperbolique, et nous ne donnerons pas de d\'etails.\\

Pour $\xi\in\dO$ et $x\in \O$, notons $c_{x,\xi}: [0,+\infty)\longrightarrow\O$ la g\'eod\'esique issue de $x$ et d'extr\'emit\'e $\xi$, soit $c_{x,\xi}(0)=x$ et $c_{x,\xi}(+\infty)=\xi$. La \emph{fonction de Busemann} bas\'ee en $\xi\in\dO$ $b_{\xi}(.,.):\O\times\O\longrightarrow\R$ est d\'efinie par:
$$b_{\xi}(x,y) = \lim_{t\to+\infty} \d(y,c_{x,\xi}(t)) - t = \lim_{z\to\xi} \d(y,z) - \d(x,z),\ x,y\in\O.$$
L'existence de ces limites est due aux hypoth\`eses de r\'egularit\'e faites sur $\O$. Les fonctions de Busemann sont de classe $\C^1$.\\

\noindent \emph{L'horosph\`ere} bas\'ee en $\xi\in\dO$ et passant par $x\in\O$ est l'ensemble
$$\H_{\xi}(x) = \{y\in\O \ | \ b_{\xi}(x,y) = 0\}.$$
\emph{L'horoboule} bas\'ee en $\xi\in\dO$ et passant par $x\in\O$ est l'ensemble
$$H_{\xi}(x) = \{y\in\O \ | \ b_{\xi}(x,y) < 0\}.$$
L'horoboule bas\'ee en $\xi\in\dO$ et passant par $x\in\O$ est un ouvert strictement convexe de $\O$, dont le bord est l'horosph\`ere correspondante, qui est elle une sous-vari\'et\'e de classe $\C^1$ de $\O$.\\
Dans une carte affine $A$ dans laquelle $\O$ appara\^it comme un ouvert convexe relativement compact, on peut, en identifiant $T\O$ avec $\O \times A$, construire g\'eom\'etriquement l'espace tangent \`a $\H_{\xi}(x)$ en $x$: c'est le sous-espace affine contenant $x$ et l'intersection $T_{\xi}\dO\cap T_{\eta}\dO$ des espaces tangents \`a $\dO$ en $\xi$ et $\eta=(x\xi)\cap\dO\smallsetminus\{\xi\}$.\\
On peut voir que l'horoboule et l'horosph\`ere bas\'ees en $\xi\in\dO$ et passant par $x\in\O$ sont les limites des boules et des sphères m\'etriques centr\'ees au point $z\in\O$ et passant par $x$ lorsque $z$ tend vers $\xi$.

\begin{center}
\begin{figure}[h!]
  \centering
\includegraphics[width=7cm]{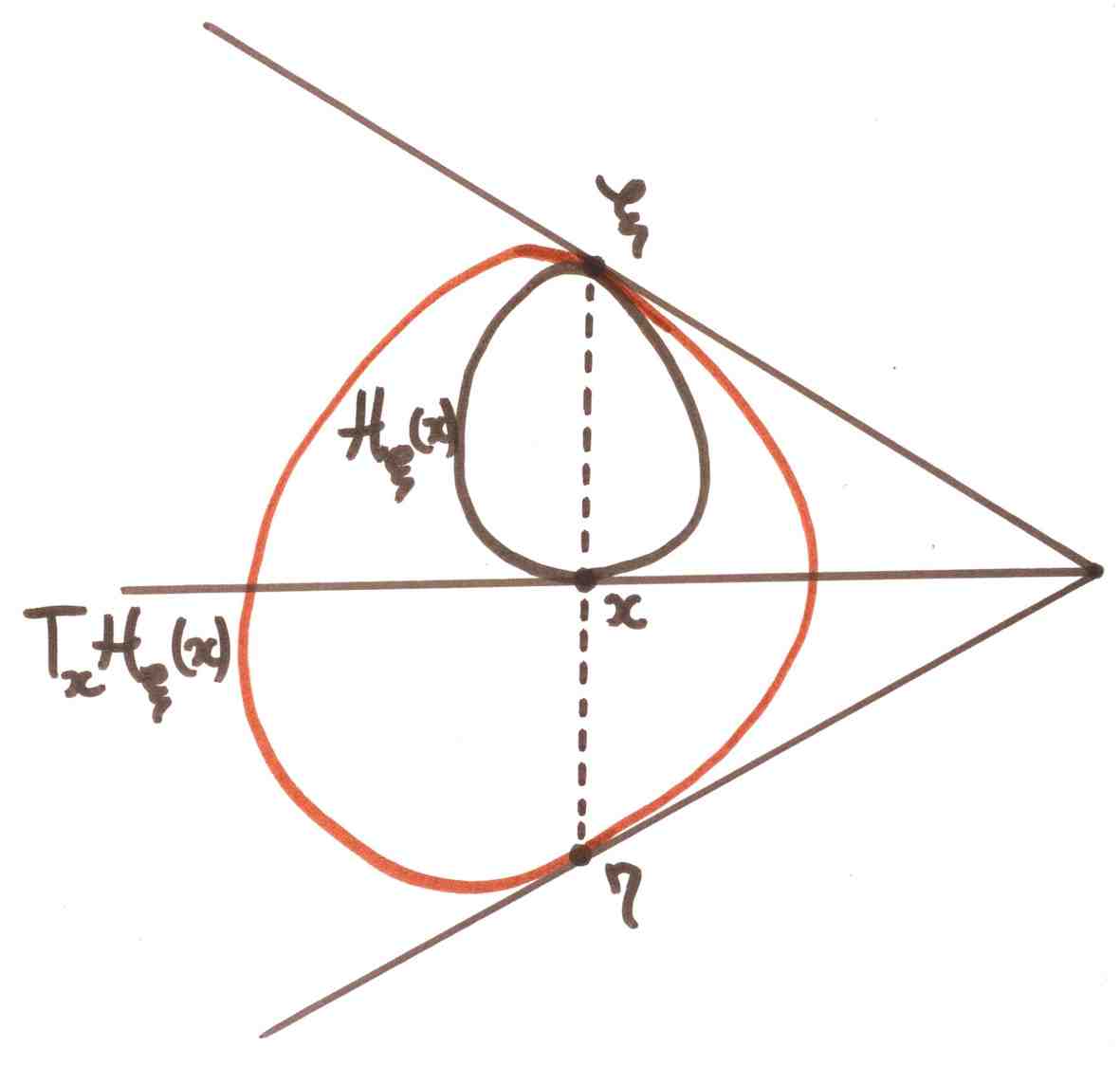}
\caption{Une horosph\`ere et son espace tangent}
\label{horo}
\end{figure}
\end{center}

\subsection{Dualité}\label{def_dualite}

\`A l'ouvert proprement convexe $\O$ de $\PP^n$ est associé l'ouvert proprement convexe dual $\O^*$: on considère un des deux cônes $C\subset \R^{n+1}$ au-dessus de $\O$, et son dual
$$C^* = \{ f\in (\R^{n+1})^*,\ \forall x\in C,\ f(x)>0\}.$$
Le convexe $\O^*$ est par définition la trace de $C^*$ dans $\PP((\R^{n+1})^*)$.\\

Le bord de $\dO^*$ est facile à comprendre, car il s'identifie à l'ensemble des hyperplans tangents à $\O$. En effet, un hyperplan tangent $T_x$ à $\dO$ en $x$ est la trace d'un hyperplan $H_x$ de $\R^{n+1}$. L'ensemble des formes linéaires dont le noyau est $H_x$ forme une droite de $(\R^{n+1})^*$, dont la trace $x^*$ dans $\PP((\R^{n+1})^*)$ est dans $\dO^*$. Il n'est pas dur de voir qu'on obtient ainsi tout le bord $\dO^*$.\\
Cette remarque permet de voir que le dual d'un ouvert strictement convexe a un bord de classe $\C^1$, et inversement. En particulier, lorsque $\O$ est strictement convexe et que son bord est de classe $\C^1$, ce qui est le cas que nous étudierons, on obtient une involution continue $x\longmapsto x^*$ entre les bords de $\O$ et $\O^*$.\\

Tout sous-groupe $\G$ de $\ss$ agit par dualité sur $(\R^{n+1})^*$ et donc sur $\PP((\R^{n+1})^*)$, via la formule suivante:
\[(\g\cdot f)(x) = f(\g^{-1}x),\ \g\in\G,\ x\in \R^{n+1}.\]
Le convexe dual $\O^*$ est préservé par un \'el\'ement $\g \in \ss$ si et seulement si $\O$ est pr\'eserv\'e par $\g$. On obtient de cette fa\c con une action de tout sous-groupe $\G$ de $\Aut(\O)$ sur le convexe dual $\O^*$.
Le sous-groupe discret de $\Aut(\O^*)$ ainsi obtenu sera noté $\G^*$. Bien entendu, on a $(\O^*)^*=\O$ et $(\G^*)^*=\G$.\\

\vspace*{.5cm}
{\bf \large Dans tout ce qui suit, sauf mention explicite, $\O$  désignera un ouvert proprement convexe, strictement convexe et à bord $\C^1$.}
\vspace*{.5cm}

\subsection{Isom\'etries}\label{para_isometrie}

Les isom\'etries d'une g\'eom\'etrie $(\O,\d)$ avec $\O$ strictement convexe \`a bord $\C^1$ ont \'et\'e classifi\'ees dans \cite{Crampon:2012fk}. Ce sont toutes des transformations projectives qui pr\'eservent $\O$, et, quitte \`a consid\'erer leur carr\'e, on les verra donc comme des \'el\'ements du groupe lin\'eaire $\ss$, agissant sur $\PP^n$. Outre les isom\'etries elliptiques qui sont de torsion et qui ne nous int\'eresseront pas ici, on trouve les isom\'etries hyperboliques et paraboliques.
\begin{itemize}
 \item Une \emph{isom\'etrie hyperbolique} $\g$ a exactement deux points fixes $x_{\g}^+,x_{\g}^- \in \partial \O$, l'un répulsif et l'autre attractif. Cela veut dire que la suite $(\g^n)_{n \in \N}$ converge uniformément sur les compacts de $\overline{\O}\smallsetminus \{ x_{\g}^-\}$ vers $x_{\g}^+$, et la suite $(\g^{-n})_{n \in \N}$ converge uniformément sur les compacts de $\overline{\O}\smallsetminus \{ x_{\g}^+\}$ vers $x_{\g}^-$. De plus, les valeurs propres $\l_{0}(\g)$ et $\l_n(\g)$ associ\'ees aux points fixes $x_{\g}^+$ et $x_{\g}^-$ sont positives: c'est une cons\'equence du fait que le rayon spectral de $\g$ est valeur propre de $\g$ (lemme 2.3 de \cite{MR2218481} par exemple); elles sont de multiplicité 1 car, sinon, il y aurait un segment dans le bord de $\O$. Finalement, $\g$ agit par translation sur le segment ouvert $]x_{\g}^-x_{\g}^+[$ de $\O$, translation de force $\tau(\g)=\ln \frac{\l_{0}(\g)}{\l_n(\g)}$.
\item Une \emph{isom\'etrie parabolique} $\g$ a exactement un point fixe $p \in \partial \O$ et préserve toute horosphère basée en $p$. De plus, la famille $(\g^n)_{n \in \Z}$ converge uniformément sur les compacts de $\overline{\O}\smallsetminus \{ p\}$ vers $p$.
\end{itemize}

On dira qu'un sous-groupe discret $\P$ de $\Aut(\O)$, sans torsion, est \emph{parabolique} si tous ses \'el\'ements sont paraboliques. Un tel groupe est nilpotent et ses \'el\'ements fixent un m\^eme point $p\in\dO$. On dira que le groupe $\P$ est \emph{de rang maximal} si son action sur $\dO\smallsetminus\{p\}$ est cocompacte.\\
Si un sous-groupe discret $\G$ de $\Aut(\O)$ est donn\'e, on dira qu'un sous-groupe parabolique de $\G$ est \emph{maximal} s'il n'est contenu dans aucun autre sous-groupe parabolique.

\subsection{Ensemble limite}

Comme en géométrie hyperbolique, on peut définir l'ensemble limite et le domaine de discontinuité d'un sous-groupe discret de $\Aut(\O)$ de la façon suivante. On utilise ici de fa\c con essentielle la stricte convexit\'e de $\O$.

\begin{defi}\label{def_ens_lim}
Soit $\G$ un sous-groupe discret de $\Aut(\O)$ et $x\in\O$. L'ensemble limite $\LG$ de $\G$ est le sous-ensemble de $\dO$ suivant:
$$\LG = \overline{\G\cdot x} \smallsetminus \G\cdot x.$$
\emph{Le domaine de discontinuité} $\Og$ de $\G$ est le complémentaire de l'ensemble limite de $\G$ dans $\dO$.
\end{defi}

L'ensemble limite $\LG$, s'il n'est pas infini, est vide ou consiste en 1 ou 2 points. On dit que $\G$ est \emph{non élémentaire} si $\LG$ est infini. Dans ce dernier cas, l'ensemble limite $\LG$ est le plus petit fermé $\G$-invariant non vide de $\dO$. En particulier, $\LG$ est l'adhérence des points fixes des éléments hyperboliques de $\G$.

\begin{defi}
Soit $\Gamma$ un sous-groupe de $\ss$. On dira que $\G$ est \emph{irréductible} lorsque les seuls sous-espaces vectoriels de $\R^{n+1}$ invariants par $\G$ sont $\{0 \}$ et $\R^{n+1}$. On dira que $\G$ est \emph{fortement irréductible} si tous ses sous-groupes d'indice fini sont irréductibles, autrement dit, si $\G$ ne préserve pas une union finie de sous-espaces vectoriels non triviaux.
\end{defi}

\begin{lemm}\label{irreductible}
Soit $\G$ un sous-groupe discret de $\Aut(\O)$. Les propositions suivantes sont \'equivalentes:
\begin{enumerate}[(i)]
 \item l'ensemble limite $\LG$ de $\G$ engendre $\PP^n$;
 \item le groupe $\G$ est irr\'eductible;
 \item le groupe $\G$ est fortement irr\'eductible.
\end{enumerate}
\end{lemm}

\begin{proof}

L'implication (i)$\Rightarrow$(ii) vient du fait que $\LG$ est l'adhérence des points fixes des éléments hyperboliques de $\G$. Pour les implications (ii)$\Rightarrow$(i) et (iii)$\Rightarrow$(i), il suffit de voir que l'espace engendr\'e par $\LG$ est invariant par $\G$.

Montrons pour finir l'implication (i)$\Rightarrow$(iii). Supposons donc que $\LG$ engendre $\PP^n$. Si $G$ est un sous-groupe d'indice fini de $\G$, alors, pour tout élément hyperbolique $h$ de $\G$, il existe un entier $n \geqslant 1$ tel que $h^n\in G$. Ainsi, $\Lambda_{G}=\LG$ et donc $\Lambda_{G}$ engendre $\PP^n$, ce qui \'equivaut \`a l'irr\'eductibilit\'e de $G$.

\end{proof}


\section{Le flot géodésique}

\subsection{G\'en\'eralit\'es}

Le flot g\'eod\'esique est le principal objet d'\'etude de ce travail. Nous le d\'efinirons sur le fibré tangent homogène, ou en demi-droites, de $\O$, qui est le fibré $\pi : H\O \longrightarrow \O$, avec $H\O =\Quotient{T\O\smallsetminus\{0\}}{\R_+}$:
deux points $(x,u)$ et $(y,v)$ de $T\O\smallsetminus\{0\}$ sont identifiés si $x=y$ et $u=\l v$ pour un certain réel $\l>0$.\\
L'image d'un point $w=(x,[\xi]) \in H\O$ par le flot géodésique $\ph^t : H\O \longrightarrow H\O$ est le point $\ph^t(w)=(x_t,[\xi_t])$ obtenu en suivant la g\'eod\'esique partant de $x$ dans la direction $[\xi]$ pendant le temps $t$. Il est engendré par le champ de vecteurs $X$ sur $H\O$, qui a la même régularité que le bord de $\O$. Ainsi, $\ph^t$ est au moins de classe $\C^1$.\\

Nous ferons les calculs de fa\c con intelligente en utilisant l'invariance projective. Une \emph{carte adapt\'ee \`a un point $w\in H\O$} est une carte affine munie d'une m\'etrique euclidienne telle que
\begin{itemize}
 \item la fermeture de $\O$ est incluse dans la carte;
 \item l'intersection des plans tangents \`a $\dO$ en $x^+$ et $x^-$ sont \`a l'infini de la carte; autrement dit, ils y sont parall\`eles;
 \item la droite $(xx^+)$ et les plans tangents \`a $\dO$ en $x^+$ et $x^-$ sont orthogonaux.
\end{itemize}

\begin{center}
\begin{figure}[h!]
  \centering
\includegraphics[width=7cm]{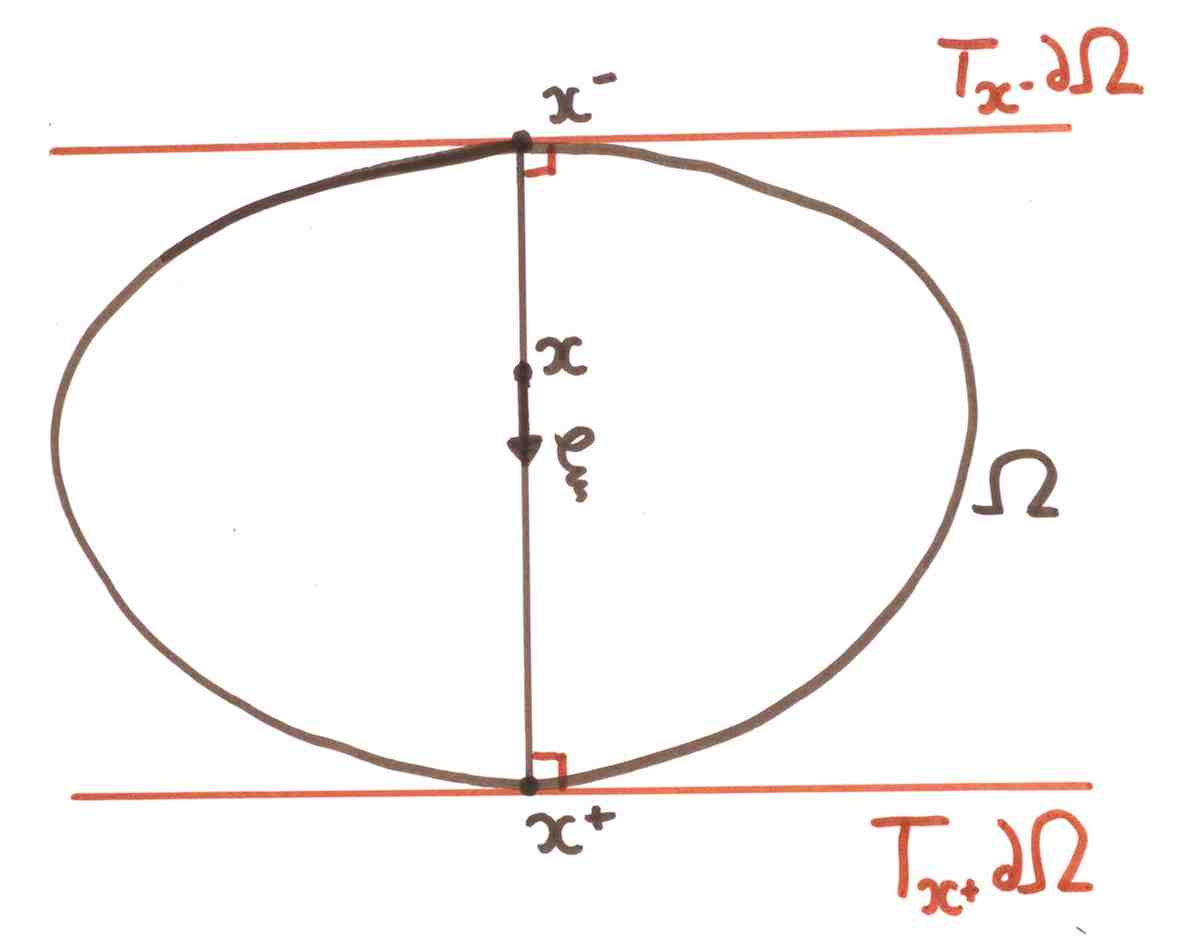}
\caption{Une carte adapt\'ee en $w$}
\label{goodchart}
\end{figure}
\end{center}

\subsection{Variétés stable et instable}

On définit les variétés stable $W^s(w)$ et instable $W^u(w)$ de $w=(x,[\xi])\in H\O$ par
$$W^s(w) = \{w'=(y,[yx^+]),\ y\in\Hh_{x^+}(x)\},\ W^u(w) = \{w'=(y,[x^-y]),\ y\in\Hh_{x^-}(x)\}.$$
Il n'est pas difficile de voir que, comme $\O$ est strictement convexe à bord $\C^1$, on a
$$\begin{array}{rcl}
W^s(w)&=&\{w'\in H\O\ |\ \displaystyle\lim_{t\to +\infty} d_{\O}(\pi\ph^t(w),\pi\ph^t(w'))=0\}\\

W^u(w)&=&\{w'\in H\O\ |\ \displaystyle\lim_{t\to -\infty} d_{\O}(\pi\ph^t(w),\pi\ph^t(w'))=0\}.
\end{array}$$
Les sous-espaces stable $E^s$ et instable $E^u$ sont les espaces tangents aux variétés stable et instable. On a clairement que $E^s\cap E^u = \{0\}$ et donc la décomposition
$$TH\O = \R.X \oplus E^s \oplus E^u,$$
qu'on appellera décomposition d'Anosov.\\

\begin{center}
\begin{figure}[h!]
  \centering
\includegraphics[width=6cm]{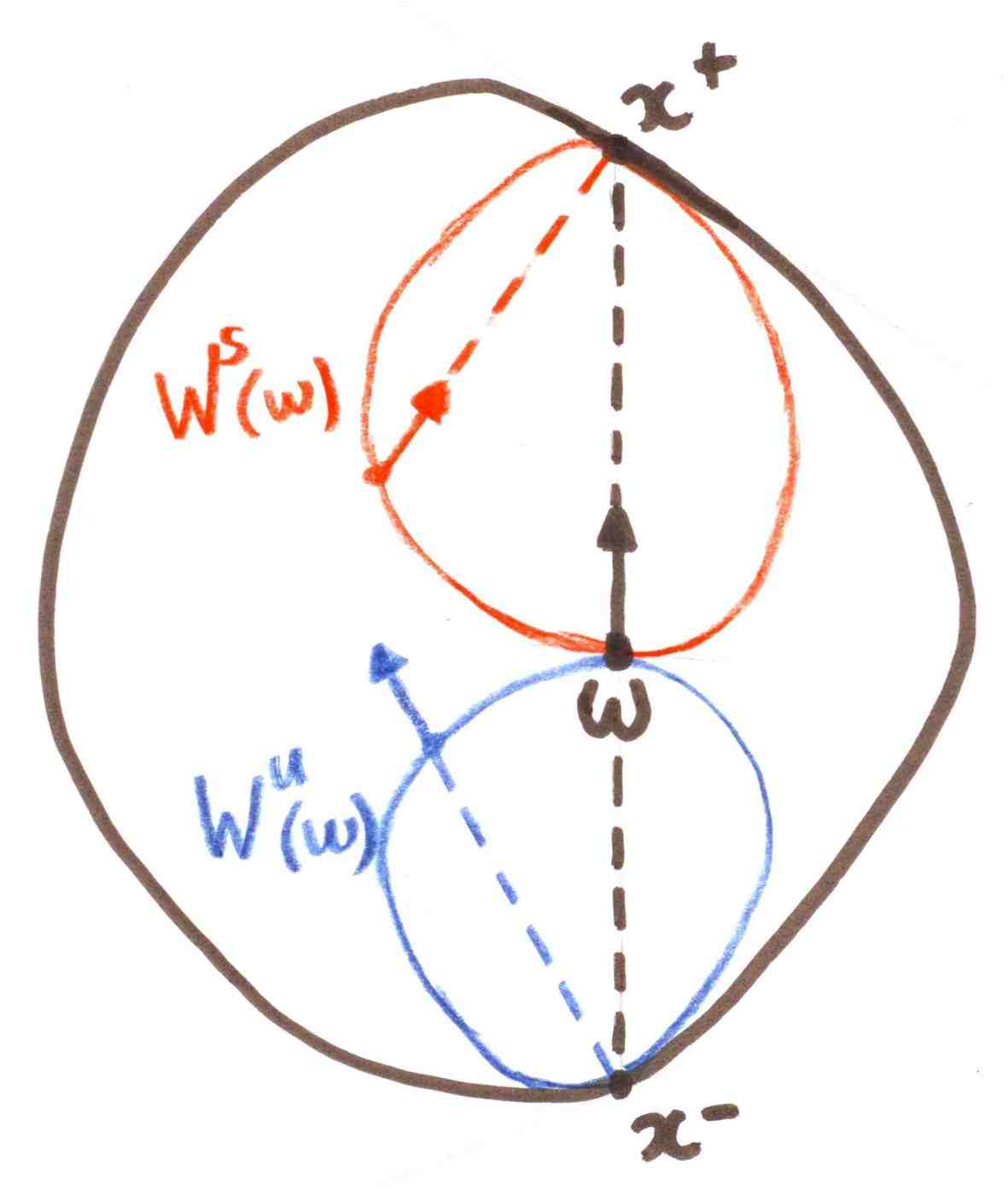}
\caption{Vari\'et\'es stable et instable}
\label{varistable}
\end{figure}
\end{center}

On peut définir une norme de Finsler $\|.\|$ sur $HM$ de la façon suivante : pour $Z = aX + Z^s + Z^u \in \R.X \oplus E^s \oplus E^u$, on pose
\begin{equation}\label{normhm}
 \|Z\| = \left( a^2 + F(d\pi Z^s) + F(d\pi Z^u) \right)^{\frac{1}{2}}.
\end{equation}
Cette m\'etrique est pr\'ecis\'ement celle qui a \'et\'e introduite dans \cite{Crampon:2011fk2}, et qui appara\^it naturellement via une d\'ecomposition en sous-fibr\'es horizontaux et verticaux. Nous n'aurons toutefois pas besoin ici de ces notions.\\
Remarquons qu'en particulier, si $Z$ est un vecteur tangent stable ou instable, c'est-à-dire $Z=Z^s$ ou $Z=Z^u$, on a $\|Z^s\| = F(d\pi Z^s)$ ou  $\|Z^u\| = F(d\pi Z^u)$. \\

Rappelons les deux lemmes suivants, dont les démonstrations permettront de fixer certaines notations.

\begin{lemm}\label{decroissance}
Soient $w\in H\O$ et $Z\in T_w H\O$ un vecteur stable (resp. instable). L'application $t \longmapsto \| d\ph^t(Z)\|$ est une bijection décroissante (resp. croissante) de $(0,+\infty)$ dans $(0,+\infty)$.
\end{lemm}
\begin{proof} Choisissons une carte adaptée au point $w$. Notons $x=\pi w$ et $x_t=\pi \ph^t(w), t\in\R$. Supposons que $Z$ est un vecteur stable, tangent à $H\O$ en $w$ et notons $z=d\pi(Z)$, $z_t = d\pi d\ph^t(Z),\ t\in\R$. Rappelons que par définition de la norme, on a $\|d\ph^t(Z)\| = F(z_t)$. Or,
$$F(z_t) = \frac{|z_t|}{2}\left(\frac{1}{|x_ty_t^+|} + \frac{1}{|x_ty_t^-|}\right),$$
où $y_t^+$ et $y_t^-$ sont les points d'intersection de $x+\R.z$ avec $\dO$. 

\begin{center}
\begin{figure}[h!]
  \centering
\includegraphics[width=7cm]{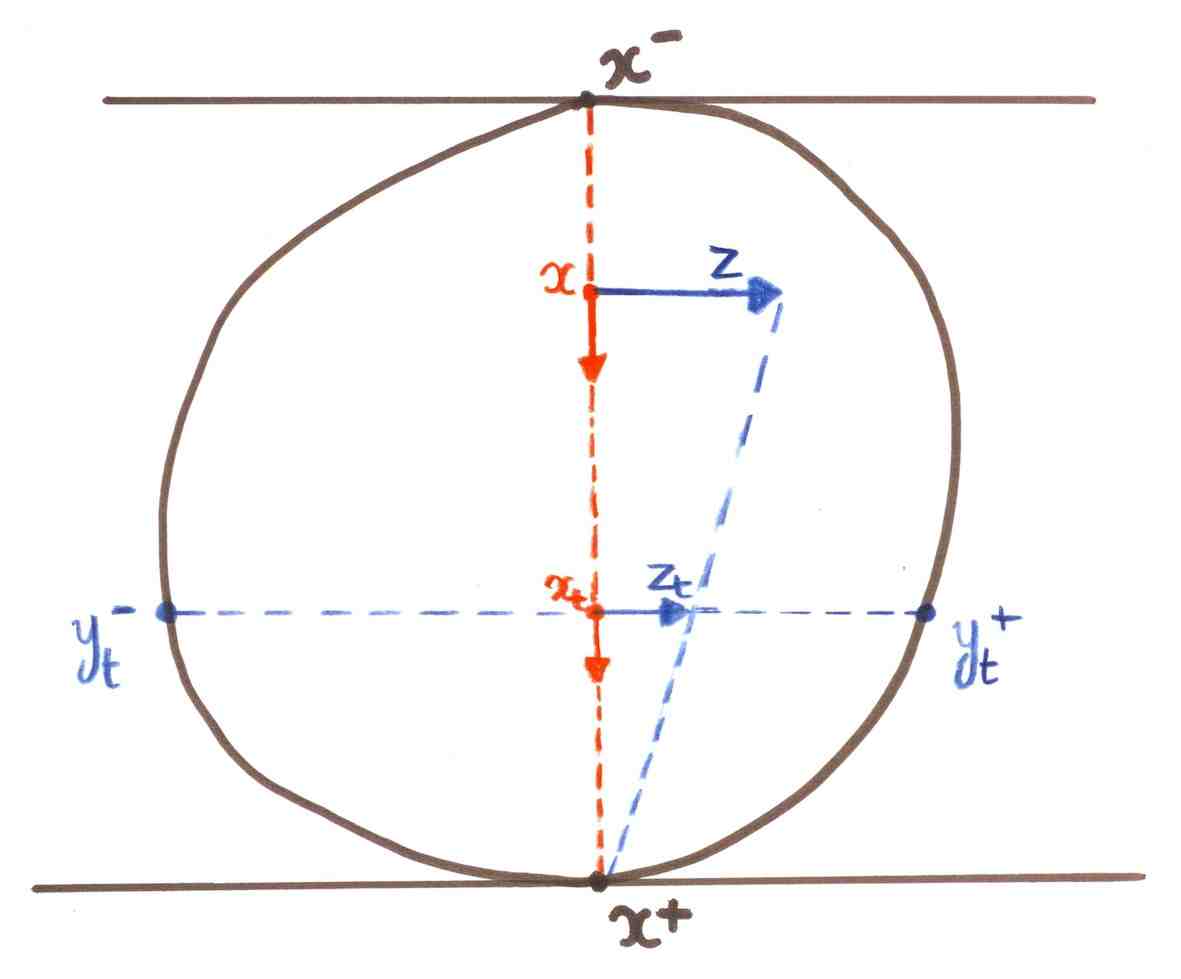}
\caption{Contraction du flot}
\label{flot}
\end{figure}
\end{center}

Si on considère l'application
$$h_t: y\in \Hh_{x^+}(x) \longmapsto y_t = \pi \ph^t(y,[yx^+])= (yx^+)\cap \Hh_{x^+}(x_t),$$ on voit que $z_t$ est en fait donné par
$$z_t=dh_t(z)=\frac{|x_tx^+|}{|xx^+|}z.$$
On obtient ainsi
$$F(z_t) = \frac{|z|}{2|xx^+|}\left(\frac{|x_tx^+|}{|x_ty_t^+|} + \frac{|x_tx^+|}{|x_ty_t^-|}\right).$$
Que $t \longmapsto \|d\ph^t(Z)\|$ soit strictement décroissante est alors une conséquence directe de la stricte convexité de $\O$. La régularité $\C^1$ au point extrémal $\ph^{+\infty}(w)$ de l'orbite de $w$ entraîne que $\|d\ph^t(Z)\|$ tend vers $0$ en $+\infty$. La stricte convexité à l'autre point extrémal $\ph^{-\infty}(w)$ de l'orbite de $w$ implique $\lim_{t\to - \infty}  \|d\ph^t(Z)\| = +\infty$.\\

Dans le cas où $Z$ est un vecteur instable, on obtient, en gardant les mêmes notations:
$$F(z_t) = \frac{|z|}{2|xx^-|}\left(\frac{|x_tx^-|}{|x_ty_t^+|} + \frac{|x_tx^-|}{|x_ty_t^-|}\right).$$
On peut donc appliquer le même raisonnement.
\end{proof}

Remarquons ici le corollaire suivant, qui dit que la décroissance et la croissance du lemme précédent sont contrôlées.

\begin{coro}\label{decroitpasvite}
Pour tout vecteur $Z\in TH\O$, on a
$$e^{-2|t|} \|Z\| \leqslant \|d\ph^t(Z)\| \leqslant e^{2|t|} \|Z\|.$$
\end{coro}
\begin{proof}
Soient $w\in H\O$ et $Z^s \in E^s(w)$ un vecteur stable. Posons $z:= d\pi Z^s$. Soit $Z^u\in E^u(w)$ l'unique vecteur instable tel que $d\pi Z^u = z$. On a vu, dans la démonstration du lemme précédent, et avec les mêmes notations que, pour tout $t\in\R$,
$$\|d\ph^t Z^s\| =  \frac{|z|}{2|xx^+|}\left(\frac{|x_tx^+|}{|x_ty_t^+|} + \frac{|x_tx^+|}{|x_ty_t^-|}\right),$$
et
$$\|d\ph^t Z^u\| = \frac{|z|}{2|xx^-|}\left(\frac{|x_tx^-|}{|x_ty_t^+|} + \frac{|x_tx^-|}{|x_ty_t^-|}\right).$$
Ainsi,
$$\frac{\|d\ph^t Z^s\|}{\|d\ph^t Z^u\|} = \frac{|x_tx^+|}{|x_tx^-|}\frac{|xx^+|}{|xx^-|}.$$
L'égalité $d_{\O}(x,x_t)=t$ implique directement que
$$\frac{|x_tx^+|}{|x_tx^-|}\frac{|xx^+|}{|xx^-|} = e^{-2t},$$
et donc
\begin{equation}\label{liensu}
\frac{\|d\ph^t Z^s\|}{\|d\ph^t Z^u\|} = e^{-2t}.
\end{equation}
Maintenant, le fait que la fonction $t\mapsto \|d\ph^t Z^s\|$ soit décroissante implique que
$$\limsup_{t\to +\infty}\frac{1}{t} \ln \|d\ph^t Z^s\| \leqslant 0;$$
de même, comme $t\mapsto \|d\ph^t Z^u\|$ est croissante,
$$\liminf_{t\to +\infty}\frac{1}{t} \ln \|d\ph^t Z^u\| \geqslant 0.$$
De l'égalité (\ref{liensu}), on déduit donc que
$$\liminf_{t\to +\infty}\frac{1}{t} \ln \|d\ph^t Z^s\| \geqslant -2,$$
et
$$\limsup_{t\to +\infty}\frac{1}{t} \ln \|d\ph^t Z^u\| \leqslant 2.$$
Ainsi, on a, pour tout $t\geqslant 0$ et $Z^s\in E^s\smallsetminus \{0\}$,
$$e^{-2|t|} \leqslant \frac{\|d\ph^t(Z^s)\|}{\|Z^s\|} \leqslant 1.$$
De manière similaire, on obtient pour tout $t\geqslant 0$ et $Z^u\in E^u\smallsetminus \{0\}$,
$$1 \leqslant  \frac{\|d\ph^t(Z^u)\|}{\|Z^u\|} \leqslant e^{2|t|}.$$
On obtient le résultat en décomposant un vecteur $Z$ selon $E^s\oplus E^u\oplus \R.X$.
\end{proof}

\subsection{Ensemble non errant}

Nous voulons par la suite \'etudier des propri\'et\'es de r\'ecurrence du flot g\'eod\'esique. Pour cela, il nous faut regarder l'ensemble des points qui ne partent pas pour toujours \`a l'infini.\\
\'Etant donnée la variété $M=\Quo$, on notera $HM$ le fibré tangent homog\`ene de $M$, quotient du fibré $H\Omega$ par le groupe $\G$. \emph{L'ensemble non errant} du flot géodésique de $M$ est l'ensemble fermé $\NW$ des points $w\in HM$ dont l'orbite passe une infinité de fois dans tout voisinage ouvert de $w$, dans le passé \emph{et} dans le futur. Cet ensemble est naturellement relié à l'ensemble limite: c'est la projection sur $HM$ de l'ensemble
$$\{w=(x,\xi) \in H\O\ |\ x^+(w),x^-(w) \in \LG\}.$$
En particulier, la projection de $\NW$ sur $M$ est incluse dans le c\oe ur convexe de $M$; cela nous permettra d'utiliser le fait \ref{decomposition} ci-dessous lorsque nous seront confront\'es \`a des vari\'et\'es g\'eom\'etriquement finies.

\section{Vari\'et\'es g\'eom\'etriquement finies}

\subsection{D\'ecomposition du c\oe ur convexe}

Les vari\'et\'es g\'eom\'etriquement finies sont le contexte de cet article. Des d\'efinitions \'equivalentes de la finitude g\'eom\'etrique ont \'et\'e donn\'ees dans \cite{Crampon:2012fk}. Rappelons seulement le r\'esultat suivant, essentiel dans le pr\'esent travail (voir la section 8 de \cite{Crampon:2012fk}):

\begin{fait}\label{decomposition}
Soit $M=\Quo$ une vari\'et\'e g\'eom\'etriquement finie. Le c\oe ur convexe $C(M)$ de $M$ est l'union d'un compact $K$ et d'un nombre fini de cusps $\C_i = (H_i\cap \overline{C(\LG)}^{\O})/\P_i$, $1\leqslant i \leqslant l$, o\`u $H_i$ est une horoboule bas\'ee en un point $p_i\in\dO$, et $\P_i$ est le sous-groupe parabolique maximal de $\G$ fixant $p_i$, soit $\P_i=\Stab_{\G}(p_i)$.
\end{fait}

\subsection{Groupes paraboliques de rang maximal}

Les sous-groupes paraboliques maximaux qui apparaissent ici sont conjugu\'es \`a des sous-groupes paraboliques d'isom\'etries hyperboliques; c'est un des r\'esultats principaux de \cite{Crampon:2012fk}. Un cas particulier est celui o\`u les sous-groupes paraboliques sont de rang maximal, c'est-\`a-dire que leur action sur $\dO\smallsetminus\{p\}$, o\`u $p$ est le point fixe du groupe en question, est cocompacte. Dans ce cas, on a le r\'esultat suivant.

\begin{theo}[\cite{Crampon:2012fk}, section 7]\label{ellipsoidesecurite}
Soit $\P$ un sous-groupe parabolique de $\Aut(\O)$, de rang maximal et de point fixe $p\in\dO$. Il existe deux ellipsoïdes $\P$-invariants $\E^{int}$ et $\E^{ext}$ tels que
\begin{itemize}
 \item $\partial\E^{int}\cap \partial\E^{ext}=\partial\E^{int}\cap\dO = \partial\E^{ext}\cap\dO=\{p\}$;
 \item $\E^{int}\subset \O\subset \E^{ext}$;
 \item $\E^{int}$ est une horoboule de $\E^{ext}$.
\end{itemize}
\end{theo}

Anticipons un peu. Pour pouvoir dire quelque chose du flot g\'eod\'esique d'une vari\'et\'e g\'eom\'etriquement finie, il va nous falloir ma\^itriser ce qui se passe dans les parties qui partent \`a l'infini, les cusps $\C_i$. De fa\c con g\'en\'erale, cela ne sera pas possible, comme le montre le contre-exemple que nous donnons dans la partie \ref{sectionex1}. Toutefois, lorsque les sous-groupes paraboliques sont de rang maximal, les deux ellipso\"ides du th\'eor\`eme pr\'ec\'edent nous donnent deux m\'etriques hyperboliques dans chaque cusp qui contr\^olent la m\'etrique de Hilbert.

\begin{coro}\label{loincusp}
Soit $M=\Quo$ une vari\'et\'e g\'eom\'etriquement finie. Supposons que les sous-groupes paraboliques maximaux de $\G$ soient tous de rang maximal. Alors, pour toute constante $C>1$, on peut trouver une d\'ecomposition
$$C(M)=K  \bigsqcup\sqcup_{1\leqslant i\leqslant l} \C_i$$
du c\oe ur convexe de $M$, et, sur chaque $\C_i$, deux m\'etriques hyperboliques $\mathtt{h}_i$ et $\mathtt{h}_i'$, telles que
\begin{itemize}
 \item $F$, $\mathtt{h}_i$ et $\mathtt{h}_i'$ ont les mêmes géodésiques, à paramétrisation près ;
 \item $\displaystyle\frac{1}{C} \mathtt{h}_i \leqslant \mathtt{h}_i' \leqslant F \leqslant \mathtt{h}_i \leqslant C\mathtt{h}_i'$.
\end{itemize}
\end{coro}
\begin{proof}
Soit $p\in\LG$ un point parabolique, $\P$ le sous-groupe parabolique maximal de $\G$ fixant $p$. Soient $\E^{int}$ et $\E^{ext}$ deux ellipso\"ides donn\'es par le th\'eor\`eme \ref{ellipsoidesecurite}. Ils d\'efinissent deux m\'etriques hyperboliques $\h$ et $\h'$ telles que $\h'\leqslant F\leqslant \h$.\\

\begin{center}
\begin{figure}[h!]
  \centering
\includegraphics[width=6cm, angle=-2]{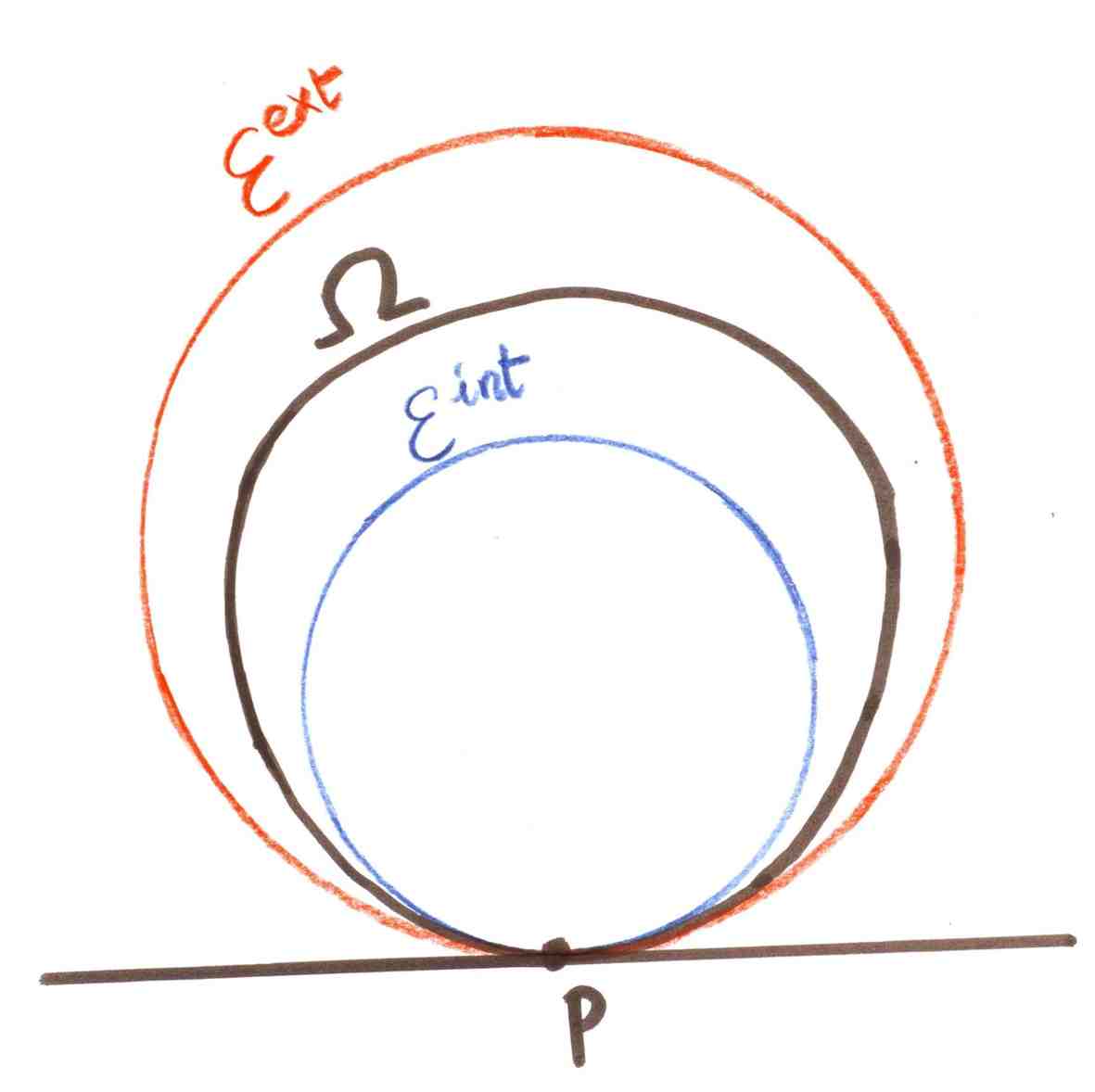}
\caption{Les ellipso\"ides tangents}
\label{secu}
\end{figure}
\end{center}

L'adh\'erence de Zariski $\U$ de $\P$ dans $\ss$ est isomorphe \`a $\R^{n-1}$. Les ensembles $\partial\E^{int}\smallsetminus\{p\}$ et $\partial\E^{ext}\smallsetminus\{p\}$ sont des orbites de $\U$. Par exemple, dans une certaine base de $\R^{n+1}$, $\partial\E^{ext}$ est d'\'equation 
$$x_nx_{n+1}=x_1^2+\cdots+x_{n-1}^2,$$ 
et alors $\partial\E^{int}$ est d'\'equation 
$$-x_n^2 + 2ax_nx_{n+1} = 2a (x_1^2+\cdots+x_{n-1}^2)$$ pour un certain $a>0$.\\
Soient $C>1$ et $x$ un point de $\E^{int}$. Pour $t\geqslant 0$, on note $x(t)$ le point du segment $[xp]$ tel que $\d(x,x(t))=t$. On consid\`ere la fonction
$$f_x: t\longmapsto \max_{v\in T_{x(t)}\O\smallsetminus \{0\}} \frac{\h(x(t),v)}{\h'(x(t),v)}.$$

\begin{lemm} 
Il existe $T=T(C)$ tel que, pour tout $t\geqslant T$, $1\leqslant f_x(t)\leqslant C$.
\end{lemm}
\begin{proof} Pour simplifier les calculs, on peut supposer que les \'equations de $\partial\E^{ext}$ et $\partial\E^{int}$ sont les pr\'ec\'edentes et qu'on travaille dans la carte affine $\{x_{n+1}=1\}$ avec la structure euclidienne induite par celle de $\R^{n+1}$. On peut aussi supposer que le point $x$ a pour coordonn\'ees $(0,\cdots,0,x_n)$ pour un certain $x_n>0$; le point $p$ est ici l'origine et $T_p\partial\E^{ext}=\{t_n=0\}$. La d\'efinition de la m\'etrique de Finsler (formule (\ref{metrique_finsler})) donne imm\'ediatement que, pour $v\in \R\cdot xp\smallsetminus\{0\}$ ou $v\in \{t_n=0\}\smallsetminus\{0\}$,
$$\lim_{t\to +\infty}  \frac{\h(x(t),v)}{\h'(x(t),v)} = 1.$$ 
Il existe donc $T(C)$ tel que pour $t\geqslant T(C)$, $\frac{\h(x(t),v)}{\h'(x(t),v)} \leqslant C$ pour $v\in \R\cdot xp\smallsetminus\{0\}$ ou $v\in \{t_n=0\}\smallsetminus\{0\}$.\\
De plus, pour tout $t\geqslant 0$, les sous-espaces $\R\cdot xp$ et $\{t_n=0\}$ de $T_{x(t)}\O$ sont orthogonaux, tant pour $\h(x(t),\cdot)$ que pour $\h'(x(t),\cdot)$. En d\'ecomposant le vecteur $v\in T_{x(t)}\O\smallsetminus\{0\}$ selon $\R\cdot xp$ et $\{t_n=0\}$, on voit que 
$$\frac{\h(x(t),v)}{\h'(x(t),v)} \leqslant C$$ 
d\`es que $t\geqslant T(C)$.
\end{proof}
Comme les m\'etriques $\h$ et $\h'$ sont invariantes par $\U$, on a aussi $1\leqslant f_{u\cdot x}(t)\leqslant C$, pour tout $u\in\U$ et tout $t\geqslant T$. Il existe donc une horoboule $H^{int}$ de $\E^{int}$ bas\'ee en $p$ telle que, sur $H^{int}$,
$\displaystyle\frac{1}{C} \mathtt{h} \leqslant \mathtt{h}' \leqslant F \leqslant \mathtt{h} \leqslant C\mathtt{h}'$.
Comme $\P$ agit de fa\c con cocompacte sur $\dO\smallsetminus\{p\}$, $H^{int}$ contient une horoboule $H$ de $\O$ bas\'ee en $p$ telle que, sur $H$,
\begin{itemize}
 \item $F$, $\mathtt{h}$ et $\mathtt{h}'$ ont les mêmes géodésiques (les droites), à paramétrisation près ;
 \item $\displaystyle\frac{1}{C} \mathtt{h} \leqslant \mathtt{h}' \leqslant F \leqslant \mathtt{h} \leqslant C\mathtt{h}'$.
\end{itemize}

On peut maintenant conclure. Consid\'erons un ensemble de repr\'esentants $\{p_i\}_{1\leqslant i \leqslant l}$ des points paraboliques de $\LG$. On note $\P_i$ le sous-groupe parabolique maximal de $\G$ qui fixe le point $p_i$. On peut faire la construction pr\'ec\'edente pour chaque point $p_i$. On obtient ainsi une horoboule $H_i$ de $\O$ bas\'ee en $p_i$ et, sur $H_i$, deux m\'etriques hyperboliques $\P_i$-invariantes $\mathtt{h}_i$ et $\mathtt{h}_i'$, v\'erifiant les propri\'et\'es pr\'ec\'edentes. Elles induisent par projection deux m\'etriques hyperboliques sur le cusp $\C_i=H_i/\P_i$, qui satisfont aux conditions de l'\'enonc\'e.\\
On peut supposer que $H_i\subset C(\LG)$, $1\leqslant i \leqslant l$. D'apr\`es le fait \ref{decomposition}, il est aussi possible de prendre les $H_i$ telles que l'union $\cup_{\g\in\G, 1\leqslant i \leqslant l} \g\cdot H_i $ soit disjointe. Le cusp $\C_i=H_i/\P_i$ s'identifie ainsi \`a une partie de $C(M)$. L'ensemble $K=C(M)\smallsetminus \sqcup_i \C_i$ est n\'ecessairement compact et cela donne la d\'ecomposition annonc\'ee.
\end{proof}

\subsection{Cusps asymptotiquement hyperboliques}

Nous allons suivre le chemin indiqu\'e par les groupes paraboliques de rang maximal en nous restreignant \`a ces vari\'et\'es g\'eom\'etriquement finies dont nous savons contr\^oler la m\'etrique de Hilbert dans les cusps:

\begin{defi}\label{asymphypdefi}
On dira qu'une vari\'et\'e $M=\Quo$ g\'eom\'etriquement finie est \emph{\`a cusps asymptotiquement hyperboliques} s'il existe une d\'ecomposition du c\oe ur convexe $C(M)=K\bigsqcup \sqcup_{1\leqslant i\leqslant l} \C_i$ telle que, sur chaque $\C_i$, il existe une m\'etrique hyperbolique $\h_i$, ayant les m\^emes g\'eod\'esiques (non param\'etr\'ees) que $F$ et qui soit \'equivalente \`a $F$, c'est-\`a-dire que, pour un certain $C_i\geqslant 1$,
$$\frac{1}{C_i}\h_i \leqslant F \leqslant C_i\h_i.$$
\end{defi}

Si la condition d'\^etre g\'eom\'etriquement fini porte sur le groupe $\G$, celle d'hyperbolicit\'e asymptotique des cusps porte sur $\O$. Le lemme \ref{loincusp1} qui suit donne une condition sur le bord de $\O$, inspir\'ee par les observations pr\'ec\'edentes, pour que les cusps soient asymptotiquement hyperboliques. Pour l'\'enoncer, il nous faut rappeler une

\begin{defi}\label{defi_calpha}
Soient $\varepsilon>0,\beta> 1$. On dit qu'une fonction $f : U \subset \R^n \longrightarrow \R$, d\'efinie et de classe $\C^1$ sur un ouvert $U$, est
\begin{itemize}
 \item de classe $\C^{1+\varepsilon}$ si, pour une certaine constante $C>0$,
$$|f(x)-f(y)-d_xf(y-x)| \leqslant C|x-y|^{1+\varepsilon},\ x,y\in U;$$
 \item $\beta$-convexe si, pour une certaine constante $C>0$,
$$|f(x)-f(y)-d_xf(y-x)| \geqslant C|x-y|^{\beta},\ x,y\in U.$$
\end{itemize}
On dit que $f$ est de classe $\C^{1+\varepsilon}$ ou $\beta$-convexe en un point $x\in U$ si on a les in\'egalit\'es pr\'ec\'edentes pour tout $y$ dans un voisinage de $x$.
\end{defi}

\begin{lemm}\label{loincusp1}
Soit $M=\Quo$ une vari\'et\'e g\'eom\'etriquement finie. Si le bord $\dO$ est de classe $\C^{1+1}$ et $2$-convexe en chaque point parabolique de $\LG$, alors la vari\'et\'e $M$ est \`a cusps asymptotiquement hyperboliques.
\end{lemm}
\begin{proof}
L'hypoth\`ese de r\'egularit\'e de $\dO$ aux points paraboliques nous permet, pour chaque point parabolique $p$ de $\LG$ de stabilisateur le groupe $\P=\Stab_{\G}(p)$, de trouver deux ellipsoïdes $\E^{int}$ et $\E^{ext}$ tels que
\begin{itemize}
 \item $\partial\E^{int}\cap \partial\E^{ext}=\partial\E^{int}\cap\dO = \partial\E^{ext}\cap\dO=\{p\}$;
 \item $\E^{int}\subset \O\subset \E^{ext}$.
\end{itemize}
Comme les sous-groupes paraboliques de $\G$ sont conjugu\'es \`a des sous-groupes de $\SO$, on peut choisir ces deux ellipso\"ides de telle fa\c con qu'ils soient $\P$-invariants. Soit $H$ une horoboule de $\O$ bas\'ee en $p$, d'adh\'erence incluse dans $\E^{int}$. Les ellipso\"ides $\E^{int}$ et $\E^{ext}$ d\'efinissent sur $H$ deux m\'etriques hyperboliques $\P$-invariantes $\h$ et $\h'$, qui ont les m\^emes g\'eod\'esiques que $F$, et telles que $\h' \leqslant F \leqslant \h$.\\
On peut maintenant voir qu'il existe une constante $C\geqslant 1$ telle que $1/C \leqslant \h/\h' \leqslant C$ sur $H\cap C(\LG)$. Comme $H$ est $\P$-invariante, il suffit de le montrer sur un domaine fondamental $D$ de $P$ sur $H\cap C(\LG)$. Soit $\H$ l'horosphère au bord de $H$. Comme $M$ est g\'eom\'etriquement finie, l'intersection $\mathcal{D}=\overline{D}^{\O} \cap \H$ est compacte. Il suffit donc de voir que pour tout $x\in \mathcal{D}$, la fonction 
$$f_x: t\in[0,+\infty) \longmapsto \max_{v\in T_{x(t)}\O\smallsetminus \{0\}} \frac{\h(x(t),v)}{\h'(x(t),v)}$$
est born\'ee dans $(0,+\infty)$, o\`u $x(t)$ est le point du segment $[xp]$ tel que $\d(x,x(t))=t$. C'est un petit calcul.\\
Comme les m\'etriques $\h$ et $\h'$ sont $\P$-invariantes, elles donnent deux m\'etriques hyperboliques sur le quotient $H\cap C(\LG)/\P$, qui v\'erifient les conditions voulues. On conclut alors comme dans la d\'emonstration du corollaire \ref{loincusp}.
\end{proof}

On doit pouvoir obtenir la m\^eme conclusion que celle du corollaire \ref{loincusp} sous l'hypoth\`ese que le bord $\dO$ est deux fois diff\'erentiable en chaque point parabolique de $\LG$. Toutefois, cette observation plus pr\'ecise ne nous sera pas utile dans ce texte: nous n'utiliserons que le corollaire \ref{loincusp}, dans la partie \ref{sec_entropie}.\\

De fa\c con g\'en\'erale, les sous-groupes paraboliques \'etant conjugu\'es \`a des sous-groupes de $\SO$, on peut se poser la:

\begin{qu}
Soit $\G$ un sous-groupe discret de $\Aut(\O)$ dont l'action est g\'eom\'etriquement finie. Existe t-il un ouvert convexe $\O'$ sur lequel $\G$ agit de fa\c con g\'eom\'etriquement finie \`a cusps asymptotiquement hyperboliques ?
\end{qu}

Beaucoup de r\'esultats dynamiques ne d\'ependent pas du convexe que l'on consid\`ere et le r\'esultat pr\'ec\'edent permettrait de se ramener \`a une situation g\'eom\'etrique et dynamique agr\'eable, qui sera notre propos dans cette article. Par exemple, le spectre des longueurs ne d\'epend que du groupe $\G$, les longueurs des g\'eod\'esiques ferm\'ees \'etant donn\'ees par les valeurs propres des \'el\'ements hyperboliques du groupe.

\subsection{Cas particuliers}

Parmi les vari\'et\'es g\'eom\'etriquement finies, on peut distinguer celles qui ont volume fini, et celles dont le c\oe ur convexe est compact.\\
Dans \cite{Crampon:2012fk}, on a pu voir que les quotients $\Quo$ qui ont volume fini sont pr\'ecis\'ement les vari\'et\'es g\'eom\'etriquement finies dont l'ensemble limite est le bord $\dO$ tout entier. En particulier, si $\Quo$ est une vari\'et\'e de volume fini, les sous-groupes paraboliques maximaux de $\G$ sont de rang maximal. Remarquons que dans tous les cas, un cusp $\C$ d'une vari\'et\'e g\'eom\'etriquement finie a un volume fini (voir la partie 8 de \cite{Crampon:2012fk}). \\
Les vari\'et\'es \emph{convexes cocompactes} sont celles dont le c\oe ur convexe est compact; autrement dit, le quotient $\Quo$ est géométriquement finie et le groupe $\G$ ne contient pas d'\'el\'ements paraboliques.

\section{Hyperbolicité uniforme du flot géodésique}

Rappelons d'abord quelques définitions.

\begin{defi}
Soit $W$ une variété munie d'une métrique de Finsler $\|\cdot\|$ continue. Soient $\ph^t : W \longrightarrow W$ un flot de classe $\C^1$ engendré par le champ de vecteurs $X$ sur $W$, et $V$ une partie $\ph^t$-invariante de $W$. On dit que le flot $\ph^t$ est \emph{uniformément hyperbolique} sur $V$ s'il existe une décomposition $\ph^t$-invariante
$$TW = \R.X \oplus E^s \oplus E^u$$
du fibré tangent à $W$ en tout point de $V$, et des constantes $a,C >0$ pour lesquelles
$$\|d\ph^t Z^s\| \leqslant C e^{-at} \|Z^s\|,\ \|d\ph^{-t} Z^u\| \leqslant C e^{-at}  \|Z^u\|,\ Z^s\in E^s,\ Z^u\in E^u,\ t\geqslant 0.$$
\end{defi}

Dans le cas où $W$ est une variété compacte et $V=W$, on parle plus souvent de \emph{flot d'Anosov}. Les distributions $E^s$ et $E^u$ s'appellent les distributions stables et instables du flot. Le but de cette partie est de montrer une telle propriété d'hyperbolicité pour notre flot géodésique, restreint à son ensemble non errant.\\
Dans le cas où la variété $M$ est compacte, l'ensemble non errant est $HM$ tout entier, et Yves Benoist a déjà prouvé que le flot géodésique était d'Anosov. Si la variété $M$ est convexe-cocompacte, c'est-\`a-dire que son c\oe ur convexe est compact, l'ensemble non errant est lui-m\^eme compact, et une démonstration similaire fonctionnerait pour prouver l'uniforme hyperbolicité sur l'ensemble non errant. Nous étendons ce résultat au cas d'une variété géométriquement finie \`a cusps asymptotiquement hyperboliques:

\begin{theo}\label{anosov}
Soit $M=\Quo$ une variété géométriquement finie \`a cusps asymptotiquement hyperboliques. Le flot géodésique est uniformément hyperbolique sur l'ensemble non errant, de décomposition
$$THM = \R.X \oplus E^s \oplus E^u.$$
\end{theo}

Nous montrerons le théorème en plusieurs temps. Fixons une fois pour toutes une d\'ecomposition du c\oe ur convexe $C(M)$ de $M$ en une partie compacte $K$ et une union finie de cusps $\C_i, 1\leqslant i\leqslant l$, chacun d'entre eux portant une métrique hyperbolique $\mathtt{h}_i$ telle que
\begin{itemize}
 \item $F$ et $\mathtt{h}_i$ ont les mêmes géodésiques, à paramétrisation près ;
 \item $\displaystyle\frac{1}{C} \mathtt{h}_i \leqslant F \leqslant C \mathtt{h}_i$, pour une certaine constante $C\geqslant 1$.
\end{itemize}

Pour la partie compacte, on se servira du lemme suivant:
\begin{lemm}
Soient $V$ une partie compacte de $HM$ et $T> 0$. Il existe un réel $0<b(V,T)<1$ tel que, si $\ph^t(w)\in V$ pour $0\leqslant t\leqslant T$, alors, pour tout $Z\in E^s(w)$, $$\|d\ph^T(Z)\| \leqslant b(V,T) \|Z\|.$$
\end{lemm}
\begin{proof}
C'est une simple conséquence du lemme \ref{decroissance}. Notons $V_T$ l'ensemble des $w\in V$ tels que $\ph^t(w)\in V$ pour $0\leqslant t\leqslant T$ et $E_1 = \{Z \in E^s(w)\ |\ w\in V_T, \|Z\|=1\}$. Les ensembles $V_T$ et $E_1$ sont compacts. La fonction $Z\in E_1 \longmapsto \|d\ph^T(Z)\|$ est continue et atteint donc son maximum pour un certain vecteur $Z_M$. Le lemme \ref{decroissance} nous dit que $\|d\ph^T(Z_M)\| <1$, d'où le résultat.
\end{proof}

Pour les cusps, c'est un peu plus délicat. Choisissons un des cusps $\C_i$, et oublions les indices: on note $\C$ le cusp et $\mathtt{h}$ la métrique hyperbolique sur $\C$.
\begin{lemm}\label{decroissancecusp}
Pour tout $0<a<1$, on peut trouver un temps $T_a=T_a(C)>0$ tel que, pour tout $w\in H\C$ tel que $\ph^t(w)\in H\C$ pour $0\leqslant t\leqslant T_a$ et $Z \in E^s(w)$, on ait $$\|d\ph^{T_a} Z\| \leqslant a \|Z\|.$$
\end{lemm}

Avant de montrer ce dernier lemme, voyons d'abord comment en déduire une
\begin{proof}[Démonstration du théorème \ref{anosov}]
Rappelons la d\'ecomposition du c\oe ur convexe en
$$C(M)= K\bigsqcup \sqcup_i \C_i.$$
Choisissons un réel $0<a<1$ et un temps $T_a>0$ comme dans le lemme \ref{decroissancecusp}, et posons
$$K_a= \bigcup_{-T_a\leqslant t \leqslant T_a} \ph^t \big( HM|_K\big).$$
Pour tout point $w$ de l'ensemble non errant $\NW$, le morceau d'orbite $\{\ph^t(w)\}_{0\leqslant t\leqslant T_a}$ est inclus soit dans $K_a$, soit dans un des $H\C_i$. Les deux lemmes précédents impliquent alors que, pour tout $Z\in E^s|_{\NW}$, on a
$$\|d\ph^{T_a} Z\| \leqslant A \|Z\|,$$ avec $A = \max(a,b(K_a,T_a))<1$.
Ainsi, pour tout $t\geqslant 0$, en posant $N=\left[\displaystyle\frac{t}{T_a}\right]$, on a
$$\|d\ph^t(Z)\| \leqslant A^N \|d\ph^{t-NT_a}(Z)\| \leqslant \frac{\|d\ph^{t-NT_a}(Z)\|}{e^{\frac{t-NT_a}{T_a}\ln A}} e^{\frac{t}{T_a} \ln A} \leqslant A^{-1} e^{-\frac{\ln A^{-1}}{T_a} t}  \|Z\|.$$
Cela prouve la décroissance uniformément hyperbolique sur la distribution stable. On fait de même pour la distribution instable en considérant $\ph^{-t}$.
\end{proof}

Le reste de cette partie est consacrée à la démonstration du lemme \ref{decroissancecusp}. Bien entendu, l'idée est de comparer les flots géodésiques des métriques $F$ et $\mathtt{h}$ sur $H\C$, qui satisfont $C^{-1} \mathtt{h} \leqslant F \leqslant C\mathtt{h}$ pour une certaine constante $C>1$. Comme $F$ et $\h$ ont les mêmes géodésiques à paramétrisation près, le flot $\ph^t$ est en effet une renormalisation du flot $\ph^t_{\h}$ de la métrique $\h$: on a
$$\ph^t(w) = \ph^{\alpha(w,t)}_{\h}(w)$$ pour un certain $\alpha(w,t)\in\R$. Bien sûr, cette expression ne fait sens que si $\ph^s(w)$ est dans $H\C$ pour tout $0\leqslant s \leqslant t$. La fonction $\alpha$ est donc définie sur l'ensemble $$W=\{(w,t)\ |\ \ph^s(w)\in H\C,\ 0\leqslant s \leqslant t\} \subset H\C \times \R.$$
Soit $g$ la fonction définie sur $H\C$ par $F=g^{-1}\h$. C'est une fonction de classe $\C^1$, qui prend ses valeurs dans l'intervalle $[\frac{1}{C},C]$. Si $X_{\h}$ est le générateur du flot géodésique de $\h$, alors on a $X=g X_{\h}$. On retrouve la fonction $\alpha$ en intégrant $g$:
$$\alpha(w,t) = \int_{0}^{t} g(\ph^s(w))\ ds;$$
la fonction $\alpha$ est donc de classe $\C^1$ et satisfait $$\frac{1}{C}t \leqslant \alpha(w,t)\leqslant Ct,\ t\geqslant 0.$$
L'espace tangent à $H\C$ se décompose de deux façons, selon que l'on considère le flot de $F$ ou de $\h$:
$$TH\C = E^s \oplus E^u \oplus \R.X = E^s_{\h} \oplus E^u_{\h} \oplus \R.X.$$
Sur $H\C$, on dispose des métriques $\|.\|$ et $\|.\|_{\h}$ associées respectivement à $F$ et $\h$ et définies par la formule (\ref{normhm}) via les décompositions précédentes. La métrique $\|.\|_{\h}$ est bien entendu une métrique riemannienne, qui n'est rien d'autre que la métrique de Sasaki, et pour laquelle la décomposition $TH\C = E^s_{\h} \oplus E^u_{\h} \oplus \R.X$ est orthogonale. Rappelons que, si $Z^s_{\h}\in E^s_{\h}$, alors
$$\|d\ph^t_{\h}(Z^s_{\h})\| = e^{-t} \|Z^s_{\h}\|,\ t\in\R,$$
sous réserve, bien sûr, que $Z$ soit tangent à $H\C$ en un point $w$ tel que $\ph_{\h}^s(w)\in H\C$ pour $0\leqslant s \leqslant t$.\\

Le lemme essentiel est le suivant:
\begin{lemm}\label{angle}
La distribution $E^s$ est, sur $H\C$, incluse dans un cône d'angle $\theta$ pour $\|.\|_{\h}$ autour de la distribution $E^s_{\h}$, avec $0 \leqslant \theta = \theta(C) < \frac{\pi}{2}$.
\end{lemm}
\begin{proof}
Il revient au m\^eme de montrer qu'il existe $\theta$ tel que, pour tout $w=(x,[\xi])$ dans $H\C$, la projection $d\pi(E^s(w))$ est dans un cône d'angle $\theta$ pour $\h$ autour de $d\pi(E^s_{\h}(w))$.\\
De la démonstration de la proposition 3.6 de \cite{MR2587084}, on peut tirer que la projection $d\pi(E^s(w))$ co\"incide avec l'espace tangent en $\xi\in T_x\C$ \`a la boule unit\'e tangente de la norme $F(x,.)$; où $\xi$ est le vecteur de norme $1$ de $[\xi]$. La m\^eme chose est bien s\^ur valable pour $d\pi(E^s_{\h}(w))$ et la m\'etrique $\h$.
\begin{center}
\begin{figure}[h!]
  \centering
\includegraphics[width=7cm]{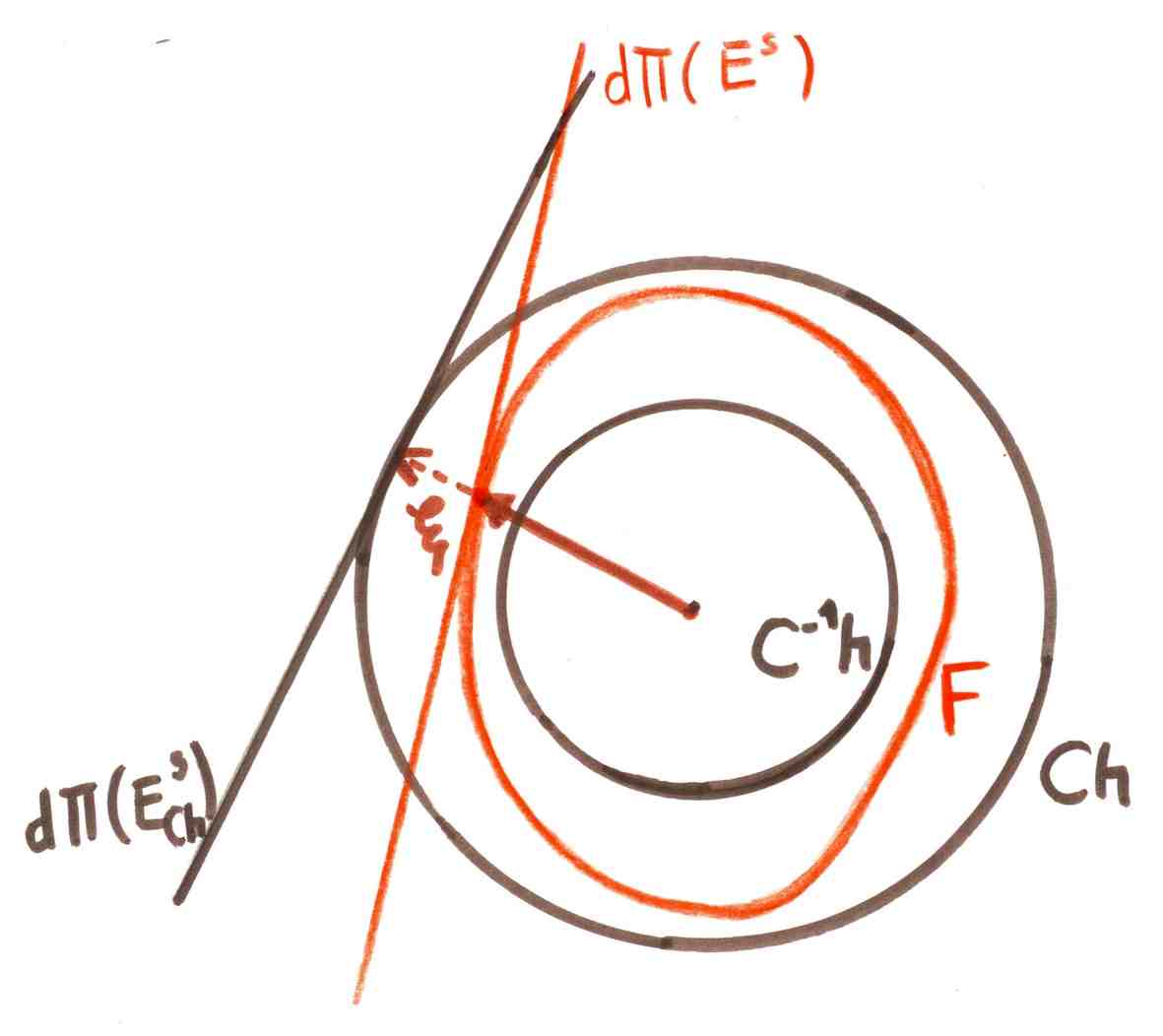}
\caption{La boule unit\'e de $F$ est coinc\'ee entre celles de $C^{-1}\h$ et $C\h$ }
\label{angle_fig}
\end{figure}
\end{center}
Or, la boule unit\'e tangente de $F$ est coinc\'ee entre les boules de rayon $\frac{1}{C}$ et $C$ de $\h$. Cet encadrement et le fait pr\'ec\'edent impliquent l'existence de $\theta$.	
\end{proof}

Tout vecteur $Z^s\in E^s$ se décompose en
$$Z^s=Z_{\h}^s + Z^u_{\h} + Z^{X}_{\h} \in E^s_{\h} \oplus E^u_{\h} \oplus \R.X_{\h}.$$
On déduit du lemme précédent que l'angle entre $Z^s$ et sa projection $Z_{\h}^s$ sur $E^s_{\h}$ est toujours inférieur à $\theta$; de même en ce qui concerne l'angle, pour $\h$, entre $d\pi Z^s$ et $d\pi Z^s_{\h}$. D'où le
\begin{coro}\label{maj}
Pour tout vecteur $Z^s\in E^s$, on a
$$\frac{1}{C} \| Z_{\h}^s\|_{\h} \leqslant \|Z^s\| \leqslant \frac{C}{\cos \theta} \| Z_{\h}^s\|_{\h} $$
\end{coro}
\begin{proof}
On a
$$\|Z^s\| = F(d\pi Z^s) \leqslant C\h(d\pi Z^s) \leqslant \frac{C}{\cos \theta} \h(d\pi Z_{\h}^s) = \frac{C}{\cos \theta} \| Z_{\h}^s\|_{\h},$$
et
$$\| Z_{\h}^s\|_{\h} = \h(d\pi Z_{\h}^s) \leqslant \h(d\pi Z^s) \leqslant C F (d\pi Z^s) = C \| Z^s\|.$$
\end{proof}

On peut maintenant terminer la
\begin{proof}[Démonstration du lemme \ref{decroissancecusp}]
Fixons $(w,t)\in W$ et un vecteur stable $Z^s \in E^s(w)$, qui se décompose en
$$Z^s=Z_{\h}^s + Z^u_{\h} + Z^{X}_{\h} \in E^s_{\h} \oplus E^u_{\h} \oplus \R.X_{\h}.$$
On a donc
\begin{equation}\label{truc1}
 d\ph^t(Z^s) = d\ph^t(Z_{\h}^s) + d\ph^t(Z^u_{\h}) + d\ph^t(Z^{X}_{\h}).
\end{equation}

D'autre part, considérons les fonctions $\ph$ et $\ph_{\h}$ définies par
$$\ph(w,t) = \ph^t(w),\ \ph_{\h}(w,t) = \ph_{\h}^t(w);$$ les fonctions $\ph$ et $\ph_{\h}$ sont définies, respectivement, sur $W$ et sur l'ensemble $$\{(w,t)\ |\ \ph_{\h}^s(w)\in H\C,\ 0\leqslant s \leqslant t\}.$$
On a ainsi $\ph(w,t) = \ph_{\h}(w,\alpha(w,t))$, d'où
\begin{equation}\label{truc2}
 d\ph^t(Z^s) = \frac{\partial\ph}{\partial w}(w,t)(Z^s) = \frac{\partial \ph_{\h}}{\partial t}\frac{\partial \alpha}{\partial w}(w,t)(Z^s) + d\ph^{\alpha(w,t)}_{\h}(Z^s).
\end{equation}

L'application $\displaystyle\frac{\partial \ph_{\h}}{\partial t}$ a son image dans $\R.X$; le premier terme $ \frac{\partial \ph_{\h}}{\partial t}\frac{\partial \alpha}{\partial w}(w,t)(Z^s)$ de la dernière expression est donc un vecteur de $\R.X$. Comme $d\ph_{\h}^t$ préserve la décomposition  $TH\C=E^s_{\h} \oplus E^u_{\h} \oplus \R.X_{\h}$, on déduit de (\ref{truc1}) et (\ref{truc2}) que
$$ (d\ph^t(Z^s))^s_{\h} =  d\ph^{\alpha(w,t)}_{\h}(Z_{\h}^s).$$

On a alors, d'après le corollaire \ref{maj},
$$\|d\ph^t Z^s\| \leqslant \frac{C}{\cos \theta} \|d\ph^{\alpha(w,t)}_{\h}(Z_{\h}^s)\|_{\h}
=  \frac{C}{\cos \theta} e^{-\alpha(w,t)} \|Z_{\h}^s\|_{\h} \leqslant \frac{C^2}{\cos \theta} e^{-t} \|Z^s\|.$$

Mais on peut écrire
$$\frac{C^2}{\cos \theta} e^{-t} = e^{-t\left(1-\frac{1}{t}\ln \frac{C^2}{\cos \theta}\right)}.$$
Aussi, en prenant $T_a = \ln \frac{C^2}{a\cos \theta}$, on obtient
$$\|d\ph^{T_a} Z^s\| \leqslant a \|Z^s\|.$$
\end{proof}

\begin{rema}\label{deccusp}
Dans le lemme \ref{angle}, on pourrait voir que plus $C$ est proche 1, plus $\theta(C)$ peut être pris proche de $0$. Or, dans le cusp $\C$ choisi pour la constante $C$, on a
$$\|d\ph^t Z^s \| \leqslant  \frac{C^2}{\cos \theta(C)} e^{-t}\|Z^s\|,$$
pour tout vecteur stable $Z^s$ tangent en un point $w\in H\C$ tel que $\ph^s(w)\in H\C,\ 0\leqslant s\leqslant t$. En particulier, sous l'hypoth\`ese de rang maximal des sous-groupes paraboliques maximaux, on peut, d'apr\`es le corollaire \ref{loincusp}, choisir le cusp de telle fa\c con que la constante $\displaystyle\frac{C^2}{\cos \theta(C)}$ soit aussi proche de $1$ qu'on le souhaite.
\end{rema}

\section{Propri\'et\'es de r\'ecurrence}

\subsection{Transitivité et mélange topologique}

Le but de cette partie est d'étudier les propriétés de récurrence du flot géodésique d'une variété quotient $M=\Quo$ \emph{quelconque}. Rappelons qu'un flot $\ph^t$ sur un espace topologique $X$ est dit:
\begin{itemize}
 \item \emph{topologiquement transitif} s'il existe une orbite dense ou, de fa\c con \'equivalente, si pour tous ouverts $U$ et $V$ de $X$, il existe $T\in\R$ tel que $\ph^T(U)\cap V \not= \emptyset$;
 \item \emph{topologiquement m\'elangeant} si pour tous ouverts $U$ et $V$ de $X$, il existe $T\in\R$ tel que, pour tout $t\geqslant T$, $\ph^t(U)\cap V \not= \emptyset$.
\end{itemize}
Le r\'esultat principal est le suivant.

\begin{prop}\label{melange}
Soit $M=\Quo$, avec $\G$ non \'el\'ementaire. Le flot géodésique est topologiquement m\'elangeant sur son ensemble non errant.
\end{prop}


Une orbite périodique du flot géodésique sur $HM$ se projette sur une géodésique fermée de $M$, parcourue dans un sens ou dans l'autre. Or, les géodésiques fermées orientées sont en bijection avec les classes de conjugaison d'\'el\'ements hyperboliques de $\G$. La géodésique fermée orientée définie par un tel $\g\in\G$ est précisément la projection sur $M$ de l'axe orienté $(x_{\g}^- x_{\g}^+)$.\\
\`A un élément $\g\in\G$ correspond ainsi une unique orbite périodique. Bien sûr, l'orbite périodique associée à $\g^{-1}$ se projette sur la même géodésique fermée que celle associée à $\g$, mais l'orientation est inversée. Le lemme suivant est immédiat:

\begin{lemm}\label{lemme_gromov_sullivan}
Soient $g$ et $h$ deux éléments hyperboliques de $\Aut(\O)$ tel que l'intersection des groupes engendrés par $g$ et $h$ est triviale. Posons $k_n = g^n h^n$, alors $\underset{n\to+\infty}{\lim} x_{k_n}^+ = x_{g}^+$ et  $\underset{n\to+\infty}{\lim} x_{k_n}^- = x_{h}^-$.
\end{lemm}

\begin{coro}
Soit $\G$ un sous-groupe de $\Aut(\O)$. L'ensemble
$\{(x_g^+,x_g^-)\ | g\in\G\}$
est dense dans $\LG\times\LG$. En particulier, les orbites périodiques de $HM$ sont denses dans $\NW$.
\end{coro}
\begin{proof}
Il suffit de se rappeler que l'action de $\G$ sur $\LG$ est minimale, puisque $\LG$ est le plus petit ferm\'e $\G$-invariant de $\dO$; en particulier, l'ensemble des $x_g^+$ pour $g\in\G$ est dense dans $\LG$. Fixons une métrique riemannienne quelconque sur $\dO$. Si on prend un couple $(x,y)$ dans $\LG\times\LG$, il existe, pour tout $\varepsilon>0$, des éléments $g$ et $h$ hyperboliques de $\G$ tels que $x_g^+$ et $x_h^-$ sont $\varepsilon$-proches de, respectivement, $x$ et $y$. Le lemme précédent affirme alors que pour $n$ assez grand, si $k_n = g^n h^n \in \G$, les points $x_{k_n}^+$ et $x_{k_n}^-$ sont $2\varepsilon$-proches de, respectivement, $x$ et $y$.
\end{proof}

On peut maintenant donner une

\begin{proof}[Démonstration de la proposition \ref{melange}]
Prenons $U$ et $V$ deux ouverts de $\NW$. Les orbites périodiques étant denses dans $\NW$, il existe une orbite périodique passant dans $U$, et une autre, distincte de la première, passant dans $V$. Considérons des relevés $(xy)$ et $(x'y')$ dans $H\O$ de ces orbites. Ce sont les axes d'élements $\g$ et $\g'$ distincts de $\G$. Le projeté de l'orbite $(xy')$ sur $HM$ est alors une orbite qui rencontre à la fois $U$ et $V$. Ainsi, il existe $t\geqslant 0$ tel que $\ph^t(U)\cap V \not = \emptyset$ et le flot est topologiquement transitif.\\

Comme le flot est topologiquement transitif, le mélange topologique est équivalent au fait que le spectre des longueurs des orbites périodiques engendre un sous-groupe dense de $\R$ (exercice 18.3.4 du livre \cite{MR1326374}). Or, la longueur de l'orbite périodique définie par l'élément hyperbolique $\g\in\G$ est exactement $\frac{1}{2}\ln \frac{\l_0(\g)}{\l_n(\g)}$, où $\l_0(\g)$ et $\l_n(\g)$ sont le module de, respectivement, sa plus grande et plus petite valeur propre. Si le groupe engendré par les longueurs n'était pas dense dans $\R$, il existerait $l>0$ tel que pour tout $\g\in\G$, il existe $k_{\g}\in\N$ tel que
\begin{equation}\label{relation}
\frac{1}{2}\ln \frac{\l_0(\g)}{\l_n(\g)} = k_{\g} l.
\end{equation}

Quitte \`a se restreindre au sous-espace projectif engendr\'e par $\LG$, on peut supposer que l'action de $\G$ est irr\'eductible. D'apr\'es le lemme \ref{irreductible}, elle est m\^eme fortement irr\'eductible. 
La proposition \ref{adh_semisimple} nous dit que l'adhérence de Zariski $G$ de $\G$ est alors un groupe semi-simple. Le théorème \ref{cone_limite} qui suit implique que la relation (\ref{relation}) ne peut \^etre v\'erifi\'ee pour tout $\g\in\G$.
\end{proof}

Le théorème permettant de conclure la démonstration est dû à Yves Benoist dans \cite{MR1770716}. Rappelons-en ici un énoncé dans notre contexte particulier.\\
Soit $G$ un sous-groupe de Lie semi-simple de $\ss$. \`A tout \'el\'ement $g$ de $G$, on associe le vecteur $\ln (g) = (\ln \l_0(g),\cdots,\ln \l_n(g)) \in \R^{n+1}$, o\`u $\l_0(g) \geqslant \l_1(g) \geqslant \cdots \geqslant \l_n(g)$ d\'esignent les modules des valeurs propres de $g$. Pour un sous-groupe $\G$ de $G$, on note $\ln \G$ l'ensemble des $\ln \g$ pour $\g\in\G$. Le résultat est le suivant:


\begin{theo}[Yves Benoist, \cite{MR1770716}]\label{cone_limite}
Soient $G$ un sous-groupe de Lie semi-simple de $\ss$ et $\G$ un sous-groupe de $G$. Si $\G$ est Zariski-dense dans $G$, alors le sous-groupe engendré par $\ln \G$ est dense dans le sous-espace vectoriel de $\R^{n+1}$ engendr\'e par $\ln G$.
\end{theo}

L'autre r\'esultat que l'on a utilis\'e \'etait la proposition suivante, due à Yves Benoist. Pour faciliter la lecture de ce texte, nous en donnons une d\'emonstration sous l'hypothèse $\O$ strictement convexe, qui n'est toutefois pas nécessaire.

\begin{prop}[Benoist, remarque suivant le corollaire 3.2 de \cite{MR1767272}]\label{adh_semisimple}
Soit $\G$ un sous-groupe irréductible de $\ss$ qui préserve un ouvert $\O\subset \PP^n$ proprement convexe et strictement convexe. La composante connexe $G$ de l'adhérence de Zariski de $\G$ est un groupe de Lie semi-simple.
\end{prop}

\begin{proof}
\par{
D'après le lemme \ref{irreductible}, $\G$ est fortement irréductible. Quitte à consid\'erer un sous-groupe d'indice fini, on peut donc supposer que $\G$ est Zariski-connexe et ainsi que $G$ est d'indice fini dans l'adhérence de Zariski de $\G$. Le groupe $G$ est Zariski-fermé à indice fini près et agit de façon irréductible sur $\R^{n+1}$. Il est donc réductif\footnote{c'est-à-dire que tout sous-groupe unipotent distingué est réduit à l'identité.}.
}
\par{
En effet, soit $N$ un sous-groupe unipotent et distingué de $G$. Considérons le sous-espace vectoriel $E$ de $\R^{n+1}$ des points fixes de $N$. Comme $N$ est distingué dans $G$, l'espace vectoriel $E$ est préservé par $G$. Comme l'action de $G$ sur $\R^{n+1}$ est irréductible, on en d\'eduit que $E$ est trivial. Or, le théorème de Kolchin affirme que tout groupe unipotent fixe un vecteur non trivial. On a donc n\'ecessairement $E=\R^{n+1}$ et $N=\{1\}$. Ainsi, $G$ est bien réductif.
}
\par{
Par conséquent, pour montrer que $G$ est semi-simple, il suffit de montrer que son centre est trivial. Soit $a$ un élément du centre de $G$. Comme $\G$ est fortement irréductible, d'apr\`es le lemme \ref{irreductible}, il possède au moins un élément hyperbolique $\g$. Notons $\rho^+$ son rayon spectral. L'espace propre $\ker(\g-\rho^+)$ est donc de dimension 1 (voir le paragraphe \ref{para_isometrie}) et préservé par $a$. Par suite, $a$ possède  une valeur propre $\lambda$ réelle. L'espace propre $\ker(a-\lambda)$ est préservé par $\G$ et doit donc être trivial. On en d\'eduit que $a$ est une homothétie. Comme $a\in\G \subset \ss$, on conclut que $a=1$ ou $a=-1$.
}
\end{proof}

\subsection{Lemme de fermeture et conséquences}

Nous rappelons ici un résultat classique pour les flots d'Anosov et une de ses conséquences, qui nous servira dans la partie suivante. On trouve dans un article d'Yves Coudene et Barbara Schapira \cite{MR2735038} une démonstration de ces deux r\'esultats dans le cadre de la courbure n\'egative (ou nulle), qui s'adapte sans changement aucun. On pourra aussi consulter \cite{MR1441541} en ce qui concerne le lemme de fermeture:

\begin{lemm}[Appendice de \cite{MR2735038}, Proposition 4.5.15 de \cite{MR1441541}]
Soient $M=\Quo$ une variété quotient, avec $\G$ non élémentaire, et $K$ une partie compacte de $HM$. Fixons $\varepsilon>0$.\\
Il existe $\delta>0$ et $T>0$ tels que, si $w\in K$ satisfait $d(w,\ph^t(w))<\delta$ pour un certain $t>T$, alors il existe un point $w'\in HM$ tel que
\begin{itemize}
 \item  $w'$ est p\'eriodique de p\'eriode $t'\in(t-\varepsilon,t+\varepsilon)$;
 \item  pour tout $0<s<\min\{t,t'\}$, $d(\ph^s(w),\ph^s(w'))<\varepsilon$.
\end{itemize}
\end{lemm}

En version courte, cela signifie que si un point revient assez proche de sa position d'origine après un temps $t$, alors il existe une orbite périodique qui suit son orbite pendant le temps $t$, et ce en restant aussi proche qu'on le veut.\\

Prenons une variété quotient $M=\O/\G$, avec $\G$ non élémentaire. Notons $\M$ l'ensemble des mesures de probabilité sur $\NW$ invariantes par le flot, qu'on munit de la convergence étroite des mesures: une suite $(\eta_n)$ de  $\M$ converge vers $\eta$ si, pour toute fonction $f$ continue sur $\NW$, $\int f\ d\eta_n$ converge vers $\int f\ d\eta$. L'ensemble $\M$ est un convexe dont les points extrémaux sont les mesures ergodiques. Parmi les mesures ergodiques, on peut distinguer le sous-ensemble $\M_{Per}$ constitué des mesures de Lebesgue portées par les orbites périodiques. Coudene et Schapira ont remarqué que le lemme de fermeture impliquait la

\begin{prop}[Corollaire 2.3 de \cite{MR2735038}]\label{densitemesures}
Soit $M=\Quo$ une variété quotient, avec $\G$ non élémentaire. L'enveloppe convexe de $\M_{Per}$ est dense dans $\M$.
\end{prop}

\section{Régularité du bord}\label{sectionregularite}

\begin{defi}
On dira qu'un point $w=(x,[\xi])\in H\O$, ou le rayon géodésique $\{\pi\ph^t(w)\}_{t\geqslant 0}$ qu'il définit, est \emph{hyperbolique} si, pour tout vecteur stable $Z^s \in E^s(w)$, on a
$$\limsup_{t\to +\infty} \frac{1}{t} \ln \|d\ph^t Z^s\| < 0.$$
\end{defi}
Si $w$ est un point hyperbolique, il existe alors $\chi>0$ tel que, pour tout vecteur stable $Z^s \in E^s(w)$,
$$\limsup_{t\to +\infty} \frac{1}{t} \ln \|d\ph^t Z^s\| \leqslant -\chi.$$
Pour tout $w\in H\O$, on notera $\chi(w)\geqslant 0$ le plus grand des réels $\chi$ qui vérifie l'inégalité précédente pour tout vecteur $Z^s \in E^s(w)$; autrement dit,
$$\chi(w) = \inf_{Z^s \in E^s(w)} \liminf_{t\to +\infty} -\frac{1}{t} \ln \|d\ph^t Z^s\|.$$
Le point $w$ est donc hyperbolique si et seulement si $\chi(w)>0$.\\

En fait, on peut facilement caractériser les points de $H\O$ qui sont hyperboliques, et même déterminer $\chi(w)$, selon la régularité du bord $\dO$ au point extrémal $x^+$ du rayon défini par $w$.


\begin{prop}\label{regularite}
Un point $w\in H\O$ est hyperbolique, de coefficient $\chi(w)>0$, si et seulement si $\dO$ est de classe $\C^{1+\varepsilon}$ en $x^+$, pour tout $0<\varepsilon<\displaystyle\left(\frac{2}{\chi(w)} -1\right)^{-1}$.
\end{prop}
\begin{proof}
Reprenons les notations du lemme \ref{decroissance}: on a choisi une carte adaptée au point $w\in H\O$, $Z$ est un vecteur stable tangent à $H\O$ au point $w$, et $z =  d\pi(Z)$, $z_t=(d\pi d\ph^t Z)$. Tout se passe dans un plan et on peut donc supposer qu'on est en dimension $2$. On a vu durant la démonstration du lemme \ref{decroissance} que
$$\|d\ph^t(Z)\|=F(z_t) = \frac{|z|}{2|xx^+|}\left(\frac{|x_tx^+|}{|x_ty_t^+|} + \frac{|x_tx^+|}{|x_ty_t^-|}\right).$$


Ainsi, le rayon géodésique défini par $w$ est hyperbolique si et seulement si
$$\displaystyle\liminf_{t\to +\infty} \frac{1}{t} \ln \left(\frac{|x_tx^+|}{|x_ty_t^+|} + \frac{|x_tx^+|}{|x_ty_t^-|}\right) < 0;$$
autrement dit si et seulement s'il existe $\chi,C>0$ tels que pour $t\geqslant 0$,
$$\frac{|x_tx^+|}{|x_ty_t^+|} + \frac{|x_tx^+|}{|x_ty_t^-|} < C e^{-\chi t},$$
soit
$$\frac{|x_tx^+|}{|x_ty_t^+|}  < C e^{-\chi t} \ \textrm{et}\ \frac{|x_tx^+|}{|x_ty_t^-|} < C e^{-\chi t}.$$

Mais de l'égalité $d_{\O}(x,x_t) = t = \frac{1}{2} \ln [x^+ : x^- : x : x_t]$, on tire
$$|x_tx^+| = e^{-2t} \frac{|x_tx^-|}{|xx^-|}|xx^+|,$$
et donc il existe une constante $C_0\geqslant 1$ pour laquelle
$$\frac{1}{C_0}e^{-2t} \leqslant |x_tx^+| \leqslant C_0e^{-2t}.$$

Ainsi, le rayon géodésique défini par $w$ est hyperbolique si et seulement s'il existe $\chi,D>0$ tels que pour $t\geqslant 0$,
$$\frac{|x_tx^+|}{|x_ty_t^+|}  < D |x_tx^+|^{\frac{\chi}{2}} \ \textrm{et}\ \frac{|x_tx^+|}{|x_ty_t^-|} < D|x_tx^+|^{\frac{\chi}{2}},$$
soit
$$|x_tx^+|^{1-\frac{\chi}{2}} < D|x_ty_t^+| \ \textrm{et}\ |x_tx^+|^{1-\frac{\chi}{2}} < D|x_ty_t^-|.$$
Appellons $f : T_{x^+}\dO \longmapsto \R$ le graphe de $\dO$, de telle façon que
$$|x_tx^+| = f(|x_ty_t^+|) = f(-|x_ty_t^-|).$$
La condition précédente est alors équivalente à : pour tout $s$ (assez petit),
$$f(s) < Ds^{\frac{1}{1-\chi/2}}.$$
Or, dans notre carte adapt\'ee, on a $d_{x^+}f=0$. L'in\'egalit\'e pr\'ec\'edente veut donc dire que la fonction $f$ est $\C^{1+\varepsilon}$ en $x^+$ avec $\varepsilon=\displaystyle\frac{1}{\frac{2}{\chi} -1}$.


\end{proof}

Cette proposition et le théorème \ref{anosov} entraînent le

\begin{coro}\label{regu_geo_fini}
Supposons que $\O$ admette une action géométriquement finie \`a cusps asymptotiquement hyperboliques d'un sous-groupe discret $\G$ de $\Aut(\O)$. Alors il existe $\varepsilon>0$ tel que le bord $\dO$ de $\O$ soit de classe $\C^{1+\varepsilon}$ en tout point de $\LG$.\\
En particulier, si $\O$ admet une action de covolume fini, alors le bord $\dO$ de $\O$ est de classe $\C^{1+\varepsilon}$  pour un certain $\varepsilon>0$.
\end{coro}

\subsection{Régularité optimale du bord}

Soit $\G$ un sous-groupe discret de $\Aut(\O)$. Notons
\begin{equation}\label{elg}
\varepsilon(\LG) = \sup \{\varepsilon\in [0,1],\ \textrm{le bord}\ \dO\ \textrm{est}\ \C^{1+\varepsilon}\ \textrm{en tout point de}\ \LG\}.
\end{equation}
Pour tout élément hyperbolique $\g \in \G$, notons
$$\varepsilon(\g) = \sup \{\varepsilon\in [0,1],\ \textrm{le bord}\ \dO\ \textrm{est}\ \C^{1+\varepsilon}\ \textrm{en}\ x_{\g}^+\},$$
et $\varepsilon(\G) = \inf\{ \varepsilon(\g),\ \g\in\G\ \textrm{hyperbolique}\}$. Le bord $\dO$ est ainsi de classe $\C^{1+\varepsilon(\G)}$ en tout point fixe d'un élément hyperbolique de $\G$. Rappelons-nous que l'ensemble des points fixes d'élements hyperboliques de $\g$ est dense dans $\LG$, dès que $\G$ n'est pas élémentaire; on pourrait donc s'attendre au théorème qui suit, qui est toutefois faux en général.

\begin{theo}\label{bordexposant}
Supposons que $\O$ admette une action géométriquement finie \`a cusps asymptotiquement hyperboliques $\G$  d'un sous-groupe discret $\G$ de $\Aut(\O)$. Alors
$$\varepsilon(\LG) = \varepsilon(\G).$$
\end{theo}

Dans le cas où le groupe $\G$ est cocompact, l'ensemble limite $\LG$ est tout le bord $\dO$ et ce résultat a déjà été prouvé par Olivier Guichard. Nous allons donner ici une toute autre démonstration. Toutefois, remarquons que la méthode de Guichard est plus précise car elle permet de prouver que le supremum dans la définition (\ref{elg}) est en fait un maximum, c'est-à-dire que $\dO$ est exactement $\C^{1+\varepsilon(\G)}$ (et pas plus, sauf si $\O$ est un ellipsoïde).\\
Remarquons que l'hypoth\`ese faite sur les cusps est essentielle. En effet, il est possible de faire en sorte que la r\'egularit\'e en un point parabolique soit aussi mauvaise que l'on veut car lorsque le groupe parabolique n'est pas de rang maximal, il n'impose la r\'egularit\'e du bord au point fixe que dans certaines directions. On consultera la partie \ref{sectionex1} \`a ce propos.\\

Notre démonstration repose sur une approche dynamique et en particulier sur l'extension d'un résultat de Ursula Hamenstädt concernant le ''meilleur'' coefficient de contraction du flot géodésique \cite{MR1279472}. Au vu de la proposition \ref{regularite}, cette approche est en fait totalement naturelle. Dans toute la suite, nous supposerons que $\G$ est sans torsion, ce qui ne change rien d'après le lemme de Selberg, et nous étudierons plus en détail le flot géodésique de $M=\Quo$.\\

Rappelons que l'ensemble non errant $\NW$ du flot géodésique sur $HM$ est la projection sur $HM$ de l'ensemble
$$\{w=(x,\xi) \in H\O \ |\ x^+,x^- \in \LG\}.$$
Notons $\chi(\NW)$ la meilleure constante d'hyperbolicité du flot géodésique sur l'ensemble non errant; autrement dit, $\chi(\NW)$ est le supremum des réels $\chi\geqslant 0$ tels qu'il existe $C>0$ tel que, pour tout $w\in \NW$ et tous $Z^s\in E^s(w), Z^u\in E^u(w)$, on ait
$$\|d\ph^t(Z^s)\| \leqslant C e^{-\chi t} \|Z^s\|,\ \|d\ph^{-t}(Z^u)\| \leqslant C e^{-\chi t} \|Z^u\|,\ t\geqslant 0.$$
On a en fait
$$\chi(\NW) = \inf_{w\in \NW} \chi(w).$$

On a déjà vu que les points périodiques formaient un ensemble $Per$ dense dans $\NW$. Notons
$$\chi(Per)=\inf\{\chi(w)\ |\ w\in \NW\ \text{périodique}\}.$$
Le résultat, inspiré de celui d'Hamenstädt, est le suivant:

\begin{theo}\label{hamenstadt}
Soit $M=\Quo$ une variété géométriquement finie \`a cusps asymptotiquement hyperboliques. On a
$$\chi(\NW) = \chi(Per).$$
\end{theo}

Voyons tout de suite comment ce dernier résultat implique directement le théorème \ref{bordexposant}. Rappelons que l'ensemble des orbites périodiques est en bijection avec les classes de conjugaison d'éléments hyperboliques de $\G$: si $\g\in\G$ est hyperbolique, la projection de la géodésique orientée $(x_{\g}^-x_{\g}^+)$ sur $HM$ est une orbite périodique du flot géodésique, qu'on note encore $\g$. Associé à cette orbite $\g$, on dispose du plus petit coefficient de contraction:
$$\chi(\g) := \inf_{Z^s\in E^s(w)} \liminf_{t\to +\infty} -\frac{1}{t} \ln \|d\ph^t Z^s\| = \chi(w),$$
où $w$ est un point quelconque de l'orbite $\g$. L'\'egalit\'e principale est la suivante, qui d\'ecoule de la proposition \ref{regularite}:
\begin{equation}\label{reghyp}
\varepsilon(\G)=\displaystyle\frac{1}{\frac{2}{\chi(Per)} -1}.
\end{equation}

Il s'av\`ere qu'on peut exprimer $\chi(\g)$ en fonction des valeurs propres de $\g$, comme l'affirme le lemme suivant, montré dans \cite{MR2094116} ou dans \cite{MR2587084}:

\begin{lemm}\label{exposanthyperbolique}
Soit $\g\in\Aut(\O)$ un élément hyperbolique. Notons $\l_0(\g) \geqslant \l_1(\g) \geqslant \cdots \geqslant \l_n(\g)>0$ les modules de ses valeurs propres, comptées avec multiplicité. Alors
$$\chi(\g) = 2\left(1 - \frac{ \ln \l_0(\g) - \ln \l_{n-1}(\g)}{ \ln \l_0(\g) - \ln \l_n(\g)}\right).$$
\end{lemm}

Ce lemme permet d'obtenir le
\begin{coro}\label{coco}
Pour toute variété quotient $M=\O/\G$, on a $\chi(Per)\leqslant 1$. De plus, si $\chi(Per)=1$, alors $\G$ n'est pas Zariski-dense dans $\ss$.
\end{coro}

\begin{proof}
En gardant les notations du lemme, on voit que si $\g\in\Aut(\O)$ est hyperbolique, alors
$$\chi(\g^{-1}) =  2\left(1 - \frac{ \ln \l_{1}(\g) - \ln \l_{n}(\g)}{ \ln \l_0(\g) - \ln \l_n(\g)}\right).$$
Ainsi,
$$\chi(\g)+\chi(\g^{-1}) = 2\left(1 - \frac{ \ln \l_{1}(\g) - \ln \l_{n-1}(\g)}{ \ln \l_0(\g) - \ln \l_n(\g)}\right) \leqslant 2.$$
Cela implique que soit $\chi(\g)\leqslant 1$ soit $\chi(\g^{-1})\leqslant 1$, et donc que
$$\chi(Per) = \inf \{\chi(\g)\ |\ \g\in\G\  \text{hyperbolique}\}\leqslant 1.$$

Maintenant, supposons que $\chi(Per)=1$, c'est-\`a-dire que $\chi(\g)\geqslant 1$ pour tout élément $\g\in\G$ hyperbolique. De l'in\'egalit\'e $\chi(\g)+\chi(\g^{-1})\leqslant 2$, on d\'eduit que $\chi(\g)= 1$ pour tout élément $\g\in\G$ hyperbolique. D'après le lemme pr\'ec\'edent, cela veut dire que, pour tout élément $\g\in\G$ hyperbolique,
$$2\left(1- \frac{ \ln \l_0(\g) - \ln \l_1(\g)}{ \ln \l_0(\g) - \ln \l_n(\g)}\right) = 1,$$
soit
$$\ln \l_0(\g) + \ln \l_n(\g) - 2\ln \l_1(\g) = 0.$$
En particulier, l'ensemble $\ln(\G)$ n'engendre pas tout l'espace
$$\ln(\ss)=\{x=(x_0,\cdots,x_n)\in\R^{n+1},\ \sum_{i=0}^n x_i = 0\}.$$ D'après le théorème \ref{cone_limite}, cela implique que $\G$ n'est pas Zariski-dense dans $\ss$.
\end{proof}

On peut maintenant donner une
\begin{proof}[Démonstration du théorème \ref{bordexposant}]
D'apr\`es l'\'egalit\'e (\ref{reghyp}), on a
$$\varepsilon(\g) =\min\left(1,\frac{1}{\frac{2}{\chi(\g)} -1}\right),$$
et donc, via le corollaire \ref{coco}, que
$$\varepsilon(\G) =\frac{1}{\frac{2}{\chi(Per)} -1}.$$
De même, on a
$$\varepsilon(\LG) =\frac{1}{\frac{2}{\chi(\NW)} -1}.$$
Le théorème \ref{hamenstadt} implique
$$\varepsilon(\LG) = \varepsilon(\G).$$
\end{proof}

\subsection{Conséquences}

La régularité optimale du bord aux points de $\LG$ se lit donc sur les valeurs propres des éléments de $\G$:

\begin{theo}
Supposons que $\O$ admette une action géométriquement finie \`a cusps asymptotiquement hyperboliques d'un sous-groupe discret $\G$ de $\Aut(\O)$. Alors
$$\varepsilon(\LG) = \inf\left\{\frac{ \ln \l_{n-1}(\g) - \ln \l_n(\g)}{ \ln \l_0(\g) - \ln \l_{n-1}(\g)},\ \g\in\G\ \text{hyperbolique}\right\}.$$
\end{theo}
\begin{proof}
C'est un simple calcul à partir de l'égalité
$$\varepsilon(\LG) = \frac{1}{\frac{2}{\chi(Per)} -1}.$$
\end{proof}

Les corollaires ci-dessous sont sûrement plus parlants, et sont compl\'ementaires du résultat de rigidité de Benoist concernant la régularité des convexes divisibles (Proposition 6.1 de \cite{MR2094116}). Il repose sur le r\'esultat suivant, cas particulier d'un des th\'eor\`emes principaux de \cite{Crampon:2012fk}: 

\begin{theo}
Soit $\G$ un sous-groupe discret de $\Aut(\O)$. Si $\G$ agit de fa\c con g\'eom\'etriquement finie sur $\O$ et contient un \'el\'ement parabolique, alors son adh\'erence de Zariski est soit $\ss$ tout entier, soit conjugu\'ee \`a $\SO$.
\end{theo}

\begin{coro}\label{rigidite_geofini}
Supposons que $\O$ admette une action géométriquement finie d'un sous-groupe discret $\G$ de $\Aut(\O)$ qui contienne un \'el\'ement parabolique. Si le bord $\dO$ est de classe $\C^{1+\varepsilon}$ pour tout $0<\varepsilon<1$ en tout point de $\LG$, alors $\G$ est conjugué à un sous-groupe de $\SO$.
\end{coro}
\begin{proof}
Les hypoth\`eses impliquent en particulier que $\varepsilon(\G)=1$, soit $\chi(Per)=1$. Le corollaire \ref{coco} implique que $\G$ n'est pas Zariski-dense dans $\ss$. D'apr\`es le th\'eor\`eme pr\'ec\'edent, $\G$ est Zariski-dense dans un conjugu\'e de $\SO$.
\end{proof}

Un cas particulier est le suivant, o\`u l'on obtient une \emph{vraie} rigidit\'e:

\begin{coro}\label{rigidite_volfini}
Si $\O$ admet un quotient de volume fini non compact, alors le bord $\dO$ est de classe $\C^{1+\varepsilon}$ pour tout $0<\varepsilon<1$ si et seulement si $\O$ est un ellipsoïde.
\end{coro}

\subsection{Démonstration du th\'eor\`eme \ref{hamenstadt}}

Cette d\'emonstration, largement inspir\'ee de \cite{MR1279472}, est assez technique. En voici d'abord le schéma, qui repousse la partie la plus d\'elicate, incluse dans le lemme \ref{lemmehamenstadt}, à la suite.

\begin{proof}
Notons pour simplifier $\chi = \chi(Per)$. Cela veut dire que pour tout $\varepsilon>0$ et point $w\in \NW$ périodique, il existe une constante $C_{\varepsilon}(w)$ telle que, pour tout vecteur tangent stable $Z$ en $w$, on ait
$$\|d\ph^t Z\| \leqslant C_{\varepsilon}(w) e^{-(\chi-\varepsilon) t} \|Z\|.$$

Considérons l'ensemble
$$A_{T,\varepsilon} = \left\{w\in \NW\ |\ \forall Z\in E^s(w),\ \frac{\|d\ph^T Z\|}{\|Z\|} \leqslant e^{-(\chi-\varepsilon)T}\right\}.$$
Un point $w$ n'est pas dans $A_{T,\varepsilon}$ s'il existe un vecteur stable en $w$ qui est contracté par $\ph^T$ avec un exposant inférieur à $\chi-\varepsilon$. En particulier, à $\varepsilon$ fixé, pour tout point $w$ périodique, il existe un temps $T(w)$ tel que pour tout $t\geqslant T(w)$, $w\in A_{t,\varepsilon}$. On va montrer qu'en fait l'orbite de tout point $w\in \NW$ sous $\ph^T$ passe ``la plupart du temps'' dans $A_{T,\varepsilon}$ si $T$ est assez grand. Pour cela on pose
$$N_{n,T,\varepsilon} = \frac{1}{n} \sum_{i=0}^{n-1} \mathbf{1}_{A_{T,\varepsilon}}(\ph^{iT}w).$$
$N_{n,T,\varepsilon}$ compte la proportion des $n$ premiers points de l'orbite de $w$ sous $\ph^T$ qui sont dans $A_{T,\varepsilon}$. Le résultat principal est le suivant :

\begin{lemm}\label{lemmehamenstadt}
Pour tous $\varepsilon,\delta >0$, il existe $T=T(\varepsilon,\delta)$ et $N=N(T)$ tels que, pour $n\geqslant N$ et $w\in \NW$,
$$N_{n,T,\varepsilon}(w) \geqslant 1-\delta.$$
\end{lemm}

Le théorème découle aisément de ce lemme. En effet, pour tous $w\in \NW$ et $Z\in E^s(w)$, on a
$$\begin{array}{rcl}
\displaystyle \frac{\|d\ph^{nT} Z\|}{\|Z\|} & \leqslant & \displaystyle\prod_{i=0}^{n-1} \displaystyle\frac{\|d\ph^{(i+1)T} Z\|}{\|d\ph^{iT}Z\|} \\

& \leqslant & \displaystyle\prod_{0\leqslant i \leqslant n-1,\ \ph^{iT}w\not\in A_{t,\varepsilon}} \displaystyle\frac{\|d\ph^{(i+1)T} Z\|}{\|d\ph^{iT}Z\|}  \displaystyle\prod_{0\leqslant i \leqslant n-1,\ \ph^{iT}w\in A_{t,\varepsilon}}  \displaystyle\frac{\|d\ph^{(i+1)T} Z\|}{\|d\ph^{iT}Z\|}  \\\\

&\leqslant & 1\ .\ e^{-(\chi-\varepsilon)N_{n,T,\varepsilon}(w) nT}.
  \end{array}
 $$
La majoration par $1$ du premier terme d\'ecoule du lemme \ref{decroissance}. Si maintenant, à $\varepsilon$ et $\delta$ fixés, on prend $T \geqslant T(\varepsilon, \delta)$ et $n\geqslant N(T)$, on obtient, pour tout $Z\in E^s$:
$$\displaystyle \frac{\|d\ph^{nT} Z\|}{\|Z\|} \leqslant  e^{-(\chi-\varepsilon)(1-\delta) nT}.$$
En prenant $\delta=\frac{\varepsilon}{\chi-\varepsilon}$, cela donne
$$\displaystyle \frac{\|d\ph^{nT} Z\|}{\|Z\|} \leqslant e^{-(\chi-2\varepsilon)nT}.$$
On en conclut qu'il existe une constante $C$ telle que pour tous $Z\in E^s$ et $t\geqslant 0$,
$$\displaystyle \frac{\|d\ph^{t} Z\|}{\|Z\|} \leqslant C  e^{-(\chi-2\varepsilon)t}.$$
C'est gagné, puisque $\varepsilon$ est arbitrairement petit.
\end{proof}

Nous allons ici nous servir de la proposition \ref{densitemesures} pour attaquer la partie technique: le lemme \ref{lemmehamenstadt}. On définit pour $k\in\N$ et $w\in \NW$,
$$Q_k(w) = \inf \{-\frac{1}{2^k}\ln \frac{\|d\ph^{2^k} Z\|}{\|Z\|},\ Z \in E^s(w)\}.$$

Comme le flot est uniformément hyperbolique, $Q_k$ est $\geqslant 0$ et majoré, indépendamment de $k$. Posons alors
$$F_{k,\varepsilon} = \max \{0, (\chi-\varepsilon)-Q_k\}.$$
Ainsi, $F_{k,\varepsilon} (w)>0$ s'il y a un vecteur stable en $w$ qui est contracté avec un exposant inférieur à $\chi-\varepsilon$. Autrement dit, $F_{k,\varepsilon} (w)>0$ si et seulement si $w\not\in A_{2^k,\varepsilon}$.\\

Les fonctions $Q_k$ et $F_{k,\varepsilon}$ sont toutes deux positives, continues sur $\NW$, et  majorées indépendamment de $k$. Le premier lemme est le suivant:

\begin{lemm}
Soit $\varepsilon>0$. Pour $k$ assez grand, la fonction $F_{k,\varepsilon}$ est à support compact sur $\NW$.
\end{lemm}
\begin{proof}
Rappelons-nous que le coeur convexe $C(M)$ de $M$ se décompose en une partie compacte $K$ et un nombre fini de cusps $\C_i,\ 1\leqslant i \leqslant l$, asymptotiquement hyperboliques. On a montré (voir la remarque \ref{deccusp}) que pour tout point $w\in HM$ tel que $\ph^s(w)\in H\C_i,\ 0\leqslant s\leqslant t$, on avait
$$\|d\ph^t Z^s \| \leqslant M e^{-t}\|Z^s\|,\ Z^s\in E^s(w),$$
pour une certaine constante $M>0$. Ainsi, pour tout $\delta>0$, il existe $T_{\delta}$ tel que pour $t\geqslant T_{\delta}$, on ait, pour tout point $w\in HM$ tel que $\ph^s(w)\in H\C_i,\ 0\leqslant s\leqslant t$,
$$\|d\ph^t Z^s \| \leqslant e^{-(1-\delta)t}\|Z^s\|,\ Z^s\in E^s(w).$$
Comme d'après le corollaire \ref{coco}, on a toujours $\chi\leqslant 1$, on peut prendre $\delta = 1-\frac{1+\chi-\varepsilon}{2}$ et $k$ tel que $2^k\geqslant T_{\delta}$. On obtient alors que, pour tout point $w\in HM$ tel que $\ph^s(w)\in H\C_i,\ 0\leqslant s\leqslant 2^k$,
$$\|d\ph^{2^k} Z^s \| \leqslant e^{-\frac{1+\chi-\varepsilon}{2}2^k}\|Z^s\|,\ Z^s\in E^s(w).$$
En particulier, $F_{k,\varepsilon}=0$ sur l'ensemble
$$\{w\in \NW\ |\ \ph^s(w)\in H\C_i,\ 0\leqslant s\leqslant 2^k\},$$
dont le complémentaire dans $\NW$ est compact. La fonction $F_{k,\varepsilon}$ est donc à support compact pour $k$ assez grand.
\end{proof}

Rappelons qu'on a not\'e $\M$ l'ensemble des mesures de probabilit\'es sur $\NW$ invariantes par le flot.

\begin{lemm}\label{pouc}
Pour tous $\varepsilon, \delta>0$, il existe $k(\varepsilon,\delta)$ tel que pour tout $k\geqslant k(\varepsilon,\delta)$ et tout $\eta\in\M$,
$$\int F_{k,\varepsilon} \ d\eta < \delta.$$
\end{lemm}
\begin{proof}
Fixons $\varepsilon>0$, et choisissons $k$ assez grand pour que la fonction $F_{k,\varepsilon}$ soit à support compact $S$ sur $\NW$. Notons $\mathfrak{M}$ l'espace vectoriel des mesures de Radon signées sur $\NW$, muni de la topologie de la convergence étroite des mesures. Si $A$ est un compact de $\NW$ et $m>0$, l'ensemble des mesures de $\mathfrak{M}$ à support dans $A$ et de masse totale $\leqslant m$ est compact pour cette topologie. En particulier, en notant, pour $\eta\in\mathfrak{M}$, $\eta_S$ la mesure définie par
$$\eta_S(B) = \eta(S\cap B),\ B\ \text{Borélien de}\ \NW,$$
l'ensemble
$$\M(S) = \{\eta_S,\ \eta\in\M\}$$
est un ensemble compact.\\

On définit une forme linéaire sur $\mathfrak{M}$  par
$$\Psi_k : \eta\in \mathfrak{M} \longmapsto \int F_{k,\varepsilon} \ d\eta \in \R.$$
Remarquons tout de suite que $\Psi_k(\eta)=\Psi_k(\eta_S)$ pour toute mesure $\eta\in \mathfrak{M}$.
La forme linéaire $\Psi_k$ est positive et continue, et surtout la suite $(\Psi_k)$ est uniformément bornée : pour tout $k$,
$$\|\Psi_k\| = \sup_{\|\eta\| \leqslant 1} |\Psi_k(\eta)| \leqslant \|F_{k,\varepsilon} \|_{\infty} < \chi < +\infty.$$

On munit l'espace $\mathfrak{M}'$ des formes linéaires continues sur $\mathfrak{M}$ de la topologie $*$-faible : une suite $(\Phi_n)$ de $\mathfrak{M}'$ converge vers $\Phi$ si pour toute mesure $\eta\in\mathfrak{M}$, $(\Phi_n(\eta))$ converge vers $\Phi(\eta)$. Pour cette topologie, les ensembles bornés sont relativement compacts. En particulier, on peut supposer, quitte à extraire une sous-suite, que la suite $(\Psi_k)$ converge vers $\Psi \in \mathfrak{M}'$.\\

Notons $C(\M_{Per})$ l'enveloppe convexe de l'ensemble $\M_{Per}$ des mesures port\'ees par les orbites p\'eriodiques. Si $\eta\in C(\M_{Per})$, on a, par construction, $\Psi(\eta)=0$. Par densité de $C(\M_{Per})$ dans $\M$ (proposition \ref{densitemesures}) et continuité de $\Psi$, on en déduit que $\Psi=0$ sur $\M$. Maintenant, en \'ecrivant, pour $Z\in E^s$,
$$-\frac{1}{2^{k+1}}\ln \frac{\|d\ph^{2^{k+1}} Z\|}{\|Z\|}=\frac{1}{2} \left(-\frac{1}{2^k}\ln \frac{\|d\ph^{2^{k}} d\ph^{2^k}Z\|}{\|d\ph^{2^k}Z\|}-\frac{1}{2^k} \ln \frac{\|d\ph^{2^k}Z\|}{\|Z\|}\right),$$
on remarque facilement que
$$Q_{k+1} \geqslant \frac{1}{2}(Q_k \circ \ph^{2^k} + Q_k).$$
De l\`a, on obtient par d\'efinition de $F_{k,\varepsilon}$ que
$$F_{k+1,\varepsilon} \leqslant \frac{1}{2}(F_{k,\varepsilon} \circ \ph^{2^k} + F_{k,\varepsilon}).$$
Cela entraîne que, si $\eta\in\M$ est une mesure de probabilit\'e invariante, 
$$\Psi_{k+1}(\eta) =  \int F_{k+1,\varepsilon} \ d\eta \leqslant \frac{1}{2} \left(\int F_{k,\varepsilon} \circ \ph^{2^k} \ d\eta + \int F_{k,\varepsilon} \ d\eta\right) = \int F_{k,\varepsilon}\ d\eta = \Psi_k(\eta).$$
Ainsi, pour tout $\eta\in\M$, la suite $(\Psi_k(\eta))$ est décroissante.\\

Ainsi, la suite de fonctions $\Psi_k : \M \longrightarrow \R$ converge en décroissant vers $0$. L'ensemble $\M(S)$ \'etant compact, le théorème de Dini entraîne que la convergence de $(\Psi_k)$ vers $0$ est uniforme sur $\M(S)$, donc sur $\M$ puisque $\Psi_k(\eta)=\Psi_k(\eta_S)$ pour tout $\eta\in\M$. Autrement dit, pour tout $\delta>0$, il existe $k(\varepsilon, \delta)$ tel que pour $k\geqslant k(\varepsilon, \delta)$ et $\eta\in\M$,
$$\Psi_k(\eta) = \int F_{k,\varepsilon} \ d\eta \leqslant \delta.$$
\end{proof}

Nous aurons enfin besoin de l'observation \'el\'ementaire qui suit:

\begin{lemm}\label{interval}
Soient $\varepsilon>0$ et $T>0$. Si $w\not\in A_{T,\varepsilon}$, alors, pour tout $-\frac{\varepsilon}{8}<\l<\frac{\varepsilon}{8}$, $\ph^{\l T}(w) \not \in A_{T,\varepsilon/2}$.
\end{lemm}
\begin{proof}
Soit $\varepsilon>0$. En vertu du corollaire \ref{decroitpasvite}, pour tout $t\geqslant 0$ et tout $Z \in E^s$,
$$\frac{\|d\ph^t Z\|}{\|Z\|} \geqslant e^{-2t}.$$
Ainsi, pour tous $T> 0$, $\l\in\R$,
$$\frac{\|d\ph^{(1+\l)T} Z\|}{\|Z\|} = \frac{\|d\ph^{\l T} d\ph^{T}Z\|}{\|d\ph^{T}Z\|} \frac{\|d\ph^{T} Z\|}{\|Z\|}\geqslant e^{-2\l T}\frac{\|d\ph^T Z\|}{\|Z\|}.$$
Par cons\'equent, si $w\not\in A_{T,\varepsilon}$, alors, pour tout $\l\geqslant 0$ et tout $Z \in E^s(\ph^t w)$,
$$
\begin{array}{rl}
\displaystyle \frac{\|d_{\ph^t w} \ph^{T} Z\|}{\|Z\|}
& = \displaystyle\frac{\|d_{w} \ph^{T+t} d_{\ph^t w}\ph^{-t} Z\|}{\|d_{\ph^t w}\ph^{-t} Z\|}\frac{\|d_{\ph^t w}\ph^{-t} Z\|}{\|Z\|}\\\\
& \geqslant e^{-2\l T}  \displaystyle\frac{\|d_{w} \ph^{T} d_{\ph^t w}\ph^{-t} Z\|}{\|d_{\ph^t w}\ph^{-t} Z\|} e^{-2\l T}\\\\ 
& \geqslant e^{-4\l T} e^{-(\chi-\varepsilon)T} \\
& \geqslant e^{-(\chi-\varepsilon + 4\l)T}.
\end{array}$$
Cela implique que, si $w\not\in A_{T,\varepsilon}$, alors $\ph^{\l T}(w) \not \in A_{T,\varepsilon/2}$ pour tout $-\varepsilon/8<\l<\varepsilon/8$.
\end{proof}

On peut maintenant démontrer le lemme \ref{lemmehamenstadt}:

\begin{proof}[Démonstration du lemme \ref{lemmehamenstadt}]
Fixons pour la suite $\varepsilon>0$. Choisissons $\delta >0$ et $T = 2^k$ avec $k \geqslant k(\varepsilon/4,\delta\varepsilon^2/16)$ donné par le lemme \ref{pouc}. On a donc, pour toute mesure de probabilit\'e invariante $\eta$,
$$\int F_{k,\varepsilon/4}\ d\eta < \delta  \frac{\varepsilon^2}{16}.$$

On procède par l'absurde en supposant qu'il existe un point $w\in \NW$ et une suite $(n_j)_{j\in\N}$ telle que $N_{n_j,\varepsilon, T}(w) \leqslant 1-\delta$. D'apr\`es le lemme \ref{interval}, \`a chaque fois qu'un point $\ph^{iT}(w)$, $i\in \N$, de l'orbite de $w$ sous $\ph^T$ n'est pas dans $A_{T,\varepsilon}$, alors $\ph^{t+iT}(w)$ n'est pas dans $A_{T,\varepsilon/2}$ pour tout $-\varepsilon T/8<t<\varepsilon T/8$; et donc $F_{k,\varepsilon/4}(\ph^t w) \geqslant \frac{\varepsilon}{4}$.\\
Or, parmi les points $\ph^{iT}(w), 0\leqslant i \leqslant n_j$ de l'orbite de $w$ sous $\ph^T$, il y en a $N_{n_j,\varepsilon, T}(w) n_j$ qui sont dans $A_{T,\varepsilon}$; cela implique qu'entre les instants $0$ et $n_jT$, l'orbite de $w$ n'est pas dans $A_{T,\varepsilon/2}$ pendant au moins le temps $n_j T\delta \varepsilon /4$. Autrement dit,

$$\begin{array}{rl}
\displaystyle \frac{1}{n_jT} \int_{0}^{n_j T} F_{k,\varepsilon/4}(\ph^t w)\ dt & \geqslant\displaystyle \frac{1}{n_jT} \int_{0}^{n_j T} F_{k,\varepsilon/4}(\ph^t w)(1-\mathbf{1}_{A_{T,\varepsilon/2}}(\ph^t w)) \ dt \\\\
& \geqslant \displaystyle\frac{1}{n_jT}\ \frac{\varepsilon}{4}\  n_j T \delta  \frac{\varepsilon}{4}\\\\
& = \delta \displaystyle \frac{\varepsilon^2}{16}.
  \end{array}$$

On définit la suite de mesures de probabilités $(\eta_j)_{j\in\N}$ par
$$\int f\ d\eta_j = \frac{1}{n_jT} \int_0^{n_jT} f(\ph^t(w))\ dt,\ f\in C(\NW).$$
Toute valeur d'adhérence $\eta$ de la suite $(\eta_j)$ vérifie
$$\int F_{k,\varepsilon/4}\ d\eta \geqslant \delta  \frac{\varepsilon^2}{16}.$$
Or, une telle mesure $\eta$ est nécessairement invariante par le flot, et cela contredit le choix de $k$.
\end{proof}

\section{Quelques exemples}\label{sectionex}

\subsection{Un exemple o\`u le flot g\'eod\'esique n'est pas uniform\'ement hyperbolique}\label{sectionex1}

On va construire un exemple ``d\'eg\'en\'er\'e'' o\`u le flot g\'eod\'esique a un exposant de Lyapunov nul. Pour cela, on consid\`ere un certain groupe fuchsien $\G$ qui contient un parabolique, et on le fait agir de fa\c con canonique sur $\HH^3$; on construit alors un nouvel ouvert convexe $\G$-invariant dont le bord n'est de classe $\C^{1+\varepsilon}$ en aucun point parabolique, pour tout $\varepsilon>0$. D'apr\`es la proposition \ref{regularite}, toute orbite ultimement incluse dans le cusp aura un exposant de Lyapunov nul.\\ 
Le groupe $\G$ ici pr\'esent\'e n'est donc pas irr\'eductible sur $\PP^3$ mais il est sans doute possible de le d\'eformer par pliage, tout en pr\'eservant les propri\'et\'es de r\'egularit\'e que l'on d\'esirait.\\
Le r\'esultat s'\'enonce ainsi:

\begin{prop}\label{nonanosov}
Il existe un ouvert proprement convexe $\O$ de $\PP^3$, strictement convexe à bord $\C^1$, et un sous-groupe discret $\G$ de $\Aut(\O)$ tels que 
\begin{itemize}
 \item l'action de $\G$ sur $\O$ soit géométriquement finie mais non convexe-cocompacte;
 \item le bord $\dO$ de $\O$ n'est pas de classe $\C^{1+\varepsilon}$ aux points paraboliques de $\LG$, pour tout $\varepsilon > 0$.
\end{itemize}
En particulier, le flot géodésique sur la variété quotient $\Quo$ a un exposant de Lyapunov nul; il n'est donc pas uniformément hyperbolique.
\end{prop}

\begin{proof}
Soient $\Sigma$ le  tore à 1 trou et $\G$ son groupe fondamental; $\G$ est un groupe libre à $2$ générateurs. On munit $\Sigma$ d'une structure hyperbolique de volume fini de la façon suivante: on se donne un carré idéal $P$ de $\HH^2$ et on identifie les côtés opposés de ce carré à l'aide de deux éléments hyperboliques $\g$ et $\delta$. Ainsi le domaine fondamental pour l'action de $\G$ sur $\HH^2$ est le carré idéal en question. Pour simplifier la discussion, on choisit ce carré idéal de telle fa\c con qu'il ait un groupe diédral d'ordre 8 de symétrie.\\

\`A présent, on plonge $\G$ dans $\so{3}$ de façon canonique. Ainsi, $\G$ agit sur l'espace hyperbolique $\HH^3$ de dimension $3$. L'ensemble limite de $\G$ sur $\HH^3$ est un cercle, intersection d'un plan projectif $\Pi$ et de $\partial\HH^3$.\\
Le groupe $\G$ préserve le plan $\Pi$ et le point $M$ intersection des hyperplans tangents à $\partial \HH^3$ en $\LG= \Pi \cap \partial \HH^3$. L'ouvert convexe $\O_{\infty}$ obtenu en prenant la réunion des deux cônes de bases $\Pi \cap \HH^3$ et de sommet $M$ est pr\'eserv\'e par $\G$ et par $\so{2}$. Un domaine fondamental $D$ pour l'action de $\G$ sur $\O_{\infty}$ est la réunion des deux cônes de base $P$ et de sommet $M$. Nous allons construire une partie convexe $D_1$ de $D$ telle que la réunion $\underset{\g \in \G}{\bigcup} \g(D_1)$ nous donne l'ouvert convexe $\O$ d\'esir\'e.

\begin{center}
\begin{figure}[h!]
  \centering
\includegraphics[width=5cm]{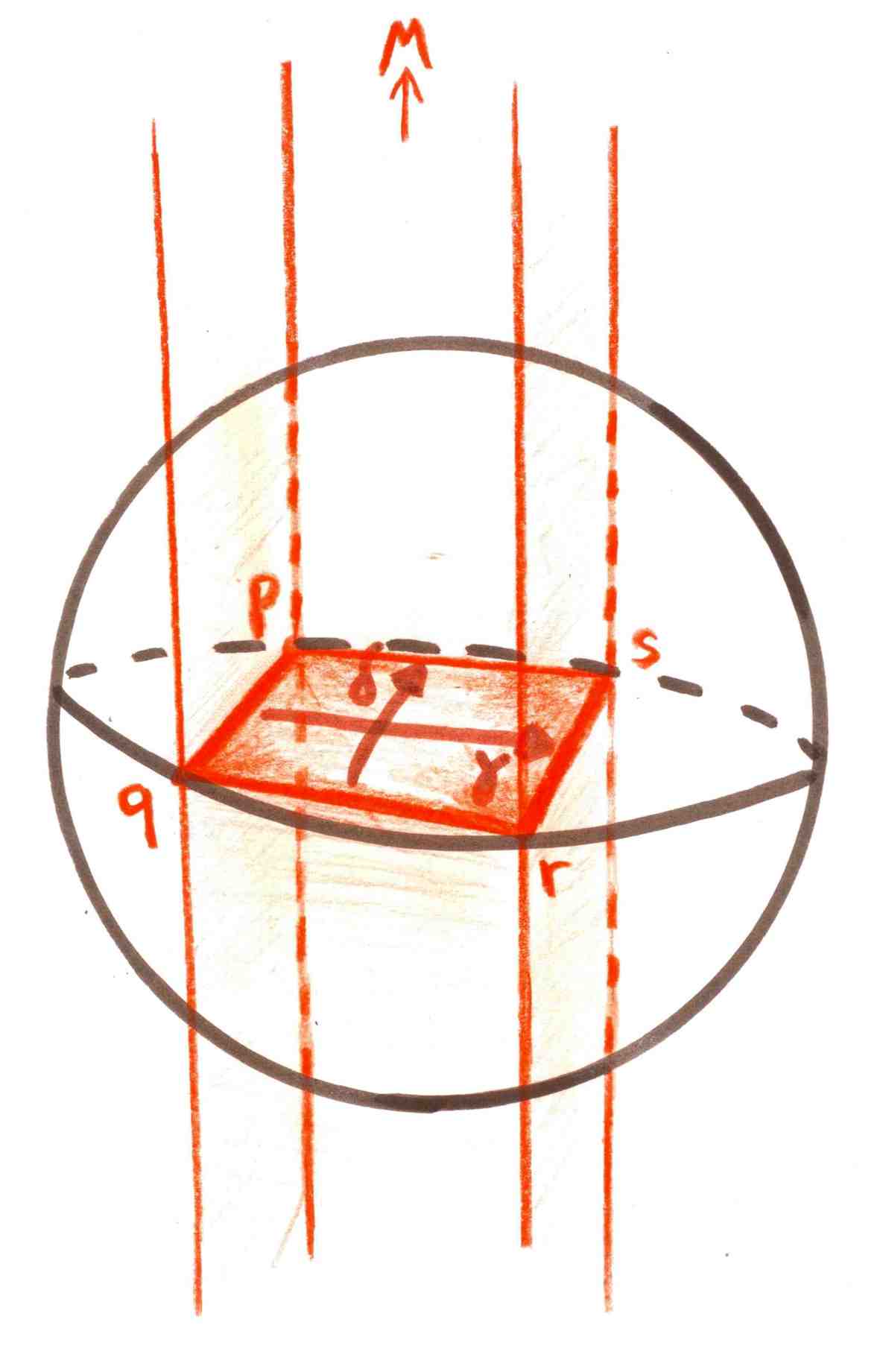}
\caption{Domaines fondamentaux}
\label{domaines}
\end{figure}
\end{center}

On note $p,q,r,s$ les sommets de $P$, $\g$ l'élément qui identifie $[pq]$ avec $[sr]$ et $\delta$ celui qui identifie $[rq]$ avec $[sp]$ (voir la figure \ref{dome}); le groupe $\G$ est engendr\'e par ces deux \'el\'ements $\g$ et $\delta$. On appelle $\Pi_{p,q}$ (resp. $\Pi_{q,s}$, ...) le plan contenant $p,q$ (resp. $q,s$, ...) et $M$.\\
On commence par s'intéresser à l'intersection de $\O_{\infty}$ avec le plan $\Pi_{p,q}$ engendré par $p,q$ et $M$.
On choisit une premi\`ere courbe $\C_{p,q}$ qui joint $p$ \`a $q$ puis $q$ \`a $p$ et qui 
\begin{itemize}
 \item est incluse dans la face de $D$ contenant $[pq]$, c'est-\`a-dire $\Pi_{p,q}\cap  \overline{D}$;
 \item est strictement convexe et de classe $\C^1$;
 \item n'est pas de classe $\C^{1+\varepsilon}$ en $p$ et $q$ pour aucun $\varepsilon > 0$.
\end{itemize}
On peut remarquer au passage que cette courbe est, par la derni\`ere propri\'et\'e, incluse dans $\HH^3$ au voisinage de $p$ et $q$. On utilise à présent les symétries du groupe $\G$ (c'est-\`a-dire celles de $P$) pour copier cette courbe et obtenir des courbes $\C_{p,s}$, $\C_{r,q}$ et $\C_{s,r}$ joignant $\{p,s\}$, $\{r,q\}$ et $\{s,r\}$.\\
Nous allons ``relier'' ces courbes pour construire le bord de $D_1$ (qu'on appellera le d\^ome). Pour simplifier la discussion, il est bon de remarquer que cet ensemble de courbes admet le même groupe de symétrie que le carré idéal de départ, à savoir un groupe diédral d'ordre 8.\\

Soit $\g^{\R}$ (resp. $\delta^{\R}$) le groupe à 1-paramètre engendré par $\g$ (resp. $\delta$). Les orbites d'un point de $\O_{\infty} \smallsetminus \Pi$ sous $\g^{\R}$ sont des demi-ellipses d'extrémités $\g^-$ et $\g^+$; de m\^eme pour $\delta^{\R}$.\\
Le domaine fondamental $D$ privé de $\Pi_{p,r} \cup \Pi_{q,s}$ possède 4 composantes connexes $D_{\delta^-}$, $D_{\g^-}$, $D_{\delta^+}$ et $D_{\g^+}$, naturellement étiquetées par $\delta^-,\g^-,\delta^+,\g^+$. L'orbite de $\C_{p,q}$ sous $\g^{\R}$ est une surface convexe\footnote{On dit qu'une hypersurface de $\PP^n$ est \emph{convexe} si elle est une partie du bord d'un convexe de $\PP^n$.} $S_{p,q}$, qui contient $\C_{s,r}$. De même, en considérant l'orbite de $\C_{r,q}$ sous $\delta^{\R}$, on obtient une surface convexe $S_{r,q}$. Soit $S$ la surface obtenue comme la réunion 
$$S =(D_{\delta^-} \cap S_{r,q}) \cup ( D_{\g^-} \cap S_{p,q}) \cup (D_{\delta^+} \cap S_{r,q})  \cup (D_{\g^+} \cap S_{p,q}).$$ 
La surface $S$ possède encore un groupe diédral d'ordre 8 de symétries.

\begin{center}
\begin{figure}[h!]
  \centering
\includegraphics[width=6cm]{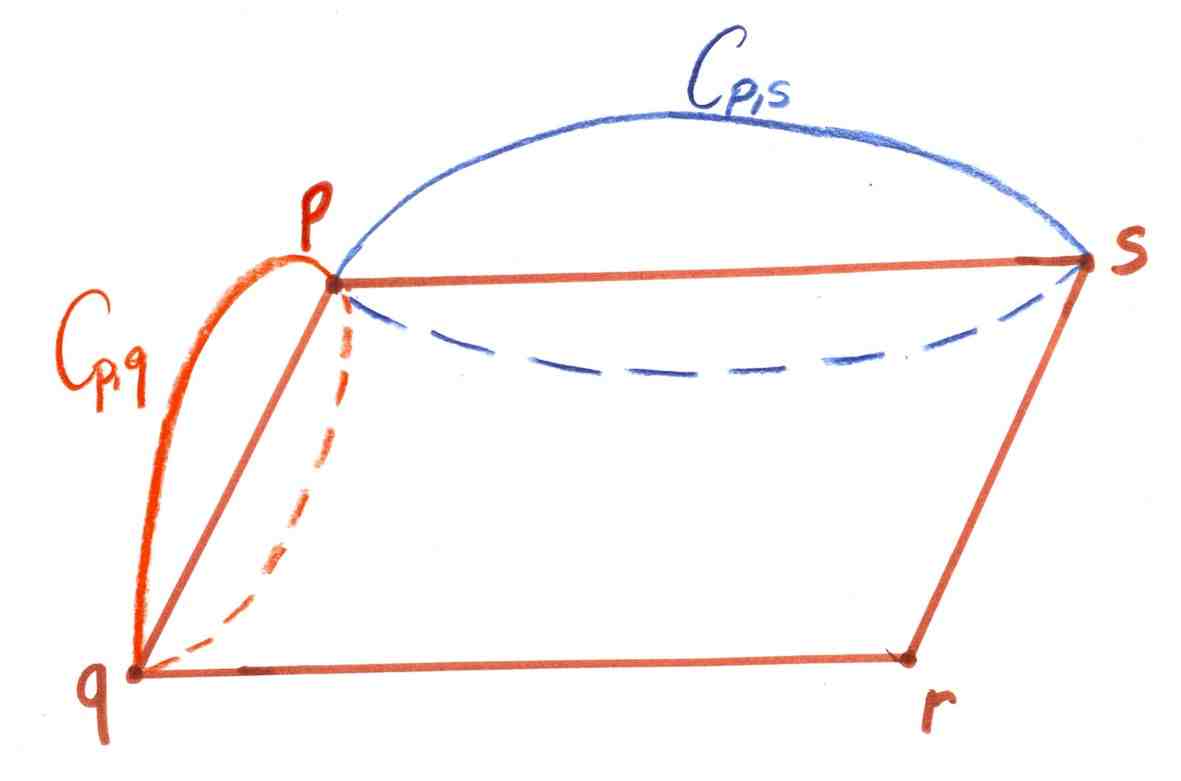}\ \includegraphics[width=6cm]{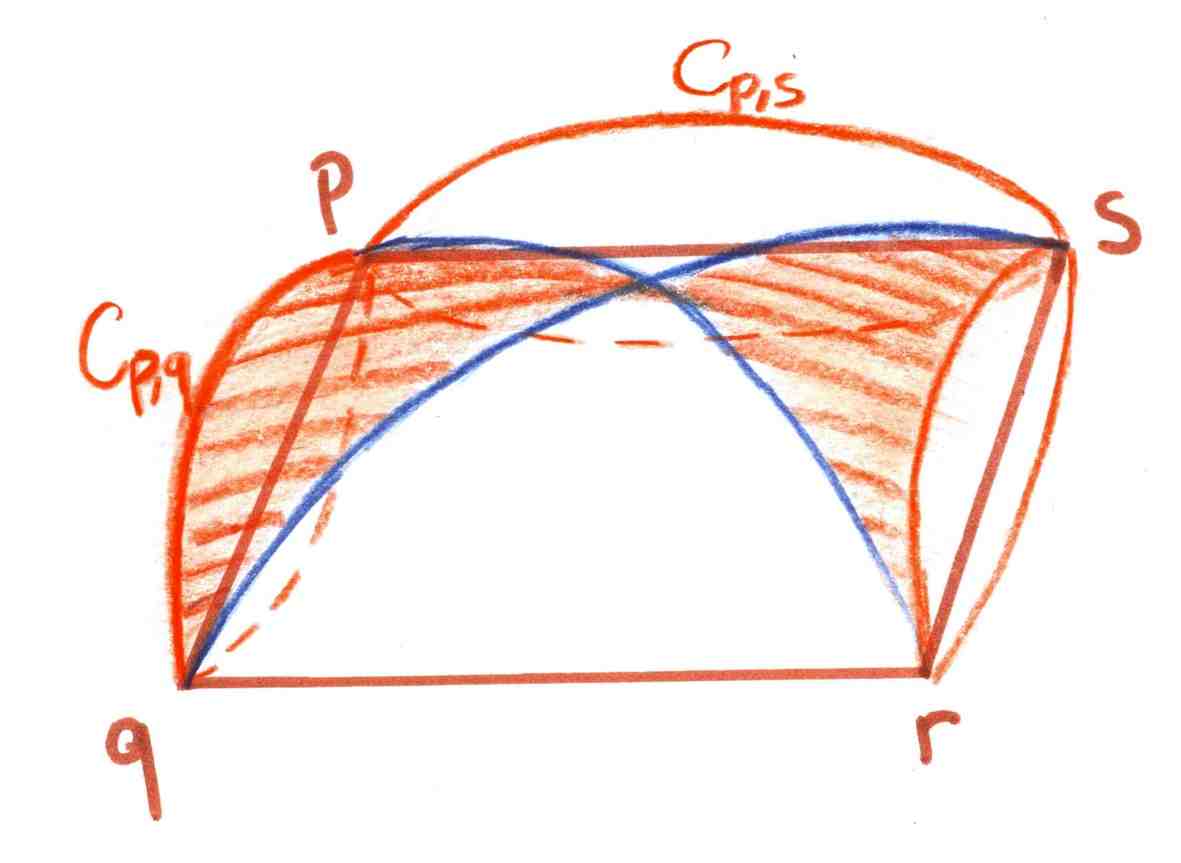}
\caption{Les courbes et le d\^ome}
\label{dome}
\end{figure}
\end{center}

Appelons $D_0$ l'adh\'erence dans $D$ de l'enveloppe convexe de la surface $S$. L'ensemble $D_0$ est convexe, inclus dans $\HH^3$. La réunion $\O_0= \underset{\g \in \G}{\bigcup} \g(D_0)$ est un ouvert proprement convexe. En effet, $\g(D_0) \cup D_0$ est encore convexe puisque, par construction, les surfaces $S_{p,q}$ et $\g(S_{p,q})$ se recollent pour donner une surface convexe; bien sûr, le m\^eme chose est valable pour $\delta$; le résultat pour $\G$ s'en déduit à l'aide d'une récurrence sur la longueur d'un élément pour la métrique des mots de $\G$. 

L'ouvert proprement convexe $\O_0$ est strictement convexe mais son bord n'est pas de classe $\C^1$ a priori. En dehors des courbes $\C_{p,r} = \Pi_{p,r} \cup \dO_0$ et $\C_{q,s} = \Pi_{q,s} \cup \dO_0$ et de leurs images par $\G$, la surface $\dO_0$ est de classe $\C^1$. En lissant $\O_0$ le long de ces courbes, on obtient qu'il existe un voisinage $\V$ de $\C_{p,r} \cup \C_{q,s}$ dans $\dO_0$ et un convexe $D_1$ tel que $\partial D_0 \smallsetminus \V= \partial D_1 \smallsetminus \V$. L'ensemble $\O_1= \underset{\g \in \G}{\bigcup} \g(D_1)$ est alors un ouvert proprement convexe, strictement convexe, à bord $\C^1$ et $\G$-invariant, mais son bord n'est pas de classe $\C^{1+\varepsilon}$ aux points paraboliques de $\LG$, pour tout $\varepsilon > 0$.
\end{proof}

\subsection{Repr\'esentation sph\'erique de $\s2$ dans $\s5$}

Dans \cite{Crampon:2012fk}, on avait introduit deux notions de finitude g\'eom\'etrique, la finitude g\'eom\'etrique sur $\O$ et sur $\dO$. Dans l'article pr\'esent, nous n'avons \'etudi\'e que la premi\`ere.\\
Comme exemple d'action g\'eom\'etriquement finie sur $\dO$ mais pas sur $\O$, on avait donn\'e la repr\'esentation sph\'erique de $\s2$ dans $\s5$: il s'agit de l'action de $\s2$ sur l'espace $V_4=\R_4[X,Y]$ des polyn\^omes homog\`enes de degr\'e $4$ en deux variables, sur lequel $\s2$ agit par coordonn\'ees. Notons $G<\s5$ l'image de $\s2$ par cette repr\'esentation.\\
Rappelons ce qui a \'et\'e vu dans \cite{Crampon:2012fk}. Il existe un point $x \in \PP^4$ dont l'orbite sous $G$ est la courbe de Veronese, dont une \'equation est donn\'ee par:
$$[t:s]\in\PP^1 \rightarrow [t^4:t^3s:t^2s^2: ts^{3}:s^4]\in\PP^4.$$
L'ensemble des ouverts proprement convexes pr\'eserv\'es par cette repr\'esentation de $\s2$ dans $\s5$ forme une famille croissante $\{\O_r,\ 0\leqslant r\leqslant \infty\}$. L'ouvert convexe $\O_0$ est l'enveloppe convexe de la courbe de Veronese, et $\O_{\infty}$ est son dual; ces deux convexes ne sont ni strictement convexes ni \`a bord $\C^1$. Les ouverts convexes $\O_r,\ 0<r<\infty$ sont les $r$-voisinages de $\O_0$ dans la g\'eom\'etrie de Hilbert $(\O_{\infty},d_{\O_{\infty}})$. On a vu dans \cite{Crampon:2012fk} que ces convexes \'etaient strictement convexes et \`a bord $\C^1$. En fait, on peut voir en proc\'edant comme dans la partie \ref{sectionregularite} que:

\begin{prop}\label{reg_spherique}
Pour $0<r<\infty$, le bord de l'ouvert convexe $\O_r$ est de classe $\C^{4/3}$ et $4$-convexe.
\end{prop}
\begin{proof}
L'ensemble limite $\L_G$ de l'action de $G$ sur $\O_r$ est dans tous les cas la courbe de Veronese. Hors de l'ensemble limite, le bord de $\O_r$ est lisse car l'action de $G$ sur $\dO_r\smallsetminus\L_G$ est libre et transitive (voir \cite{Crampon:2012fk}, section 10); $\dO_r\smallsetminus\L_G$ s'identifie donc \`a une orbite de $G$. De plus, le dual de $\O_r$ est un certain $\O_{r'}$. Le fait que $\dO_{r'}\smallsetminus\L_G$ soit lisse implique que $\dO_r\smallsetminus\L_G$ est $2$-convexe. Autrement dit, le bord de $\O_r$ est lisse \`a hessien d\'efini positif hors de l'ensemble limite.\\
La courbe de Veronese $\L_G$ est une courbe alg\'ebrique lisse sur laquelle $G$ agit transitivement. La r\'egularit\'e $\dO_r$ est donc la m\^eme en tout point de $\L_G$. Or, un point $x$ de $\L_G$ est un point fixe d'un certain \'el\'ement hyperbolique $g\in G$. La r\'egularit\'e de $\dO_r$ en $x$ se lit sur les valeurs propres de $g$. Il n'est pas dur de voir que si $g$ est l'image par la repr\'esentation d'un \'el\'ement hyperbolique $\g$ de $\s2$. Si $\l,\l^{-1}$ sont les valeurs propres de $\g$, avec $\l>1$, alors celles de $g$ sont $\l^4,\l^2,1,\l^{-2},\l^{-4}$. D'apr\`es la proposition \ref{regularite} et le lemme \ref{exposanthyperbolique}, $\dO_r$ est $\C^{1+\varepsilon}$ en $x$ pour tout $\varepsilon<\varepsilon(g)$ avec
$$\varepsilon(g) = \frac{\ln \l^{-2} - \ln \l^{-4}}{\ln \l^{4} - \ln \l^{-2}} = \frac{1}{3}.$$
En fait, dans le cas d'un point fixe hyperbolique, on peut \^etre plus pr\'ecis dans la proposition \ref{regularite} et voir que la valeur $\varepsilon(g)$ est atteinte, autrement dit que le bord est $\C^{1+\varepsilon(g)}$ en $x$ (et pas plus). On obtient donc que $\dO_r$ est $\C^{4/3}$ et, par dualit\'e, que $\dO_r$ est $4$-convexe. (En effet, le bord $\dO$ est $\beta$-convexe au point $x$ si et seulement si le bord $\dO^*$ du convexe dual est $\C^{1+\varepsilon}$ au point $x^*$, avec $\frac{1}{1+\varepsilon}+\frac{1}{\beta}=1$.)
\end{proof}

\section{Entropie volumique et exposant critique}\label{sec_entropie}

Si $\G$ est un sous-groupe discret de $\Aut(\O)$, on notera, pour $x\in\O$ et $R\geqslant 0$,
$$N_{\G} (x,R) = \sharp\{g\in \G,\ d_{\O}(x,gx)\leqslant R\}$$
le nombre d'éléments $g$ de $\G$ tels que $gx \in B(x,R)$. L'exposant critique du groupe $\G$, défini par
$$\dgg = \limsup_{R\to +\infty} \frac{1}{R} \ln N_{\G} (x,R),$$
mesure le taux de croissance exponentiel du groupe $\G$ agissant sur $\O$; il est immédiat que la limite précédente ne dépend pas du point $x$ considéré.\\
L'exposant critique $\dgg$ de $\G$ est nommé ainsi car c'est l'exposant critique des séries de Poincaré de $\G$ données par
$$g_{\G}(s,x)=\displaystyle\sum_{\g\in\Gamma} e^{-s \d(x,\g x)},\ x\in\O;$$
cela veut dire que pour $s>\dgg$, la série converge, et pour $s<\dgg$, elle diverge.\\

L'entropie volumique d'une géométrie de Hilbert
$$h_{vol}(\O) = \limsup_{R\to +\infty} \frac{1}{R} \ln \Vol_{\O} B(x,R)$$
réprésente le taux de croissance exponentiel du volume des boules de l'espace métrique $(\O,d_{\O})$.\\

\subsection{Groupes de covolume fini}

Si la géométrie $(\O,\d)$ admet une action cocompacte d'un sous-groupe discret $\G$ de $\Aut(\O)$, on a évidemment l'égalité
$$\dgg = h_{vol}(\O)$$
puisqu'alors l'entropie volumique ne dépend pas de la mesure de volume considérée, pourvu qu'elle soit $\G$-invariante; aussi peut-on prendre la mesure de comptage de l'orbite d'un point $x$ de $\O$ sous $\G$ pour retrouver $\dgg$. Si le groupe est ``trop petit'', cette égalité devient en général fausse, et on a seulement $\dgg\leqslant h_{vol}$. Dans \cite{MR2494912}, Françoise Dal'bo, Marc Peigné, Jean-Claude Picaud et Andrea Sambusetti ont étudié cette question pour les sous-groupes de covolume fini de variétés de Hadamard, à courbure négativé pincée. Ils ont montré le résultat suivant.

\begin{theo}
\
\begin{itemize}
\item Soit $M$ une vari\'et\'e riemannienne \`a courbure strictement n\'egative, de volume fini. Si $M$ est asymptotiquement $1/4$-pinc\'ee, alors $h_{vol}=h_{top}$.
 \item Pour tout $\varepsilon>0$, il existe une vari\'et\'e riemannienne de volume fini et de courbure strictement n\'egative $(1/4+\varepsilon)$-pinc\'ee telle que $h_{top}< h_{vol}$.
\end{itemize}
\end{theo}

L'hypothèse de pincement asymptotique concerne la géométrie de la variété à l'infini, c'est-à-dire dans ses cusps. Dans notre contexte, c'est le lemme \ref{loincusp} qui va nous permettre de montrer le

\begin{theo}\label{volegaltop}
Soit $\G$ un sous-groupe discret de $\Aut(\O)$, de covolume fini. Alors
$$\delta_{\G} = h_{vol}(\O).$$
\end{theo}



La démonstration de ce résultat est fort similaire à celle de \cite{MR2494912}, elle se simplifie par certains aspects et n\'ecessite des arguments un peu diff\'erents par d'autres. Elle reste malgr\'e tout un brin technique... \\
On va commencer par calculer l'exposant critique d'un sous-groupe parabolique de rang maximal.

\begin{lemm}\label{deltainclu}
Soit $\Gamma$ un sous-groupe de $\Aut(\O)$ et $\Aut(\O')$ avec $\O\subset\O'$. Appelons $g_{\Gamma,\O}(s,x)$ et $g_{\Gamma,\O'}(s,x)$ les séries de Poincar\'e pour l'action de $\G$ sur $\O$ et $\O'$, $\dgg(\O)$ et $\dgg(\O')$ leur exposant critique. Alors, pour tout $s>\dgg(\O')$, $g_{\Gamma,\O}(s,x)\leqslant g_{\Gamma,\O'}(s,x)$. En particulier, $\delta_{\Gamma}(\O) \leqslant \delta_{\Gamma}(\O')$.
\end{lemm}
\begin{proof}
Pour $x,y\in\O$, on a $d_{\O'}(x,y)\leqslant d_{\O}(x,y)$. Donc si $x\in\O$ et $s> \dgg(\O')$, on a $g_{\Gamma,\O}(s,x)\leqslant g_{\Gamma,\O'}(s,x).$ En particulier, la convergence de $g_{\Gamma,\O'}(s,x)$ implique celle de $g_{\Gamma,\O}(s,x)$, d'où le résultat.
\end{proof}

\begin{lemm}\label{par}
L'exposant critique d'un sous-groupe parabolique de rang maximal $\mathcal{P}$ de $\Aut(\O)$ est $\delta_{\mathcal{P}}=\frac{n-1}{2}$ et les séries de Poincaré de $\mathcal{P}$ divergent en $\delta_{\mathcal{P}}$:
$$\forall x\in\O,\ \displaystyle\sum_{\g\in\P} e^{-\delta_{\P} \d(x,\g x)}=+\infty.$$
\end{lemm}
\begin{proof} Appelons $p$ le point fixe de $\P$. D'après le lemme \ref{ellipsoidesecurite}, on peut trouver deux ellipsoïdes $\E^{int}$ et $\E^{ext}$ $\P$-invariants tels que 
$$\partial\E^{int}\cap \partial\E^{ext}=\partial\E^{int}\cap\dO = \partial\E^{ext}\cap\dO=\{p\}\ \text{et}\ \E^{int} \subset \O \subset \E^{ext}.$$ 
Il est connu en géométrie hyperbolique que $\delta_{\mathcal{P}}(\E^{int})=\delta_{\mathcal{P}}(\E^{ext})=\frac{n-1}{2}$ et que les séries de Poincaré de $\mathcal{P}$ divergent en l'exposant critique (on pourra consulter par exemple \cite{MR1776078}, partie 3, m\^eme si les calculs, \'el\'ementaires, remontent \`a Alan Beardon, dans les ann\'ees 1970). D'après le lemme \ref{deltainclu}, on a de même pour $\mathcal{P}$ agissant sur $\O$.
\end{proof}

Nous aurons aussi besoin du lemme suivant:

\begin{lemm}\label{nombrevol}
Soit $C>1$ arbitrairement proche de $1$ et $\mathcal{P}$ un sous-groupe parabolique maximal de $\Aut(\O)$ fixant $p\in\dO$. Alors il existe une horoboule $H_C$ basée en $p$, d'horosph\`ere au bord $\H_C$ et une constante $D>1$ telles que
$$ \frac{1}{D} N_{\mathcal{P}} (x,\frac{R}{C}) \leqslant \Vol_{\O}(B(x,R)\cap H_C) \leqslant DN_{\mathcal{P}} (x,CR),\ x\in\H_C,\ R>0.$$
\end{lemm}

\begin{proof}
Dans l'espace hyperbolique, on sait (voir par exemple la proposition 3.3 de \cite{MR2494912}) que, pour tout sous-groupe parabolique maximal $\mathcal{P}$, toute horoboule $H$ stable par $\mathcal{P}$, d'horosph\`ere au bord $\H$, il existe un réel $D\geqslant 1$ tel que
\begin{equation}\label{eqhyp}
 \frac{1}{D} \Vol_{\HH^n}(B(x,R)\cap H_C) \leqslant N_{\mathcal{P}} (x,R) \leqslant D \Vol_{\HH^n}(B(x,R)\cap H),\ x\in\H.
\end{equation}

Soit donc $\mathcal{P}$ un sous-groupe parabolique maximal de $\Aut(\O)$ fixant $p\in\dO$. Le corollaire \ref{loincusp} nous fournit une horoboule $H_C$ basée en $p$ qui porte deux métriques hyperboliques $\mathcal{P}$-invariantes $\mathtt{h}$ and $\mathtt{h}'$ telles que
$$\frac{1}{C} \mathtt{h}' \leqslant  \mathtt{h} \leqslant F \leqslant \mathtt{h}' \leqslant C \mathtt{h}.$$

Prenons $x\in\H$. D'apr\`es la proposition \ref{compa}, on a, pour tout $R>0$,
$$ B_{\mathtt{h}'}\left(x,\frac{R}{C}\right) \subset   B_{\mathtt{h}}(x,R) \subset   B(x,R) \subset  B_{\mathtt{h}'}(x,R) \subset   B_{\mathtt{h}}(x,CR),$$
où, par $B_{\mathtt{h}}$ et $B_{\mathtt{h}'}$, on note les boules métriques pour $\mathtt{h}$ et $\mathtt{h}'$. En appelant $\Vol_{\mathtt{h}}$ et $\Vol_{\mathtt{h}'}$ les volumes riemanniens associés à $\mathtt{h}$ et $\mathtt{h}'$, on a, toujours d'apr\`es la proposition \ref{compa},
$$\Vol_{\mathtt{h}'} \leqslant \Vol_{\O} \leqslant \Vol_{\mathtt{h}}.$$
Ainsi,
$$\Vol_{\mathtt{h}'}(B_{\mathtt{h}'}\left(x,\frac{R}{C}\right)\cap H) \leqslant \Vol_{\O}(B(x,R)\cap H) \leqslant \Vol_{\mathtt{h}}(B_{\mathtt{h}}(x,CR)\cap H).$$
D'après l'encadrement (\ref{eqhyp}), il existe une constante $D>1$ telle que
$$ \frac{1}{D} N^{\mathtt{h}'}_{\mathcal{P}} \left(x,\frac{R}{C}\right)  \leqslant \Vol_{\O}(B(x,R)\cap H) \leqslant D N^{\mathtt{h}}_{\mathcal{P}} (x,CR) ,$$
où $N^{\mathtt{h}}_{\mathcal{P}} (x,R)$ est le nombre de points de l'orbite $\mathcal{P}.x$ dans la boule de rayon $R$ pour $\mathtt{h}$; de même pour $\mathtt{h}'$.\\
(Bien entendu, les horoboules considérées dans l'encadrement (\ref{eqhyp}) sont les horoboules hyperboliques et pas celles de $F$, et il faut donc faire un peu plus attention lorsqu'on dit qu'une telle constante $D$ existe. Si $\H_{\mathtt{h}}$ est l'horosphère pour $\mathtt{h}$ basée en $p$ et passant par $x$, alors la $\mathtt{h}$-distance maximale entre $\H$ et $\H_{\mathtt{h}}$ est finie, car $\mathcal{P}$ agit de façon cocompacte sur $\H\smallsetminus\{p\}$ et $\H_{\mathtt{h}}\smallsetminus\{p\}$. Donc, pour une certaine constante $D'>0$, on a, pour tout $R>0$,
$$ |\Vol_{\mathtt{h}}(B_{\mathtt{h}}(x,R)\cap H) -  \Vol_{\mathtt{h}}(B_{\mathtt{h}}(x,R)\cap H_{\mathtt{h}})| \leqslant D'N_{\mathcal{P}} (x,R),$$
où $H_{\mathtt{h}}$ est l'horoboule définie par $\H_{\mathtt{h}}$. D'où l'existence de $D$.)\\

Pour conclure, il suffit de remarquer que, comme $\mathtt{h} \leqslant F \leqslant \mathtt{h}' $, on a
$$N^{\mathtt{h}}_{\mathcal{P}} (x,R) \leqslant N_{\mathcal{P}} (x,R) \leqslant N^{\mathtt{h}'}_{\mathcal{P}} (x,C).$$
\end{proof}

On peut maintenant donner une

\begin{proof}[Démonstration du théorème \ref{volegaltop}]
On sait déjà que $\delta_{\Gamma}\leqslant h_{vol}$, et il faut donc seulement prouver l'in\'egalit\'e inverse.\\

Fixons $C>1$ arbitrairement proche de $1$, et choisissons $o\in\O$ ainsi qu'un domaine fondamental convexe localement fini pour l'action de $\G$ sur $\O$, qui contienne le point $o$. Décomposons ce domaine fondamental en
$$C_0\bigsqcup \sqcup_{i=1}^l C_i,$$
où $C_0$ est compact et les $C_i,\ 1\leqslant i\leqslant l$ correspondent aux cusps $\xi_i\in\dO$: chaque $C_i$ est une partie d'un domaine fondamental pour l'action d'un sous-groupe parabolique maximal $\mathcal{P}_i$ sur une horoboule $H_{\xi_i}$ basée au point $\xi_i$; les points $\xi_i$ sont les points de $\dO$ adhérents au domaine fondamental. On suppose que les $H_{\xi_i}$ sont choisies de telle façon qu'elles satisfassent au lemme \ref{nombrevol} pour la constante $C$ qu'on a fixée.\\

La boule $B(o,R)$ de rayon $R\geqslant 0$ peut être décomposée en
$$B(o,R) = (\Gamma.C_0\cap B(o,R)) \sqcup \left(\displaystyle\sqcup_{i=1}^l  \Gamma.H_{\xi_i} \cap B(o,R) \right),$$
de telle façon que
$$\Vol_{\O}(B(o,R))=\Vol_{\O}(\Gamma.C_0\cap B(o,R)) + \sum_{i=1}^l \Vol_{\O}(\Gamma.H_{\xi_i} \cap B(o,R)).$$

Pour le premier terme, on a $\Vol_{\O}(\Gamma.C_0\cap B(o,R))\leqslant N_{\Gamma}(o,R) \Vol_{\O}(C_0)$; c'est donc le second qu'il nous faut étudier.\\

Pour chaque horoboule $H_{\g \xi_i}=\g H_{\xi_i}$, appelons $x_{\g,i}$ le point d'intersection de la droite $(o\g\xi_i)$ avec l'horosphère $\partial H_{\g \xi_i}\smallsetminus \{\g \xi_i\}$, qui n'est rien d'autre que la projection de $o$ sur $H_{\g \xi_i}$. Pour chaque $\g\in\Gamma$, notons $\overline{\gamma}\in\Gamma$ un des éléments $g\in\G$, en nombre fini, tels que
$x_{\g,i}\in g.C_i$; $\overline{\g}$ est le ``premier élément pour lequel $H_{\g \xi_i}$ intersecte $B(o,R)$''. Appelons $\overline{\Gamma}$ l'ensemble de ces éléments $\overline{\g}$.

La remarque principale est le lemme ci-dessous, \'equivalent du fait classique suivant en courbure négative pincée: pour chaque $\theta \in (0,\pi)$, on peut trouver une constante $C(\theta)$ telle que, pour chaque triangle géodésique $xyz$ dont l'angle au point $y$ est au moins $\theta$, le chemin $x\to y \to z$ sur le triangle est une quasi-géodésique entre $x$ et $z$ avec une erreur au plus $C(\theta)$.

\begin{lemm}\label{quasigeo}
Il existe $r>0$ tel que, pour chaque $\g\in\G,\ 1\leqslant i\leqslant l$ et $z\in H_{\g \xi_i}$, le chemin formé des segments $[ox_{\g,i}]$ et $[x_{\g,i}z]$ est une quasi-géodésique avec une erreur au plus $r$, c'est-à-dire que
$$d_{\O}(o,z) \geqslant d_{\O}(o,x_{\g,i})+d_{\O}(x_{\g,i},z) -r.$$
\end{lemm}


\begin{proof}
Prenons $\g\in\G,\ 1\leqslant i\leqslant l$ et $z\in H_{\g \xi_i}$. Rappelons que l'espace m\'etrique $(\O,d_{\O})$ est Gromov-hyperbolique (voir \cite{Crampon:2012fk}, section 9). Aussi existe t-il un réel $\delta\geqslant 0$ pour lequel chaque triangle géodésique est $\delta$-fin. Ainsi, il existe $p\in [oz]$ tel que
$$d_{\O}(p,[x_{\g,i}z])\leqslant \delta,\ d_{\O}(p,[ox_{\g,i}])\leqslant \delta.$$
On peut donc trouver des points $o'\in [ox_{\g,i}]$ et $z'\in [x_{\g,i}z]$ de telle façon que
$$d_{\O}(o',p) + d_{\O}(p,z') \leqslant 2\delta.$$

\begin{center}
\begin{figure}[h!]
  \centering
\includegraphics[width=7cm]{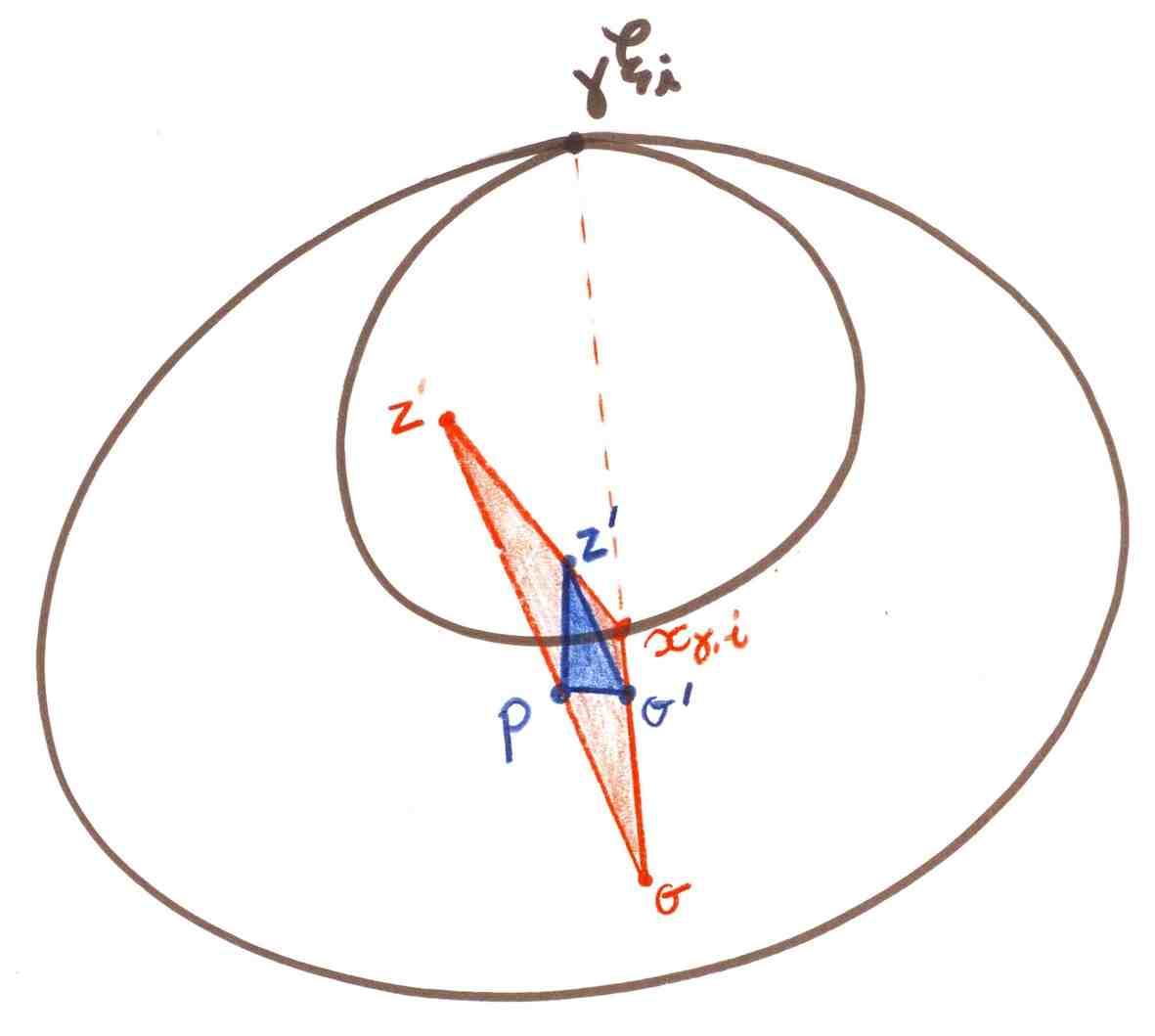}
\caption{}
\label{ouille}
\end{figure}
\end{center}

Par l'inégalité triangulaire, la distance entre $o'$ et $z'$ est alors plus petite que $2\delta$. Puisque $x_{\g,i}$ est la projection de $o'$ sur l'horoboule $H_{\g \xi_i}$ et que $z'$ est dans l'horoboule $H_{\g \xi_i}$, on a $\d(o',x_{\g,i}) \leqslant \d(o',z') \leqslant 2\delta$ et, par l'in\'egalit\'e triangulaire, $d_{\O}(x_{\g,i},z')\leqslant 4\delta$. Ainsi, on obtient
$$d_{\O}(o',x_{\g,i})+d_{\O}(x_{\g,i},z') \leqslant 6\delta.$$
Cela amène
$$\begin{array}{rl}
d_{\O}(o,x_{\g,i})+d_{\O}(x_{\g,i},z) &\leqslant d_{\O}(o,o')+d_{\O}(o',x_{\g,i})+d_{\O}(x_{\g,i},z')+d_{\O}(z',z)\\\\
&\leqslant 6\delta + d_{\O}(o,p)+d_{\O}(p,o')+d_{\O}(z',p)+d_{\O}(p,z)\\\\
&\leqslant 8\delta + d_{\O}(o,z).
\end{array}$$
\end{proof}

Maintenant, si $z$ est un point dans $\g.H_{\xi_i} \cap B(o,R)$, pour certains $\g\in\G$, $1\leqslant i\leqslant l$ et $R>0$, le lemme \ref{quasigeo} implique que
$$ d_{\O}(o,x_{\g,i}) + d_{\O}(x_{\g,i},z) \leqslant \d(o,z) + r \leqslant R+r.$$
Or, il existe $c\geqslant 0$ tel que  $\d(o,x_{\g,i}) \geqslant \d(o,\overline{\g}o)-c$: il suffit de prendre pour $c$ la distance maximale entre $o$ et le bord $\partial C_i \cap \partial H_{\xi_i} \smallsetminus \{\xi_i\}$. D'où
$$ d_{\O}(x_{\g,i},z)  \leqslant R+r - \d(o,\overline{\g}o) + c.$$

Posons $K=r+c$. Pour tous $\g\in\Gamma$, $1\leqslant i\leqslant l$, et $R>0$, on a ainsi
$$\g.H_{\xi_i} \cap B(o,R) \subset \g.H_{\xi_i} \cap B(x_{\g,i}, R-d_{\O}(o,\overline{\g} o)+K).$$

Cela permet d'évaluer le volume $\Vol_{\O}(\Gamma.H_{\xi_i} \cap B(o,R))$. En effet,
$$\begin{array}{rl} \Vol_{\O}(\Gamma.H_{\xi_i} \cap B(o,R))& =\displaystyle\sum_{\overline{\g}\in\overline{\Gamma}} \Vol_{\O}(\overline{\g}.H_{\xi_i} \cap B(o,R))\\\\
 & \leqslant
\displaystyle\sum_{\overline{\g}\in\overline{\Gamma}} \Vol_{\O}(\overline{\g}.H_{\xi_i} \cap B(x_{\g,i}, R-d(o,\overline{\g} o)+K))\\\\

& \leqslant \displaystyle \sum_{0\leqslant k\leqslant [R]} \displaystyle\sum_{\begin{array}{c}\overline{\g}\in\overline{\Gamma}\\ k \leqslant\d(o,\overline{\g}o) \leqslant k+1 \end{array}}
\Vol_{\O}(\overline{\g}.H_{\xi_i} \cap B(x_{\g,i}, R-k+K))\\\\

& \leqslant  \displaystyle \sum_{0\leqslant k\leqslant [R]} \displaystyle N_{\overline{\Gamma}}(o,k,k+1) \Vol_{\O}(H_{\xi_i} \cap B(x_{i}, R-k+K)),
\end{array}$$
où $x_i=x_{Id,i}$ et, pour toute partie $S$ de $\G$ et tout $0\leqslant r < R$,
$$N_S(o,r,R) = \sharp \{\g\in S,\ r\leqslant \d(o,\g o)< R\}.$$

Le lemme \ref{nombrevol} donne
\begin{equation}\label{p1}
  \Vol_{\O}(\Gamma.H_{\xi_i} \cap B(o,R)) \leqslant D \sum_{0\leqslant k\leqslant [R]} \displaystyle N_{\overline{\Gamma}}(x_i,k,k+1) N_{\mathcal{P}_i}(x_i, C(R-k+K))
\end{equation}

pour une certaine constante $D>1$ qui peut être choisie indépendante de $i$. De plus, comme l'exposant critique de chaque $\mathcal{P}_i$ est $\frac{n-1}{2}$, il existe un réel $M\geqslant 1$, indépendant de $i$ mais dépendant de $C$, tel que
$$\frac{1}{M} e^{(\frac{n-1}{2}-(C-1)) R} \leqslant N_{\mathcal{P}_i}(x_i, R)) \leqslant M e^{(\frac{n-1}{2}+(C-1)) R}.$$

D'un c\^ot\'e, cela implique que 
$$N_{\mathcal{P}_i}(x_i, C(R-k+K))\leqslant L N_{\mathcal{P}_i}(x_i, C(R-k)),$$
o\`u $L=M^2 e^{(\frac{n-1}{2}-(C-1)) CK}$. D'un autre c\^ot\'e, cela nous donne
$$\begin{array}{rl}
N_{\mathcal{P}_i}(x_i, CR) & \leqslant  M e^{(\frac{n-1}{2}+(C-1)) CR}\\\\
& = M e^{(\frac{n-1}{2}-(C-1)) R} e^{(\frac{n-1}{2}+C+1)(C-1) R} \\\\
& \leqslant   M^2 e^{(\frac{n-1}{2}+C+1)(C-1) R}  N_{\mathcal{P}_i}(x_i, R).
\end{array}$$
Avec (\ref{p1}), on obtient
\begin{equation}\label{p11}
  \Vol_{\O}(\Gamma.H_{\xi_i} \cap B(o,R)) \leqslant D L M^2 e^{(\frac{n-1}{2}+C+1)(C-1) R}  \sum_{0\leqslant k\leqslant [R]} \displaystyle N_{\overline{\Gamma}}(x_i,k,k+1) N_{\mathcal{P}_i}(x_i, R-k).
\end{equation}

Finalement, remarquons que tout élément $\g\in\G$ tel que $\d(x_i,\g x_i)< R$ peut être écrit de façon unique $\g=\overline{\g_i} p_i$, avec $\d(x_i,\overline{\g_i} x_i)< R$ et $p_i\in\mathcal{P}_i$, de telle façon que
$$d_{\O}(x_i, p_i x_i) + \d(x_i,\overline{\g_i} x_i) \geqslant R.$$
D'où
\begin{equation}\label{p2}
N_{\G}(x_i,R) \geqslant \sum_{0\leqslant k\leqslant [R]} N_{\overline{\G}}(x_i,k,k+1) N_{\mathcal{P}_i}(x_i,R-k).
\end{equation}

Les in\'egalit\'es (\ref{p11}) et (\ref{p2}) impliquent alors
$$\Vol_{\O}(\Gamma.H_{\xi_i} \cap B(o,R)) \leqslant  D L M^2 e^{(\frac{n-1}{2}+C+1)(C-1) R}  N_{\G}(x_i,R);$$
et donc, en mettant tout ensemble
$$\Vol_{\O}(B(o,R)) \leqslant  N e^{(\frac{n-1}{2}+C+1)(C-1) R} N_{\Gamma}(o,R),$$
pour un certain réel $N>1$. Cela donne

$$h_{vol}\leqslant\delta_{\Gamma} + \left(\frac{n-1}{2}+C+1\right)(C-1).$$
Comme $C$ est arbitrairement proche de $1$, on obtient
$$h_{vol} \leqslant \dgg.$$
\end{proof}

\subsection{Groupes dont l'action est g\'eom\'etriquement finie sur $\O$}

En fait, par la même démonstration et les r\'esultats de \cite{Crampon:2012fk}, on peut obtenir un r\'esultat similaire pour des groupes dont l'action est g\'eom\'etriquement finie sur $\O$:

\begin{theo}\label{volegaltopgeneral}
Soit $\G$ un sous-groupe discret de $\Aut(\O)$ dont l'action sur $\O$ est géométriquement finie. Alors
$$\delta_{\G} = \limsup_{R\to +\infty} \frac{1}{R} \ln \Vol_{\O}(B(o,R)\cap C(\LG)),$$
o\`u $o$ est un point quelconque de $\O$.
\end{theo}

Nous avons pr\'ef\'er\'e pr\'esenter la démonstration dans le cas du volume fini que nous consid\'erions d\'ej\`a assez technique pour ne pas la surcharger. Les seuls points \`a v\'erifier pour \'etendre le r\'esultat sont les trois lemmes \ref{par}, \ref{nombrevol} et \ref{quasigeo}; le reste se lit tel quel en pensant seulement \`a consid\'erer l'intersection avec $C(\LG)$.\\
Pour le lemme \ref{quasigeo}, il suffit de se souvenir que, d'apr\`es \cite{Crampon:2012fk}, l'espace $(C(\LG),\d)$ est Gromov-hyperbolique.\\
Pour les deux autres, il nous faut rappeler quelques \'el\'ements de \cite{Crampon:2012fk}, dont on conseille de consulter la partie 7.\\
Tous les groupes paraboliques apparaissant dans une action g\'eom\'etriquement finie d'un groupe $\G$ sur $\O$ sont conjugu\'es dans $\ss$ \`a des sous-groupes paraboliques de $\SO$. En particulier, un sous-groupe parabolique $\P$ de $\G$ est virtuellement isomorphe \`a $\Z^d$ pour un certain $1 \leqslant d\leqslant n-1$; $d$ est le rang de $\P$. De plus, si $p\in\dO$ est le point fixe de $\P$, il existe une coupe $\O_{p}$ de dimension $d+1$ de $\O$, contenant $p$ dans son adh\'erence, c'est-\`a-dire l'intersection de $\O$ avec un sous-espace projectif de dimension $d+1$, qui est pr\'eserv\'ee par $\P$; ainsi, $\P$ appara\^it comme un sous-groupe parabolique de rang maximal de $\Aut(\O_p)$. On obtient la g\'en\'eralisation suivante du lemme \ref{par}:

\begin{lemm}
L'exposant critique d'un sous-groupe parabolique $\mathcal{P}$ de $\Aut(\O)$, de rang $d\leqslant n-1$, est $\delta_{\mathcal{P}}=\frac{d}{2}$ et les séries de Poincaré de $\mathcal{P}$ divergent en $\delta_{\mathcal{P}}$:
$$\forall x\in\O,\ \displaystyle\sum_{\g\in\P} e^{-\delta_{\P} \d(x,\g x)}=+\infty.$$
\end{lemm}
\begin{proof}
Comme rien ne d\'epend pas du point base consid\'er\'e, on peut le prendre dans le convexe $\O_p$ pour se ramener au cas original du lemme \ref{par}.
\end{proof}

Le lemme \ref{nombrevol} s'\'etendrait imm\'ediatement sous les conclusions du lemme \ref{loincusp}. En g\'en\'eral, on peut \'etendre au moins la majoration, qui est le point que l'on utilise dans la démonstration du th\'eor\`eme:

\begin{lemm}\label{nbvolnew}
Soit $C>1$ arbitrairement proche de $1$ et $\mathcal{P}$ un sous-groupe parabolique maximal de $\Aut(\O)$ fixant $p\in\dO$. Il existe une horoboule $H_C$ basée en $p$ d'horosph\`ere au bord $\H_C$ et une constante $D>1$ telle que
$$\Vol_{\O}(B(x,R)\cap H\cap C(\LG)) \leqslant DN_{\mathcal{P}} (x,CR),\ x\in\H_C\cap C(\LG),\ R\geqslant 1.$$
\end{lemm}

On aura besoin du r\'esultat \ref{compare} ci-dessous. Il se d\'eduit du fait suivant:

\begin{lemm}[Bruno Colbois-Constantin Vernicos \cite{MR2245997}]\label{faitvol}
Pour tout $m\geq 1$ et $R>0$, il existe deux constantes $v_m(R), V_m(R)>0$ telles que, pour tout ouvert proprement convexe $\O$ de $\PP^m$ et $x\in\O$,
$$v_m(R) \leqslant Vol_{\O}(B(x,R)) \leqslant V_m(R).$$
\end{lemm}

\begin{rema}
Le lemme précédent est contenu dans le théorème 12 de \cite{MR2245997}, qui donne en plus des bornes explicites, dont la dépendance en $R$ est exponentielle. Pour l'\'enonc\'e pr\'esent\'e ici, on peut donner une démonstration qualitative, basé sur le théorème de compacité de Benzécri \cite{MR0124005}; ce th\'eor\`eme affirme que l'action de $\ss$ sur l'ensemble des couples $(\O,x)$, où $\O$ est un ouvert proprement convexe de $\PP^n$ et $x$ un point de $\O$, est propre et cocompacte. On peut trouver cette d\'emonstration dans \cite{Crampon:2012fk}.
\end{rema}

\begin{lemm}\label{compare}
Soient $r>0$ et $1\leqslant d \leqslant n$. Il existe deux constantes $M,m > 0$, d\'ependant seulement de $r$, $n$ et $d$, telles que, pour tout ouvert proprement convexe $\O$ de $\PP^n$, tout sous-espace $\PP^d$ de dimension $d$ intersectant $\O$ et toute partie $A$ compacte de l'ouvert proprement convexe $\O_d=\PP^d\cap\O$ de $\PP^d$, le volume du $r$-voisinage $V_r(A)$ de $A$ dans $\O$ est comparable au volume du $r$-voisinage de $A$ dans $\O_d$:
$$m \leqslant \frac{\Vol_{\O}(V_r(A))}{\Vol_{\O_d}(V_r(A)\cap\O_d)}  \leqslant M.$$
\end{lemm}

\begin{proof} Consid\'erons un ensemble $\{x_i\}_{1\leqslant i\leqslant N}$ $2r$-s\'epar\'e maximal de $A$: deux points $x_i, x_j$, $i\not=j$, sont \`a distance au moins $2r$ et il est impossible de rajouter un point \`a l'ensemble qui satisfasse \`a cette propri\'et\'e. En particulier, les boules de de rayon $r$ centr\'ees aux points $x_i$ sont disjointes et incluses dans $V_r(A)$, alors que $V_r(A)$ est recouvert par les boules de rayon $4r$ centr\'ees aux points $x_i$. On a ainsi, en utilisant le lemme \ref{faitvol},
$$ N v_n(r) \leqslant \sum_{i=1}^N Vol_{\O}(B(x_i,r))\leqslant \Vol_{\O}(V_r(A)) \leqslant \sum_{i=1}^N Vol_{\O}(B(x_i,4r)) \leqslant N V_n(4r);$$
de m\^eme,
$$N v_d(r)\leqslant \Vol_{\O_d}(V_r(A)\cap\O_d) \leqslant N V_d(4r).$$
En prenant le quotient, on obtient
$$ m:=\frac{v_n(r)}{V_d(4r)}\leqslant\frac{\Vol_{\O}(V_r(A))}{\Vol_{\O_d}(V_r(A)\cap\O_d)} \leqslant \frac{V_n(4r)}{v_d(r)}=:M.$$
\end{proof}

\begin{proof}[D\'emonstration du lemme \ref{nbvolnew}]
Notons $d$ le rang de $\P$ et $\O_p = \O \cap \PP^{d+1}$ une coupe de $\O$ de dimension $d+1$, contenant $p$ dans son ad\'erence, et pr\'eserv\'ee par $\P$. Pour toute horoboule $H$ bas\'ee en $p$, l'intersection $C(\LG)\cap H$ est dans un $\d$-voisinage de taille $r=r(H)\geqslant 0$ finie de $\O_p\cap C(\LG)$. Cela est simplement d\^u au fait que $\P$ agit de fa\c con cocompacte sur $C(\LG)\cap \H$, o\`u $\H$ est l'horosph\`ere au bord de $H$. De plus, comme le bord de $\O$ est $\C^1$ en $p$, on peut, en consid\'erant une horoboule plus petite, prendre $r$ aussi petit que l'on veut.\\

On fixe l'horoboule $H$ de telle fa\c con que l'intersection $H\cap\O_p$, qui est une horoboule de $\O_p$ bas\'ee en $p$, satisfasse au lemme \ref{nombrevol} pour $\O_p$, avec la constante $C$. On fixe aussi $r>0$ tel que $C(\LG)\cap H$ soit dans un $\d$-voisinage de taille $r$ de $\O_p\cap C(\LG)$.\\

Si $x$ est un point de $\H \cap C(\LG)$, il existe un point $x'$ de $\H\cap\O_p$ \`a distance moins que $r$ de $x$; on a alors $B(x,R)\subset B(x',R+r)$ et
$$\Vol_{\O}(B(x,R)\cap H\cap C(\LG))\leqslant \Vol_{\O}(B(x',R+r)\cap H\cap C(\LG)).$$
Maintenant, si $x' \in\H\cap\O_p$, l'ensemble $B(x',R+r)\cap H\cap C(\LG)$ est inclus dans le $r$-voisinage de $B_{\O_p}(x',R)\cap H$ dans $\O$. On a donc, d'apr\`es le lemme \ref{compare},
\begin{equation}
\Vol_{\O}(B(x',R+r)\cap H\cap C(\LG)) \leqslant M \Vol_{\O_p}(A_r).
\end{equation}
o\`u $A_r$ est le $r$-voisinage de $B_{\O_p}(x',R)\cap H$ dans $\O_p$.\\

La partie de $A_r$ qui est dans $H$ correspond pr\'ecis\'ement \`a l'intersection $B_{\O_p}(x',R+r)\cap H$ \`a laquelle on peut appliquer le lemme \ref{nombrevol}, qui donne:
$$\Vol_{\O_p}(B_{\O_p}(x',R+r)\cap H) \leqslant D N_{\P}(x,C(R+r)).$$

Le reste de $A_r$ est dans un voisinage de taille $r'=\max\{r,diam\}$ de l'ensemble fini de points $\P.x\cap B(x,R+r)$, o\`u $diam$ est le diam\`etre d'un domaine fondamental compact pour l'action de $\P$ sur $\H\cap\O_p\cap C(\LG)$ (tout cela pour la distance $d_{\O_p}$). Le volume de cette partie est donc major\'e par
$$N_{\P}(x',R+r) V_d(r'),$$
o\`u $V_d(r')$ est la constante donn\'ee par le lemme \ref{faitvol}. Au final, on obtient
$$\Vol_{\O_p}(A_r) \leqslant N_{\P}(x',R+r) V_d(r')+D N_{\P}(x',C(R+r)) \leqslant D' N_{\P}(x',C(R+r)),$$
pour une certaine constante $D'$.\\

En regroupant le tout, on arrive \`a 
$$\Vol_{\O}(B(x,R)\cap H\cap C(\LG))\leqslant D' N_{\P}(x',C(R+r)) \leqslant D' N_{\P}(x,C(R+r)+r).$$
Cela donne le r\'esultat, la condition $R\geqslant 1$ \'etant due \`a la pr\'esence du $r$.
\end{proof}

Le th\'eor\`eme \ref{volegaltopgeneral} soul\`eve la question suivante, sur laquelle nous terminerons ce texte. C'est une question qu'on peut poser de fa\c con tr\`es g\'en\'erale mais une r\'eponse dans des cas particuliers serait d\'ej\`a int\'eressante.

\begin{qu}
Soit $\O$ un ouvert proprement convexe de $\PP^n$ (strictement convexe et \`a bord $\C^1$). A t-on $\dgg = h_{vol}(C(\LG),d_{C(\LG)})$ pour tout sous-groupe discret $\G$ de $\Aut(\O)$ ?
\end{qu}

\bibliographystyle{alphaurl}

\end{document}